\documentclass[12pt]{amsart}
\usepackage{svn}
\SVN $Revision: 397 $

\usepackage{amssymb,amsthm,amsmath,amsfonts}
\usepackage{graphicx}
\usepackage{subcaption}
\usepackage{epstopdf}
\usepackage{mathtools}
\usepackage{mathabx}
\usepackage{blkarray}

\usepackage{amsaddr}

\usepackage{textcomp}

\usepackage{enumerate}

\usepackage{commath}
\usepackage{color}

\usepackage[colorlinks=true,urlcolor=black,anchorcolor=black,citecolor=black,linkcolor=black]{hyperref}
\usepackage[hypcap]{caption}
\usepackage[a4paper,margin=2cm]{geometry}
\usepackage{tikz,pgfplots,pgfplotstable}
\pgfplotsset{compat=1.16}

\usetikzlibrary{arrows.meta}
\definecolor{histcolor}{rgb}{0.97858, 0.678934, 0.157834}

\allowdisplaybreaks{}

\DeclareMathOperator{\id}{\mathrm{id}}

\DeclareMathOperator{\IE}{\mathbb{E}}

\newcommand{\itensor}{\oocirc} \newcommand{\llangle}{\langle\!\langle} 
\newcommand{\rrangle}{\rangle\!\rangle}

\newcommand{\freeprod}{\mathbin{\amalg{}}} 
\newcommand{\FP}{\freeprod}

\renewcommand{\star}{\mathbin{\ast{}}}

\DeclareMathOperator{\diag}{\mathrm{diag}}
\DeclareMathOperator{\Ord}{\mathcal{O}}

\DeclareMathOperator{\tensT}{\mathcal{T}}
\DeclareMathOperator{\tensTb}{\overline{\mathcal{T}}}

 \DeclareMathOperator{\bdeconc}{\mbox{$\widebar{\Delta}$}_d}
\DeclareMathOperator{\fdeconc}{\tilde{\Delta}_d}

\DeclareMathOperator{\bdeconci}{\mbox{$\widebar{\Delta}$}_d^{\itensor}}

\DeclareMathOperator{\IntPart}{\mathrm{Int}}
\DeclareMathOperator{\Comp}{\mathcal{C}}
\DeclareMathOperator{\descents}{\mathcal{D}}

\DeclareFontFamily{U} {MnSymbolC}{}
\DeclareFontShape{U}{MnSymbolC}{m}{n}{
  <-6> MnSymbolC5
  <6-7> MnSymbolC6
  <7-8> MnSymbolC7
  <8-9> MnSymbolC8
  <9-10> MnSymbolC9
  <10-12> MnSymbolC10
  <12-> MnSymbolC12}{}
\DeclareFontShape{U}{MnSymbolC}{b}{n}{
  <-6> MnSymbolC-Bold5
  <6-7> MnSymbolC-Bold6
  <7-8> MnSymbolC-Bold7
  <8-9> MnSymbolC-Bold8
  <9-10> MnSymbolC-Bold9
  <10-12> MnSymbolC-Bold10
  <12-> MnSymbolC-Bold12}{}

\DeclareSymbolFont{MnSyC} {U} {MnSymbolC}{m}{n}

\DeclareMathSymbol{\oocirc}{\mathbin}{MnSyC}{99}

\newcommand{\bPsi}{\boldsymbol{\Psi{}}}

\def\R{{\mathbb R}}

\newcommand\IC{{\mathbb C}}
\newcommand\IK{{\mathbb K}}
\newcommand\IF{{\mathbb F}}

\newcommand\IN{{\mathbb N}}

\newcommand\uXn{{X_1,X_2,\ldots,X_n}}
\newcommand\uXk{{X_1,X_2,\ldots,X_k}}
\newcommand\bX{{\mathcal{X}}}

\newcommand\bub[1]{\mathring{#1}}

\def\<{\langle}
\def\>{\rangle}

\def\E{\mathbb{E}}

\DeclareMathOperator{\NC}{\mathit{NC}} \DeclareMathOperator{\NCirr}{\NC^\mathit{irr}} 

 \newcommand{\alg}[1]{\mathcal{#1}}
\renewcommand{\norm}[1]{\lVert#1\rVert}
\renewcommand{\abs}[1]{\lvert#1\rvert}

\newcommand{\harpleftsign}{\scriptstyle\leftharpoonup}
\newcommand{\harpleft}[2]{\ifx\displaystyle#1\doalign{$\harpleftsign$}{#1#2}\fi
  \ifx\textstyle#1\doalign{$\harpleftsign$}{#1#2}\fi
  \ifx\scriptstyle#1\doalign{\scalebox{.6}[.9]{$\harpleftsign$}}{#1#2}\fi
  \ifx\scriptscriptstyle#1\doalign{\scalebox{.5}[.8]{$\harpleftsign$}}{#1#2}\fi
}

\newcommand{\harprightsign}{\scriptstyle\rightharpoonup}
\newcommand{\harpright}[2]{\ifx\displaystyle#1\doalign{$\harprightsign$}{#1#2}\fi
  \ifx\textstyle#1\doalign{$\harprightsign$}{#1#2}\fi
  \ifx\scriptstyle#1\doalign{\scalebox{.6}[.9]{$\harprightsign$}}{#1#2}\fi
  \ifx\scriptscriptstyle#1\doalign{\scalebox{.5}[.8]{$\harprightsign$}}{#1#2}\fi
}
\newcommand{\doalign}[2]{{\vbox{\offinterlineskip\ialign{\hfil##\hfil\cr#1\cr$#2$\cr}}}}

\usepackage{yhmath}

\newcommand{\hl}[1]{\mathpalette\harpleft{#1}}
\newcommand{\hr}[1]{\mathpalette\harpright{#1}}

\DeclareMathOperator{\rDelta}{\mathpalette\harpright{\Delta}}
\DeclareMathOperator{\lDelta}{\mathpalette\harpleft{\Delta}}

\makeatletter
 \newcommand*{\closeindex}[1]{_{\mkern-4.5mu#1}}

 \DeclareRobustCommand{\rdelta}{\mathpalette\harpright{\delta}\@ifnextchar_{\expandafter\closeindex\@gobble}{}}
 \DeclareRobustCommand{\ldelta}{\mathpalette\harpleft{\delta}\@ifnextchar_{\expandafter\closeindex\@gobble}{}}

 \DeclareRobustCommand{\rnabla}{\hr{\nabla}\@ifnextchar_{\expandafter\closeindex\@gobble}{}}
 \DeclareRobustCommand{\lnabla}{\hl{\nabla}\@ifnextchar_{\expandafter\closeindex\@gobble}{}}
 \makeatother

\makeatletter
\renewcommand{\tocsection}[3]{\indentlabel{\@ifnotempty{#2}{\bfseries\ignorespaces#1 #2\quad}}\bfseries#3}
\renewcommand{\tocsubsection}[3]{\indentlabel{\@ifnotempty{#2}{\ignorespaces#1 #2\quad}}#3}

\newcommand\@dotsep{4.5}
\def\@tocline#1#2#3#4#5#6#7{\relax
	\ifnum #1>\c@tocdepth \else
	\par \addpenalty\@secpenalty\addvspace{#2}\begingroup \hyphenpenalty\@M
	\@ifempty{#4}{\@tempdima\csname r@tocindent\number#1\endcsname\relax
	}{\@tempdima#4\relax
	}\parindent\z@ \leftskip#3\relax \advance\leftskip\@tempdima\relax
	\rightskip\@pnumwidth plus1em \parfillskip-\@pnumwidth
	#5\leavevmode\hskip-\@tempdima{#6}\nobreak
	\leaders\hbox{$\m@th\mkern \@dotsep mu\hbox{.}\mkern \@dotsep mu$}\hfill
	\nobreak
	\hbox to\@pnumwidth{\@tocpagenum{\ifnum#1=1\bfseries\fi#7}}\par \nobreak
	\endgroup
	\fi}
\AtBeginDocument{\expandafter\renewcommand\csname r@tocindent0\endcsname{0pt}
}
\def\l@subsection{\@tocline{2}{0pt}{2.5pc}{5pc}{}}
\makeatother

\newcommand{\bbeta}[1]{\beta^b_{#1}}
\newcommand{\bbetan}[2]{\beta^{b^{(\scriptstyle #2)}}_{#1}\!}

\newcommand{\fbeta}[1]{\beta^\delta_{#1}}
\newcommand{\fbetan}[2]{\beta^{\delta^{(\scriptstyle #2)}}_{#1}\!}

\newcommand{\be}{\begin{equation}}
\newcommand{\ee}{\end{equation}}
\newtheorem{theorem}{Theorem}[section]
\newtheorem{proposition}[theorem]{Proposition}
\newtheorem{corollary}[theorem]{Corollary}
\newtheorem{lemma}[theorem]{Lemma}
       
\theoremstyle{definition}
\newtheorem{notation}[theorem]{Notation}
\newtheorem{definition}[theorem]{Definition}

\newtheorem{assumption}[theorem]{Assumption}
\newtheorem{algorithm}[theorem]{Algorithm}
\theoremstyle{remark}
\newtheorem{remark}[theorem]{Remark}

\newtheorem{example}[theorem]{Example}

\title{Free Integral Calculus I}
\author[F. Lehner]{Franz Lehner}
\address[F. Lehner]{Institut f\"ur Diskrete Mathematik\\
Technische Universit\"at Graz\\
Steyrergasse 30\\
A-8010 Graz, Austria\\
ORCID: 0000-0002-6902-5148
}
\email{lehner@math.tu-graz.ac.at}

\author[K. Szpojankowski]{Kamil Szpojankowski}
\address[K. Szpojankowski]{Wydzia\l{} Matematyki i Nauk Informacyjnych\\
Politechnika Warszawska\\
ul. Koszykowa 75\\
00-662 Warszawa, Poland\\
ORCID: 0000-0003-0926-1285
}
\email{kamil.szpojankowski@pw.edu.pl}

\thanks{Supported by the Austrian Federal Ministry of Education, Science and
  Research and the Polish Ministry of Science and Higher Education, grants
  N$^{\textrm{os}}$ PL 08/2016 and PL 06/2018.\\
  FL: This research was funded in part by the Austrian Science Fund (FWF) grant
  I 6232-N BOOMER (WEAVE).\\
KSz: This research was funded in part by National Science Centre, Poland WEAVE-UNISONO grant BOOMER 2022/04/Y/ST1/00008.
\\ For the purpose of Open Access, the authors have applied a CC-BY public copyright licence to any Author Accepted Manuscript (AAM) version arising from this submission.}

\numberwithin{equation}{section}

\begin{document}
\date{Rev.\SVNRevision, \today}

\begin{abstract}
  We study the problem of conditional expectations in free random variables
  and provide closed formulas for the conditional expectation of a resolvent
  $\Psi(z) = (1-zP(a_1,\ldots,a_n))^{-1}$  of an arbitrary non-commutative
  polynomial $P(a_1,\ldots,a_n)$ in free random variables
  $a_1,\ldots,a_n$ onto the subalgebra of an arbitrary subset of the
  variables $a_1,\ldots,a_n$.
  More precisely, given a linearization of $\Psi(z)$, our methods allow to compute a
  linearization of its conditional expectation.
  The coefficients of the expressions obtained in this process involve certain  Boolean
  cumulant functionals which can be computed by solving a system of equations.
  On the way towards the main result we introduce a   non-commutative
  differential calculus which allows to evaluate
  conditional expectations and certain Boolean cumulant functionals $\bbeta{}$
  and $\fbeta{}$.
  We conclude the paper with several examples which
  illustrate the working of the developed machinery.
  Two appendices complement the paper. 
  The first appendix contains a purely algebraic approach to Boolean cumulants
  and the second appendix provides a crash course on linearizations of rational
  series. 
\end{abstract}

\keywords{Free Probability, subordination, Boolean cumulants, conditional expectation}
\subjclass[2020]{Primary: 46L54. Secondary:  15B52.}

\maketitle

\setcounter{tocdepth}{2}

\tableofcontents{}

\section{Introduction}

\subsection{Outline}

Free probability was introduced by Voiculescu 40 years ago
\cite{Voiculescu:1985:symmetries}
as a means to solve long-standing problems in von Neumann algebras.
Over the years deep connections to several branches of mathematics came
to light, most notably random matrix theory.
Complementing Voiculescu's analytic approach,
Speicher developed a combinatorial theory of
free cumulants \cite{SpeicherNC}, the latter already being announced in
\cite{Voiculescu:1985:symmetries}.
In this approach, free independence is characterized by vanishing of mixed free
cumulants, in analogy to classical independence, which can by characterized by
vanishing of mixed classical cumulants.
Indeed, most properties of free cumulants can be obtained from the
corresponding properties of classical cumulants by replacing the lattice of all
set partitions by the lattice of non-crossing partitions,
following the general scheme of multiplicative functions on lattices  \cite{DoubiletRotaStanley:1972:foundations6}.
The discovery of free cumulants triggered a lot of progress in Free
Probability,
and it was the starting point of many deep combinatorial studies of various
structures in free probability, see the standard reference \cite{NicaSpeicherLect} for a detailed treatment of free cumulants.

Boolean cumulants were introduced in classical probability as 
``centered moments'' by Statulevi\v{c}ius \cite {statulevicius:1969:limit2}
and in physics as ``partial cumulants''  by von Waldenfels
\cite{Waldenfels:1973:approach}
in order to simplify certain calculations in mathematical physics
The name ``Boolean cumulants'' (together with the  corresponding notion of
\emph{Boolean independence}) was introduced later in\cite{SpeicherWoroudi:1997:boolean} as a special case
of  conditionally free independence
\cite{Bozejko:1986:positive,BozejkoSpeicher:psi:1991,BozLeinSpeich}.
Let us also mention that Boolean cumulants are known as ``self-energy''
in quantum field theory \cite {ItzyksonZuber:QFT:1980}
and have a  natural interpretation as first
return probabilities of random walks \cite{Woess:1986:nearest}.

Combinatorially Boolean cumulants follow the pattern of classical and free cumulants by
replacing the lattice of set partitions (resp.~non-crossing partitions)
with the lattice of interval partitions, which is isomorphic to the Boolean
lattice
and thus from a combinatorial point of view Boolean cumulants are the
simplest kind of cumulants.
 
Recently it was noticed that despite their simplicity
Boolean cumulants are useful for non-commutative probability in general
\cite{JekelLiu:2019:operad} and free probability in particular
\cite{BelinschiNicaBBPforKTuples,FMNS2,LehnerSzpojan,SzpojanWesol}.
Boolean cumulants were used for the first combinatorial solution to the problem
of the free anti-commutator in \cite{FMNS2} (an analytic solution was found
previously by Vasilchuk \cite{Vasilchuk:2003:commutator}; a solution in terms of free
cumulants was presented more recently in \cite{Perales:anticommutator}),
and for the identification of the coefficients of the power series expansion of
subordination functions \cite{LehnerSzpojan}.
From a graph theoretic perspective, first return probabilities are a basic tool
in the study of random walks on free product groups  \cite{Woess:1986:nearest,Woess:2000}.

In the present paper we continue these investigations and show that Boolean
cumulants may indeed be used for a systematic study of free random variables.
The first step in this direction is a surprisingly simple characterization of
freeness by the vanishing of \emph{some} (but not all) mixed Boolean cumulants.

\begin{theorem}[Characterization of freeness in terms of Boolean cumulants]\label{thm:charfree}
  Let $(\alg{M},\varphi)$ be a tracial non-commutative probability space. Subalgebras $\alg{A}$ and $\alg{B}$ are free if and only if  $\beta_{m}(a_1,a_2,\ldots,a_n)=0$ whenever $n>1$ and  $a_j\in\alg{A}\cup\alg{B}$ for $j=1,2,\ldots,n$ with $a_1$ and $a_n$ coming from different subalgebras.
\end{theorem}
For a random walk interpretation of this characterization see Remark~\ref{rem:quenell}.
This property turns out to be the key to the efficient calculation of conditional
expectations in free random variables.
In particular, for free random variables $a_1,a_2,\ldots,a_n$ we determine the
conditional expectation of the resolvent $(1-z P(a_1,a_2,\ldots,a_n))^{-1}$
of an arbitrary non-commutative polynomial $P(a_1,a_2,\dots,a_n)$ onto the subalgebra
generated by some subset of the variables $a_1,a_2,\ldots,a_n$.
To this end we employ the concept of linearizations, a method also used in
\cite{BelinschiMaiSpeicherSubordination,HeltonMaiSpeicher:ApplicationsOfRealizations};
more precisely, given a linearization of the resolvent, we obtain a
linearization of its conditional expectation.
While our discussion is focused on non-commutative rational functions of the form
$(1-z P(a_1,a_2,\ldots,a_n))^{-1}$ where $P$ is a polynomial,  an example at the end of this paper
shows that with some extra effort the methods  are applicable to a wider class
of rational functions.
In subsequent work we will extend the method
to conditionally free random variables \cite{LehnerSzpojankowski:freeint2}.

The first step of our approach is based on a recurrence which
can be summarized as a free integration formula.
In the case of resolvents it naturally leads to a system of linear equations
for the conditional expectation which can be turned into a linearization of the latter.
The coefficients  appearing  in this linearization are certain
generating functions of mixed Boolean cumulants of free random variables.

For the sake of simplicity let us consider the case of non-commutative
polynomials in two free random variables $X$ and $Y$.
We start with a constructive proof of the well known fact
that the conditional expectation of a polynomial in free random variables
is a polynomial again, whose coefficients are obtained by evaluating
a certain functional $\bbeta{Y}$ which depends on the Boolean cumulants
of $X$ and $Y$. The proof is recursive and involves a certain derivation
$\rdelta_{X}$ acting on non-commutative polynomials.
For precise definitions we refer to Section \ref{sec:calcCond}.  
\begin{theorem}[Free integration formula]\label{thm:free_int_formula}
  The conditional expectation of a non-commutative polynomial $P\in\IC\langle
  X,Y\rangle$
  satisfies the identity 
	\begin{align*}
          \E_{X}[P]=\bbeta{Y}(P)+(\bbeta{Y}\otimes\E_{X})(\rdelta_{X}(P)).
	\end{align*}
\end{theorem}

This formula has a certain resemblance to classical integration.
$\E_X$ ``integrates away'' the variable $Y$ and if we denote
$\mathfrak{I}_Y=\bbeta{Y}\otimes \E_X$, then the formula reads
$$
\mathfrak{I}_Y(\rdelta_{X}(P)) = \E_{X}[P] - \bbeta{Y}(P)
= \mathfrak{I}_Y(I\otimes P - P\otimes I).
$$

The formula above allows to calculate conditional expectations in terms of Boolean cumulants. Thus we are faced with the problem to calculate Boolean cumulants of functions
in free random variables.
In order to do this we introduce an algebraic calculus
involving linear functionals $\bbeta{}$ and
$\fbeta{}$ on the free algebra which evaluate Boolean cumulants in two ways,
\emph{blockwise} and \emph{fully factored}.
We then establish mutual recursive equations between these 
with the help of the previously developed algebraic devices
which lead to a closed system of algebraic equations.
Although the functionals $\bbeta{}$ and $\fbeta{}$ are defined in view
of conditional expectations, we expect them to have much wider
applications, and the calculus we introduce for them is of independent
interest. It is based on generalizations of
observations from \cite{FMNS2,LehnerSzpojan} and 
subsumes the combinatorial case-by-case analysis of functions in
free variables into a general algebraic machinery.

On first sight it would seem more natural to use free cumulants for such a
calculus,
since mixed free cumulants of free random variables vanish and they turn additive free
convolution into the simple operation of addition of cumulants.
 Indeed 
first steps towards such a calculus were done in \cite{Cebron:2013:free} and
 we expect  the unshuffle algebras of \cite{EbrahimiFardPatras:2015:halfshuffles} to play a role here. 
However it turned out that when it comes to multiplicative free convolution,
the formulas are actually identical \cite{BelinschiNicaBBPforKTuples} and 
free cumulants offer no advantage over Boolean cumulants.
In fact the combinatorial simplicity
of  Boolean cumulants allows us to almost entirely dispense with the
combinatorial machinery of set partitions.
The main steps of our calculations are as follows:
\begin{enumerate}[(i)]
 \item
  Turn the simple combinatorics of Boolean cumulants of products of free
  variables into an algebraic rule involving certain derivations
  which appeared previously in free probability.
 \item
  Apply these derivations to resolvents and use the fact that these
  play the role of ``eigenfunctions''  analogous to the
  exponential function in classical calculus.
 \item
  Use the formula for Boolean cumulants with free entries in order to separate
  free variables and obtain equations for the generating functions.
\end{enumerate}

Thus we combine combinatorial and algebraic identities into
an effective machinery to make them available for free analysis.
It allows to systematically establish equations for generating
functions of Boolean cumulants of functions in free variables
which in many cases can be effectively solved
and used to compute conditional expectations and distributions
of arbitrary polynomials in free random variables.

In particular we present an algebraic interpretation
of the formula for Boolean cumulants with products as
entries and a characterization of freeness in terms of Boolean
cumulants from \cite{FMNS2} in terms of Voiculescu's free
derivative.
We refer to Theorem~\ref{thm:intro_bbeta_sbeta} for the precise statement and
Definitions~\ref{def:bbeta} and \ref{def:fbeta} for definitions of
the functionals $\bbeta{}$ and $\fbeta{}$.

The paper is organized as follows.

The rest of the introduction is devoted to an exposition of the problem of
conditional expectations.

Section~\ref{sec:Preliminaries} presents results about the basic ingredients: conditional
expectations, Boolean cumulants and derivations.

In Section~\ref{sec:boolchar} we prove the characterization of freeness announced above in Theorem \ref{thm:charfree}.

In Section \ref{sec:calcCond} we provide a method to compute conditional expectations
in terms of Boolean cumulants. In particular we prove Theorem \ref{thm:free_int_formula}.

In Section~\ref{sec:calcbbeta}
 we introduce a calculus for the Boolean functionals $\bbeta{}$ and
$\fbeta{}$ which is summarized in Theorem \ref{thm:intro_bbeta_sbeta}.

In Section~\ref{sec:LinearizationAndCondExp} we show how linearizations allow to solve the
problem for subordinations for general polynomials. In particular we prove Theorem \ref{thm:1.1} which we discuss below.

Section~\ref{sec:examples}  contains explicit examples, featuring among others:
the conditional expectations of commutators and anticommutators (which
coincide, see Section~\ref{ssec:commutator});
a Lie polynomial which is  a sum of elements which are in additive
free position without actually being free (see Section~\ref{ssec:ex:lie});
a symmetric polynomial of order 3 and in particular a random walk of length 3
on the free group (see Section~\ref{ex:XYZ}); and finally, an example
of a noncommutative rational function (Section~\ref{ssec:ex:rational}).

In Appendix~\ref{sec:algbool} we give self-contained algebraic proofs
of the basic results about Boolean cumulants\ as well as a reformulation of these in terms of tensor algebras.

Finally Appendix~\ref{app:linearization} contains a brief survey of regular
non-commutative rational functions and provides the essentials
required for the computation of linearizations. 

\subsection{Subordination for general polynomials}

Let $X$ and $Y$ be classically independent random variables with distributions $\mu_X$ and $\mu_Y$,
then their joint distribution is $\mu_{X,Y}=\mu_X\otimes\mu_Y$, i.e.,
the expectation of any integrable function $f(X,Y)$ can in practice
be computed as a double integral
$$
\IE f(X,Y) = \iint f(x,y)\, d\mu_Y(y)\,d\mu_X(x)
.
$$
In the non-commutative case there is no such integral representation,
however the inner integral is in fact the conditional expectation
\begin{equation}
  \label{eq:intfXy=EfXY}
\int f(X,y) \,d\mu_Y(y) = \IE[f(X,Y)|X]
\end{equation}
and thus
$$
\IE f(X,Y) = 
\IE [ \IE[f(X,Y)| X] ]
$$
which does have a non-commutative analogue.
In the present paper we propose a method to compute
this non-commutative conditional expectation
$$
\IE_X[P(X,Y)] 
$$
for arbitrary non-commutative polynomials $P$ and more general rational
functions in \emph{free} random variables.
It is the analogy with \eqref{eq:intfXy=EfXY} 
which motivated us to call our endeavour \emph{free integral calculus}.

This approach follows ideas of Voiculescu \cite{VoiculescuAnaloguesOfEntropyI} and Biane
\cite{Biane98} who showed that for the sum $a+b$ of two
free random variables $a$, $b$  there exists an analytic self map of
$\omega:\IC^+\to\IC^+$ such that
the conditional  expectation of the resolvent onto the algebra
generated by $a$ is a resolvent again
\begin{equation}
  \label{eq:subordination}
  \IE_a[(z-a-b)^{-1}]=(\omega(z)-a)^{-1}
  ,
\end{equation}
which after application of $\varphi$ yields results in the subordination
relation $G_{a+b}(z)=G_a(\omega(z))$ (already observed implicitly in \cite{Woess:1986:nearest}).
Moreover the function $\omega$ is known to satisfy a fixed point equation. Subordination turned out to be a very effective tool for the study of regularity properties of free convolutions (see \cite{BelBerNewApproach,BelinschiRegularity})
and its operator valued generalization is one of the main  tools in the linearization approach of \cite{BelinschiMaiSpeicherSubordination}.
It is still an active area of research in free probability \cite{BelinschiBercovici:Upgrading}.

In the present paper we generalize this method to arbitrary polynomials using an algebraic-analytic method based on
the observation from our previous work \cite{FMNS2,LehnerSzpojan} that Boolean
cumulants turn out to be a convenient device to ``store'' the
results of the ``partial integral'' described above.

Fix a non-commutative probability
space $(\alg{M},\varphi)$.
Given a non-commutative polynomial $P\in \IC\langle X_1,X_2,\ldots,X_n\rangle$ and self-adjoint random variables $a_1,a_2,\ldots,a_n\in\alg{M}$,
our objective is an explicit formula for the conditional expectation of $
\bigl(1-zP(a_1,a_2,\ldots,a_n)\bigr)^{-1}$
onto the algebra generated by some subset of the variables $a_1,a_2,\ldots,a_n$. Without loss of generality we may assume that the subset consists of $a_1,a_2,\ldots,a_k$, where $k<n$. 

In the first step we construct a \emph{linearization} of the resolvent, i.e., 
matrices $C_1,C_2,\ldots,C_n\in M_N(\IC)$ such that for $L=C_1\otimes a_1+C_2\otimes a_2+\dots +C_n \otimes a_n\in M_N(\IC[z])\otimes \alg{M}$ and $z$ in some neighborhood of zero we have 
\begin{equation}
  \label{eq:intro:linearization}
  (1-z^mP(a_1,a_2,\dots,a_n))^{-1}=u^t (I_N-zL)^{-1}v
\end{equation}
for some vectors $u,v\in\IC^N$, where $m$ is the total degree of $P$.
In general a polynomial may have many linearizations;
in Appendix~\ref{app:linearization} we discuss in detail algorithms for finding
linearizations which work for our purposes.
It suffices to say for the moment that in contrast to
\cite{BelinschiMaiSpeicherSubordination} some technical issues force us to
work with regular linearizations which are not self-adjoint
and  to  restrict the calculations to the level of formal power series.

In the case when  $\alg{M}$ is the von Neumann subalgebra freely generated by $a_1,a_2,\ldots,a_n$,
we prove the following theorem.
It follows by evaluating  the formal expression  from
Theorem~\ref{thm:EIPsimain} below in the variables $a_1,a_2,\dots,a_n$.
\begin{theorem}[Subordination for general polynomials]
  \label{thm:1.1}
  
  Given a linearization   \eqref{eq:intro:linearization} for a polynomial
  $P\in\IC\langle X_1,X_2,\dots,X_n\rangle$ of degree $m$
  we have for $z$ in some neighbourhood of $0$ the following linearization for
  the conditional expectation of the resolvent onto the subalgebra
  $\alg{A}_k$ generated by $a_1,a_2,\dots,a_k$ is
  \begin{multline}
    \label{eq:thm:1.1}    
          \E_{\alg{A}_k}\left[(1-z^mP(a_1,a_2,\ldots,a_n))^{-1}\right]
          \\
          =u^t\left(I_N-z\left(C_1 \otimes a_1 +C_2 \otimes a_2 +\dots+C_k\otimes a_k+ (C_{k+1}
              F_{k+1}+\dotsm+C_n F_n)\otimes 1\right)\right)^{-1}v
	\end{multline}
	where the matrices $F_1,F_2,\ldots,F_n$ constitute the unique
        fixed point of the system of equations
	\begin{equation}
          \label{eq:thm:1.1:eta}              
          F_i(z)=\widetilde{\eta}_{a_i}
            \biggl(z\Bigl(I_N-z \sum_{j\neq i} C_jF_j(z)\Bigr)^{-1} C_i\biggr)
          \qquad\text{ for $i=1,2,\ldots,n$}
	\end{equation}
        with entries which are analytic at $0$.
        Here by  $\widetilde{\eta}_a(z)=\sum_{n=1}^{\infty} \beta_{n}(a)z^{n-1}$
        we denote the shifted Boolean cumulant generating function of a random
        variable $a$.

        In particular, the moment generating function is
        $$
        M(z^m) = u^t
        \bigl(
        I_N - z\sum_{j=1}^n C_jF_j(z)
        \bigr)^{-1}v
        .
        $$
\end{theorem}
\begin{remark}
  \begin{enumerate}[(i)] \item []
   \item The  shifted Boolean transform is essentially the same as the
    reciprocal Cauchy transform
    $$
    \tilde{\eta}_\mu(1/z)
    =z-1/G_\mu(z)
    $$
    which replaces the Cauchy transform in the subordination
    approach to free convolution.

   \item From a practical point of view Theorem~\ref{thm:1.1}
    asserts that the evaluation of the conditional expectation of the resolvent
        \[
          \left(
            1-z^mP(a_1,a_2,\ldots,a_n)
          \right)^{-1}=u^t \left(I_N - z\left(C_1\otimes a_1+C_2\otimes a_2+\dots+ C_n \otimes a_n\right)\right)^{-1}v
        \]
        onto $\alg{A}$ (i.e., ``integrating out'' $a_{k+1},\ldots,a_n$)
        amounts to  replacing the corresponding summands $I_N\otimes a_i$ with
        the matrices $F_i\otimes 1$.
       \item Although only the matrices  $F_{k+1},F_{k+2},\ldots,F_n$ explicitly appear in
        the final formula     \eqref{eq:thm:1.1},
        the matrices $F_1,F_2,\ldots,F_k$ are required as well in order to extract
        the former from the solution of equation \eqref{eq:thm:1.1:eta}.
        Moreover observe that in general the equations cannot be effectively decoupled
        (unless $n=2$, see example~\ref{ex:1:additiveconvolution} below)
        and each matrix from $F_1,F_2,\ldots,F_n$ depends on the distributions of all variables $a_1,a_2,\ldots,a_n$.
  \end{enumerate}
\end{remark}

\begin{remark}[Subordination]
  The previous theorem generalizes the subordination phenomenon in the
  following sense.
	For simplicity consider the case of two free variables
        $a,b$ and fix a polynomial $P\in \IC \langle X,Y\rangle$ of
        degree $m$. We fix an $N\times N$ linearization
        \[
          (1-z^mP(X,Y))^{-1}
          =u^t(I_N-z C_1X-z C_2Y)^{-1}
        \]
        then the above theorem asserts that the moment
        generating function $M_P(z)=\varphi\left((1-z P(a,b))^{-1}\right)$ is
        obtained
        from the linearization via
	\[
          M_{P(a,b)}(z^m)=u^t\varphi^{(N)} (I_N-z C_1 a-z C_2 b)^{-1} v
        \]
	where $\varphi^{(N)}$ is the entry-wise application of $\varphi$.
	Let $H_1=(I_N-zC_1 F_1)^{-1}$, then the preceding identity can be written as
	\[M_{P(a,b)}(z^m)=u^t\varphi^{(N)} (I_N-z H_1 C_2 b)^{-1} H_1 v\]
	Suppose that the matrix $H_1C_2$ is diagonalizable and write $zH_1C_2=Q D Q^{-1}$ with $D=\diag(\lambda_1,\ldots,\lambda_n)$. Then we obtain
	\[M_{P(a,b)}(z^m)=(Pu)^t\varphi (I_N-D b)^{-1} P^{-1}H_1
          v=\widehat{u}^t
          \begin{bmatrix}
            M_{b}(\lambda_1) & 0 & \ldots & 0 \\
            0&  M_{b}(\lambda_2)   & \ldots &0 \\
            \vdots & & \ddots   & \vdots \\
		0   & 0 & \ldots &M_{b}(\lambda_n)
	\end{bmatrix}
	\widehat{v}\]	
      where $\widehat{u}=Pu$, $\widehat{v}=P^{-1}H_1 v$ and
      $M_{b}(z) = \varphi((1-zb)^{-1})$
      is the moment generating function of $b$.
	
	Therefore the eigenvalues $\lambda_i=\lambda_i(z)$ of $z H_1 C_2$ can be understood as a generalization of the subordination functions from free additive convolution.
	At the time of this writing we do not know whether $H_1 C_2$ is actually diagonalizable in general.
        If this turns out  not to be the case, one can use the Jordan canonical form and the
        derivatives $M_{b}'$, $M_{b}''$,\ldots{} will populate the upper triangular
        parts fo the  Jordan blocks.
\end{remark}

Let us illustrate Theorem~\ref{thm:1.1} with a quick derivation of the
subordination functions for additive free convolution and for the anti-commutator.
For further examples see Section~\ref{sec:examples}.
\begin{example}
  \label{ex:1:additiveconvolution}
  The polynomial $P(a,b)=a+b$  is already linear and  we can apply
  Theorem~\ref{thm:1.1} to obtain
	\begin{align*}
          \E_{a}\bigl[(1-z(a+b))^{-1} \bigr]
          =(1-z a-z F_b(z))^{-1}	
	\end{align*}
        and
        $$
        \tilde{\eta}_{a+b}(z) = F_a(z)+F_b(z)
        $$
	where
	\begin{align*}
          F_a(z) &= \widetilde{\eta}_{a}\Bigl(\frac{z}{1-zF_b(z)}\Bigr) & 
          F_b(z) &= \widetilde{\eta}_{b}\Bigl(\frac{z}{1-zF_a(z)}\Bigr)
                   .
	\end{align*}
	It is
easy to see that these are equivalent to the well known subordination
        results for additive free convolution.
        In this case it is particularly
        straightforward to decouple the equations and obtain separate fixed point equations for $F_1$ and $F_2$
        after a simple substitution of one equation into the other.
        \end{example}
        \begin{example}
          Let $P(a,b)=a b+b a$ be the anti-commutator, then for $z$
          in some neighbourhood of zero
          the conditional expectations of the resolvent are
          \begin{align*}
            \IE_{a}\left[(1-z^2(ab+ba))^{-1}\right]&=\left(1-f_{2,43} z-z^2 a \left(f_{2,33} +f_{2,44} \right)-f_{2,34} a^2 z^3\right)^{-1},\\
            \IE_{b}\left[(1-z^2(ab+ba))^{-1}\right]&=\left(1-f_{1,12} z-z^2 b \left(f_{1,11} +fx_{22} \right)-f_{1,21} b^2 z^3\right)^{-1}.
          \end{align*}
          This result is obtained with the linearization involving the matrices
          \[
            C_1=
            \begin{bmatrix}
              0 & 0 & 0 & 0 \\
              1 & 0 & 0   & 0 \\
              1 & 0 & 0   & 0 \\
              0   & 1 & 0 & 0
            \end{bmatrix}
            \qquad
            C_2=
            \begin{bmatrix}
              0 & 0 & 1 & 0 \\
              0 & 0 & 0 & 1 \\
              0 & 0 & 0 & 1 \\
              0 & 0 & 0 & 0
            \end{bmatrix}.\]
          A priori the equations \eqref{eq:thm:1.1:eta} involve a 
          total of 32 variables $f_{k,ij}$, $k\in\{1,2\},1\leq i,j\leq 4$.
          However it is easy to see with the help of the projection
          matrices onto $\ker C_1$ and $\ker C_2$
          (cf.~Remark~\ref{rem:projectionsQi})
          that many variables vanish and the solution matrices have the form
\[
            F_1 = 
            \begin{bmatrix}
              f_{1,11} & f_{1,12} & 0 & 0 \\
              f_{1,21} & f_{1,22} & 0 & 0 \\
              0 & 0 & 0 & 0 \\
              0 & 0 & 0 & 0
            \end{bmatrix}
            \,
            F_2=
            \begin{bmatrix}
              0 & 0 & 0 & 0 \\
              0 & 0 & 0 & 0 \\
              0 & 0 & f_{2,33} & f_{2,34} \\
              0 & 0 & f_{2,43} & f_{2,44}
            \end{bmatrix}
          \]
          and satisfy the following system of matrix equations
          \begin{align*}
            \begin{cases}
              F_{1}&=\widetilde{\eta}_{a}\left(z Q_1 (I-zC_2 F_2)^{-1}C_1\right),\\
              F_{2}&=\widetilde{\eta}_{b}\left(z Q_2 (I-zC_1 F_1)^{-1}C_2\right),
            \end{cases}
          \end{align*}
          where $Q_i$ is the projection onto the cokernel of the matrix $C_i$,
          i.e.,
          $C_iQ_i=C_i$.
          This system is the analogue for the anti-commutator of the fixed point equation in the case of additive convolution stated in the previous example.
          Note that the explicit formulas for the matrices $H_1=Q_1
          (I-zC_2F_2)^{-1}C_1$ and $H_2=Q_2 (I-zC_1 F_1)^{-1}C_2$ are
          essentially the same as those for matrices $H_a$ and $H_b$ from
          Theorem~6.1 in \cite{FMNS2}, i.e., linearizations were already
          implicitly present in \cite{FMNS2}.
\end{example}

\clearpage{}
\section{Preliminaries}
\label{sec:Preliminaries}

In this section we introduce the main ingredients of the paper:
conditional expectations, Boolean cumulants and derivations.

\subsection{Non-commutative probability spaces}

We assume that $\alg{M}$ is a unital $*$-algebra and
$\varphi:\alg{M}\mapsto\IC$ is a positive unital linear functional, commonly
called a \emph{state} and we usually assume it to be faithful.
We will refer to the pair $(\alg{M},\varphi)$ as a \emph{non-commutative probability
space}.
For technical reasons we will mostly work with the free associative algebra.
\begin{notation}\label{notation:nc_polynomials}
  Let $\alg{X}=\{X_1,X_2,\ldots,X_n\}$ be an alphabet.
  We denote by $\alg{X}^+=\{X_{i_1}X_{i_2}\cdots X_{i_k}\mid k\in\IN, i_j\in
  \{1,2,\dots,n\}\}$
  the free semigroup it generates and by
  $\alg{X}^*=\alg{X}^+\cup\{1\}$ the free monoid.
  
  We denote by $\IC\langle\alg{X}\rangle= \IC\langle X_1,X_2,\ldots,X_n\rangle$ the free associative
  algebra generated by the variables $X_1,X_2,\ldots,X_n$, i.e., the linear
  span of $\alg{X}^*$ with the concatenation product, also known as the algebra of
  non-commutative polynomials.

  For elements $a_1,a_2,\ldots,a_n\in\alg{M}$ and $P\in\IC\langle
  X_1,X_2,\dots,X_n\rangle$
  we denote
  by $P(a_1,a_2,\ldots,a_n)$ the evaluation of a
  polynomial $P\in\IC\langle X_1,X_2,\ldots,X_n\rangle$, i.e., the element
  of $\alg{M}$ obtained after substituting every $X_i$ with $a_i$ for 
  $i=1,2,\ldots,n$.

  The \emph{joint distribution} of $a_1,a_2,\ldots,a_n$ is the linear functional
  $\mu:\IC\langle X_1,X_2,\ldots,X_n\rangle\to\IC$ given by
  \[\mu(P)=\varphi\left(P(a_1,a_2,\ldots,a_n)\right).\]
\end{notation}

\begin{definition}
  \label{def:freeindep}
  A family of subalgebras $(\alg{A}_i)_{i\in I}$ of a ncps $(\alg{M},\varphi)$ is called
  \emph{free} or \emph{free independent} if
  \begin{equation*}
    \varphi(u_1u_2\dotsm u_n)=  0
  \end{equation*}
  for any choice of $u_j\in \bigcup_i\alg{A}_i$ such that $\varphi(u_j)=0$ and
  $u_j\in\alg{A}_{i_j}$ with $i_j\ne i_{j+1}$ for all $j\in\{1,2,\dots,n-1\}$.
\end{definition}

\subsection{Conditional expectations}

Fix a non-commutative probability space $(\alg{M},\varphi)$
and let $\alg{A}\subseteq \alg{M}$ be a subalgebra.
A \emph{conditional expectation}
is a state-preserving projection $\E_{\alg{A}}:\alg{M}\to \alg{A}$,
i.e.,
such that $\varphi\circ\E_{\alg{A}}=\varphi$.
In general such a map not necessarily needs to exists,
unless $\alg{M}$ is a finite von Neumann algebra and $\varphi$ is
tracial \cite[Proposition~5.2.36]{Takesaki1}.
If it does exist and  the state $\varphi$ is faithful, then the conditional expectation
$\E_{\alg{A}}[u]$ 
is the unique  element $\tilde{u}\in\alg{A}$ such that for any
$a\in\alg{A}$ 
one has $\varphi(ua)=\varphi(\tilde{u}a)$.
$\E_{\alg{A}}:\alg{M}\to\alg{A}$ is a unital $\alg{A}$-bimodule map, i.e.,
$E_{\alg{A}}[a_1ua_2]=a_1E_{\alg{A}}[u]a_2$, for all $u\in\alg{M}$ and $a_1,a_2\in\alg{A}$. 

In the present paper we will always assume that
the algebra $\alg{M}$ is freely generated by two of its subalgebras
$\alg{A},\alg{B}\subseteq\alg{M}$,
or more specifically, $\alg{M}$ is freely generated by 
some subalgebras $\alg{A}_i$, $i\in I$
and the subalgebra $\alg{A}$ is generated by some subset of these,
i.e., $\alg{A}=\langle\alg{A}_j\rangle_{j\in J}$ where $J\subseteq I$,
and $\alg{B}=\langle \alg{A}_i\rangle_{i\in I\setminus J}$.
In this case the existence of the conditional expectation $\E_{\alg{A}}$ onto $\alg{A}$
is generally warranted by combinatorial arguments \cite[\S2.5]{MingoSpeicher}.
Alternatively, in the C${}^*$-algebraic context, if the state is faithful,
then $\alg{M}$ is isomorphic to the reduced free product 
$(\alg{A},\varphi|_{\alg{A}})*(\alg{B},\varphi|_{\alg{B}})$
and the existence of the conditional expectation also follows from  \cite[Proposition~1.3]{Avitzour:1982:free}.

More precisely, in order to find the conditional expectation of a non-commutative polynomial
in variables $a_1,a_2,\ldots,a_n$ onto the algebra $\alg{A}$ generated by
$a_1,a_2,\ldots,a_k$  one has to find suitable expressions for moments of the form 
\begin{equation*}
\varphi\left(a_{i_1}a_{i_2}\dotsm a_{i_r} b\right),
\end{equation*}
where $i_1,i_2,\ldots,i_r\in \{1,\ldots,n\}$ and $b\in\alg{A}$.
It is a fundamental property of freeness
(as one of the  universal notions of independence in the sense of \cite{Muraki:2002:five})
that all joint moments of
freely independent random variables are uniquely determined by the marginal moments
of the variables in question. Thus for each moment of the form indicated above
there is a universal formula (not depending on the particular choice of the
distributions of $a_1,a_2,\ldots,a_n$) which expresses any joint moment as a sum
of products of marginal moments.

After fixing random variables $a_1,a_2,\ldots,a_n$ all relevant information we need
for finding the conditional expectation is contained in the pair
$(\IC\langle \uXn\rangle,\mu)$ defined in Notation~\ref{notation:nc_polynomials} and therefore we will mostly work on a formal
level and focus on this non-commutative probability space. Note that such $\mu$ need not to be faithful.

Working with non-commutative polynomials is very useful as it allows us to ignore algebraic relations satisfied by
variables $a_1,a_2,\ldots,a_n$. Another advantage of $\IC\langle \uXn\rangle$
is that it is augmented (see below) and has a natural linear basis, hence we can easily define linear mappings and functional on this algebra by prescribing values on the basis elements. It is also straightforward then, again by linearity, to extend all those mappings to the algebra of formal power series in non-commuting variables $\uXn$, denoted by $\IC\llangle \uXn \rrangle$.

\subsection{Main idea --- Boolean cumulants and conditional expectations}
Boolean cumulants appeared in various contexts and disguises in the literature
\cite{Waldenfels:1973:approach,Terwiel:1974:projection,SpeicherWoroudi:1997:boolean}.
In our previous work \cite{LehnerSzpojan} we observed
that Boolean cumulants appear naturally in connection with conditional expectations of functions in free
variables. The present paper is an exploration of this connection based on a
recursive reformulation which is suitable for explicit calculations in closed form.
In addition we introduce a non-commutative differential calculus for Boolean
cumulants of polynomials in free variables which reduces the 
combinatorial apparatus to a minimum.

Although Speicher's free cumulants are the tool of choice in free probability
\cite{SpeicherNC,NicaSpeicherLect}, 
more recently it turned out that for certain questions Boolean cumulants are useful as well.
Indeed some problems like the free anti-commutator \cite{FMNS2}
and subordination functions \cite{LehnerSzpojan,SzpojanWesol2} are easier
to describe in terms of Boolean cumulants rather than free cumulants.
The present paper extends and unifies these ideas. Before discussing this, 
let us start with a review of some basic facts about Boolean cumulants.
An interval partition is a partition $\pi=\{B_1,B_2,\ldots,B_k\}$ of the set
$\{1,2,\ldots,n\}$ such that all blocks are intervals, i.e.,
for all $1\leq i\leq k$ we have $B_i=\{k,k+1,\ldots,l\}$ for some $k\leq l$ in
$\{1,2,\ldots,n\}$.
The set of all interval partitions of $\{1,2,\ldots,n\}$ is denoted by $\IntPart(n)$.

For any tuple $a_1,a_2,\ldots,a_n\in\alg{M}$ we define the Boolean cumulant
functional $\beta_n:\alg{M}^n\to\IC$ implicitly via the formula
\begin{equation}
  \label{eq:phi=sumbeta}
  \varphi(a_1a_2\dotsm a_n)=\sum_{\pi\in \IntPart(n)} \beta_\pi(a_1,a_2,\ldots,a_n),
\end{equation}
where $\beta_\pi(a_1,a_2,\ldots,a_n)=\prod_{B\in \pi} \beta_{|B|}\left((a_1,a_2,\ldots,a_n)|B\right)$, and by $\beta_{|B|}\left((a_1,a_2,\ldots,a_n)|B\right)$ we mean that we take the functional $\beta_k$ such that $k=|B|$, and we evaluate it at those $a_i$ for which $i\in B$, and the arguments $a_i$ appear in the natural order (one should note that in general Boolean cumulants are not invariant under permutations of arguments).
One way to  inverting the formula  \eqref{eq:phi=sumbeta} and to
obtain an explicit formula for Boolean cumulants is M\"obius inversion on the
lattice of interval partitions.
We refrain from doing so and rather base our calculations on a well known
recurrence, namely
\begin{equation}
	\label{eq:recurrenceboolcum}
	\varphi(a_1a_2\dotsm a_n)
	= \sum_{k=1}^n \beta_k(a_1,a_2,\dots,a_k)\,\varphi(a_{k+1}a_{k+2}\dotsm a_n)
\end{equation}
or equivalently
\begin{equation}
	\label{eq:recurrenceboolcumright}
	\varphi(a_1a_2\dotsm a_n)
	= \sum_{k=1}^n \varphi(a_1a_2\dotsm a_{k-1})\,\beta_{n-k+1}(a_{k},a_{k+1},\dots,a_n).
\end{equation}
This elementary recurrence is the starting point for our investigation.
It immediately implies a similar recurrence
for conditional expectations of products of free random variables:
Assume that $\{a_1,a_2,\ldots,a_{n}\}\in\alg{A}$ and
$\{b_1,b_2,\ldots,b_{n-1}\}\in\alg{B}$ are two families of variables and
assume that subalgebras $\alg{A}$ and $\alg{B}$ are free.

In order to calculate the conditional expectation of
$a_1b_1a_2\dotsm b_{n-1}a_n\in \alg{M}$ onto the subalgebra $\alg{B}$,
we have to find an element $\IE_{\alg{B}}[a_1b_1a_2\dotsm b_{n-1}a_n]\in\alg{B}$ such that 
\[\varphi(a_1b_1a_2\dotsm b_{n-1}a_n b)=\varphi\left(\IE_{\alg{B}}[a_1 b_1a_2\dotsm a_n] b\right)\]
for any $b\in\alg{B}$.
If $\varphi$ is faithful then this element is uniquely determined and can be
found using the following recursive reformulation of  \cite[Proposition~1.1]{LehnerSzpojan}.
\begin{proposition}
Assume that $\varphi$ is faithful. Then for $a_i\in\alg{A}$ and
$b_i\in\alg{B}$ the conditional expectation of the alternating product
satisfies the recurrence
\begin{align}\label{eq:introCondExprecurrence}
\IE_{\alg{B}}[a_1b_1a_2\dotsm a_{n-1}b_{n-1} a_n]=\sum_{k=1}^{n}\beta_{2k-1}(a_1,b_1,a_2,\ldots,
  a_k) \, b_k \,\IE_{\alg{B}}\left[a_{k+1}b_{k+1}\dotsm a_n\right].
\end{align}
\end{proposition}

\begin{proof}
The recurrence \eqref{eq:recurrenceboolcum} yields
\begin{multline*}
  \varphi\left(\IE_{\alg{B}}[a_1b_1a_1\dotsm a_n] b\right)=\varphi(a_1b_1a_2\dotsm a_n b)\\
  =\sum_{k=1}^{n}
    \beta_{2k-1}(a_1,b_1,a_2,\dots, a_k)\,
    \varphi(b_{k}a_{k+1}\dotsm b)
    +\sum_{k=1}^{n}
    \beta_{2k}(a_1,b_1,a_2,\dots, b_k)\,
    \varphi(a_{k+1}b_{k+1}\dotsm b)
\end{multline*}

Now it follows from Theorem \ref{thm:charfree}
(see also Definition~\ref{def:propertyCAC} and
Proposition~\ref{prop:characterization} below)
that for free random variables the mixed Boolean cumulant
$\beta_{2k}(a_1,b_1,a_2,\ldots, a_k,b_k)$ vanishes for $k=1,2,\ldots,n$
and hence
\begin{align*}
  \varphi\left(\IE_{\alg{B}}[a_1b_1a_2\dotsm a_n] b\right)
    &=\sum_{k=1}^{n}\beta_{2k-1}(a_1,b_1,a_2,\dots, a_k) \, \varphi(b_{k}a_{k+1}b_{k+1}\dotsm a_nb)
    \\
    &=\varphi
      \Bigl(
      \sum_{k=1}^{n}\beta_{2k-1}(a_1,b_1,a_2,\dots, a_k)\, b_{k} a_{k+1}b_{k+1}\dotsm       a_nb
      \Bigr)
    \\
    &=\varphi
      \Bigl(
      \sum_{k=1}^{n}\beta_{2k-1}(a_1,b_1,a_2,\dots, a_k)\, b_k \IE_{\alg{B}}\left[a_{k+1}b_{k+1}\dotsm a_n\right]b
      \Bigr).
\end{align*}
\end{proof}
Observe that vanishing of cumulants of  the form
$\beta_{2k}(a_1,b_1,a_2,\ldots,a_k,b_k)$ is essential in this proof,
because it eliminates $\beta_{2n}(a_1,b_1,a_2,\ldots,a_{n-1},b)$.
Otherwise $b$ would be trapped inside this term
and the recurrence would fail.

We will also make use of the original non-recursive version of the formula for the conditional
expectation found in \cite{LehnerSzpojan}.

\begin{corollary}
  \label{cor:condexp}
  Let $(\alg{A},\varphi)$ be a ncps and 
  $\alg{B}\subseteq \alg{A}$  a subalgebra such that the conditional
  expectation $\E_{\alg{B}}:\alg{A}\to\alg{B}$ exists.
  Let  $\{b_1,b_2,\ldots,b_{n-1}\}\subseteq\alg{B}$ and
  assume that the family $\{a_1,a_2,\ldots,a_n\}\subseteq\alg{A}$ is free
  from $\alg{B}$.
  \begin{enumerate}[(i)]
   \item 
  \begin{multline}
    \label{eq:phiYXY}
    \varphi\left(a_1b_1a_2\dotsm b_{n-1}a_nb_n\right)
    \\
    =
    \sum_{k=0}^{n-1}\sum_{1\leq i_1<i_2<\dots< i_k\leq n-1}
    \varphi(b_{i_1}b_{i_2}\dotsm b_{i_k}b_n)\prod_{j=0}^{k}\beta_{2(i_{j+1}-i_j)-1}(a_{{i_j}+1},b_{i_{j}+1},a_{{i_j}+2},\ldots,a_{i_{j+1}}),
  \end{multline}
   \item 
  Then the conditional expectation of alternating monomials can be evaluated as follows
  \begin{multline}
    \label{eq:EYXY}
    \E_{\alg{B}}\left[a_1b_1a_2\dotsm b_{n-1}a_n\right]
    \\
    =
    \sum_{k=0}^{n-1}\sum_{1\leq i_1<i_2<\dots< i_k\leq n-1}
    b_{i_1}b_{i_2}\dotsm b_{i_k}\prod_{j=0}^{k}\beta_{2(i_{j+1}-i_j)-1}(a_{{i_j}+1},b_{i_{j}+1},a_{{i_j}+2},\ldots,a_{i_{j+1}}),
  \end{multline}
  where in the above sum for each sequence $0< i_1<i_2<\dots< i_k<n$ we set $i_0=0$ and $i_{k+1}=n$.
  \end{enumerate}
\end{corollary}

\subsection{Boolean cumulants with products as entries}
In this subsection we present the basic tools preparing
the non-commutative differential calculus for Boolean cumulants of polynomials in free variables.

The main tool is the formula for Boolean cumulants with products as
entries in terms of cumulants of individual variables.
Such formulas are known for classical cumulants
\cite{LeonovShiryaev:1959:method,Speed:1983:cumulantsI} and free cumulants
\cite{KrawczykSpeicher:2000:combinatorics}
and follow a general pattern \cite {Lehner:2004:cumulants1}.
The analogous formula for Boolean cumulants is given below. The proof may go
via a standard argument involving M\"obius inversion (see Lecture 14 in
\cite{NicaSpeicherLect}).
In view of possible generalizations and for the
sake of completeness we present alternative algebraic proofs
based solely on the recurrence  \eqref{eq:recurrenceboolcum} in Appendix~\ref{ssec:algproof}.

\begin{proposition}
  \label{prop:boolprod}
  Let $a_1,a_2,\dots,a_n\in \alg{A}$ be random variables
  then
  \begin{equation}
    \label{eq:BoolProd}
    \beta_{m+1}(a_1a_2\cdots a_{d_1},a_{d_1+1}a_{d_1+2}\cdots a_{d_2},\ldots,a_{d_{m}+1}a_{d_{m}+2}\cdots a_{n})
    =\sum_{\substack{\pi \in \IntPart(n)\\ \pi \vee \rho=1_n}}\beta_\pi(a_1,a_2,\ldots,a_n),
  \end{equation}
  where $\rho=\{\{1,2,\ldots,d_1\},\{d_1+1,d_1+2,\ldots,d_2\},\ldots,\{d_m+1,\ldots,n\}\}\in \IntPart(n)$, and $\vee$ is the join in the lattice of interval partitions.
  The condition $\pi \vee \rho=1_n$ is equivalent to
  \begin{equation*}
\pi \geq \{\{1\},\{2\},\dots,\{d_1-1\},\{d_1,d_1+1\},\{d_1+2\},\dots,\{d_m-1\},\{d_m,d_m+1\},\dots,\{d_{n}\}\}.
  \end{equation*}
\end{proposition}
Proposition~\ref{prop:boolprod} can be proved by iteration of the following
lemma and Corollary~\ref{cor:RecursiveBoolProd} below. 
\begin{lemma}
  \label{lem:productformula}  
  \begin{multline}
    \label{eq:lem:productformula}
    \beta_{n-1}(a_1,a_2,\dots,a_pa_{p+1},a_{p+2},a_{p+3},\dots,a_n)
    \\
    =  \beta_{p}(a_1,a_2,\dots,a_p)
    \,
    \beta_{n-p}(a_{p+1},a_{p+2},a_{p+3},\dots,a_n)
    +
    \beta_{n}(a_1,a_2,\dots,a_n)
  \end{multline}
\end{lemma}
We will actually mostly make use of the following recursive version of Proposition~\ref{prop:boolprod}.
\begin{corollary}
  \label{cor:RecursiveBoolProd}
Let $a_1,a_2,\dots,a_n\in \alg{A}$ be random variables consider the interval partition $\rho=\{\{1,\ldots,d_1\},\{d_1+1,\ldots,d_2\},\ldots,\{d_m+1,\ldots,n\}\}\in \IntPart(n)$. We write $\rho=\{B_1,B_2,\ldots,B_{m+1}\}$, where blocks are ordered in natural order. For $j\in\{1,\ldots,n\}$ denote by $\rho(j)$ the number of block containing $j$, i.e. we have $\rho(j)=k$ if $j\in B_{k}$, 
then
\begin{equation}
  \label{eq:RecursiveBoolProd}
  \begin{multlined}
	\beta_{m+1}(a_1a_2\cdots a_{d_1},a_{d_1+1}a_{d_1+2}\cdots a_{d_2},\ldots,a_{d_{m}+1}a_{d_{m}+2}\cdots a_{n})\\
	=\sum_{j\in\{1,\ldots,n\}\setminus\{d_1,d_2,\ldots,d_m\}}\beta_{j}(a_1,a_2,\ldots,a_j)\beta_{m-\rho(j)+1)}(a_{j+1}a_{j+2}\cdots a_{d_{\rho(j)}},\ldots,a_{d_m+1}\cdots a_n).
  \end{multlined}
\end{equation}
\end{corollary}

We record here one immediate consequence of Lemma~\ref{lem:productformula}
which allows to eliminate units. This also follows from Proposition \ref{prop:vnrp} below.
\begin{corollary}[{\cite[Proposition~3.3]{Popa:2009:multiplicative}}]
  \label{cor:boolcumunit}
  For any $n\geq 2$ we have
  \begin{align}
  \label{eq:boolcumunit:1}
    \beta_n(1,a_2,\dots,a_n)
    &=0\\
\beta_n(a_1,a_2,\dots,a_{n-1},1)
    &=0\\
  \label{eq:boolcumunit:3}
    \beta_n(a_1,a_2,\dots,a_{k-1},1,a_{k+1},\dots,a_n)
    &= \beta_{n-1}(a_1,a_2,\dots,a_{k-1},a_{k+1},\dots,a_n)
  \end{align}
\end{corollary}

\subsection{Tensor products and tensor algebras}
The \emph{tensor product} $U\otimes V$ of two vector spaces has the universal property
that every bilinear map $B:U\times V\to W$ has a unique extension to a linear map $B:U\otimes V \to W$. 
Consequently, a family of multilinear maps $f_n:V^n\to W$, $n\geq 0$,
corresponds uniquely to a linear map $f:T(V) \to W$ on the tensor algebra
$T(V) = \oplus_{n=0}^\infty V^{\otimes n}$.
Moreover, any linear map on $V$ can be extended to a derivation of the tensor
algebra (Lemma~\ref{lem:derivTV}).

\begin{notation}\label{notation:tensor}
  It will be convenient to denote the product operation on the tensor algebra by $\odot$ and extend it to matrices as follows:
  Given a matrix $A=[a_{ij}]\in M_n(\alg{A})$ and $a\in\alg{A}$, let $a\odot A = [a\otimes a_{ij}] \in M_n(\alg{A}\otimes \alg{A})$.
  Similarly, for two matrices $A,  B\in M_n(\alg{A})$ we denote by $A\odot B\in M_n(\alg{A}\otimes \alg{A})$ the matrix
  with entries
  $$
  (A\odot B)_{ij} = \sum_k a_{ik}\otimes b_{kj}
  .
  $$
  Note that associativity holds in connection with multiplication with scalar matrices $C\in M_n(\IC)$,
  in the sense that  $A\odot CB = AC\odot B$.
\end{notation}

\subsection{Free products of algebras}
  A \emph{coproduct} or  (algebraic) \emph{free product} of unital algebras $\alg{A}_1$ and $\alg{A}_2$
  over a field $\IK$ is a unital algebra $\alg{A}$ with embeddings
  $\iota_1:\alg{A}_1\to\alg{A}$
  and $\iota_2:\alg{A}_2\to\alg{A}$ such that the images generate $\alg{A}$
  and every pair of homomorphisms
  $h_1:\alg{A}_1\to \alg{B}$ and  $h_2:\alg{A}_2\to \alg{B}$ has a unique
  extension to a homomorphism $\alg{A}\to\alg{B}$.

  For details about free products and tensor algebras we refer to
  \cite[\S1.4]{BeidarMartindaleMikhalev:1996:rings}
  and \cite{Bourbaki:A1-3}.
  
  \begin{proposition}
\label{prop:algfreeproduct}
    \begin{enumerate}[(i)]
     \item []
     \item The free product is unique and is given by the quotient of the
      tensor algebra
      $T^+(A_1\oplus A_2)=\sum_{n=1}^\infty (A_1\oplus A_2)^{\otimes n}$ with respect to the ideal generated by all elements of
      the form
      $$
      a_1\otimes b_1 - a_1b_1
      \qquad
      a_2\otimes b_2 - a_2b_2
      \qquad
      1_{\alg{A}_1} - 1_{\alg{A}_2}
      ,
      \qquad
      a_i,b_i\in\alg{A}_i
      .
      $$
      The algebraic free product is denoted by $\alg{A}_1\freeprod\alg{A}_2$.
     \item \cite[Lemma~1.4.5]{BeidarMartindaleMikhalev:1996:rings}
    Let $\alg{A}_1$ and $\alg{A}_2$ be algebras with 1 over a field $\IK$ with
    respective
    $\IK$-bases $\{1\}\cup M_1$ and  $\{1\}\cup M_2$.
    Then $\{1\}\cup M$ is a $\IK$-basis for $\alg{A}_1\FP \alg{A}_2$ where
    $M$ is the set of alternating monomials in letters from $M_1$ and $M_2$.
    \end{enumerate}
  \end{proposition}

  \begin{definition}
    \label{def:augmentation}
    An algebra $\alg{A}$ is called \emph{augmented} if it comes with an
    \emph{augmentation map}, i.e., an algebra homomorphism $\epsilon:\alg{A}\to\IC$.
    Its kernel $\widebar{\alg{A}}=\ker\epsilon$ is called the
    \emph{augmentation ideal}.
  \end{definition}

  \begin{example}
    \label{ex:polyalg:augmented}
    Typical examples of augmented algebras are
    polynomial algebras (both commutative and non-commutative)
    where the augmentation map
    $$
    \epsilon(P) = P(0) = \text{constant coefficient}
    $$
    is clearly a homomorphism.
  \end{example}
  
  The free product of augmented algebras is clearly augmented.
  Moreover, it is isomorphic to a subalgebra of the
  tensor algebra and this fact will be helpful for the definition of certain
  functionals to be defined in Section~\ref{sec:calcbbeta} below.
  \begin{proposition}[\cite{Hungerford:1968:free}]
    \label{prop:freeprodaugmented}
    The free product of augmented algebras is isomorphic to
    \begin{equation}
      \label{eq:freeprodaugmalg}
    \alg{A}\freeprod\alg{B} \simeq
    \IC1\oplus\bigoplus_{n=1}^\infty T_n(\widebar{\alg{A}},\widebar{\alg{B}})\oplus T_n(\widebar{\alg{B}},\widebar{\alg{A}})
    \end{equation}
    where by
    $$
    T_n(U,V) = U\otimes V\otimes U\otimes \dotsm \quad\text{($n$ factors)}
    $$
    we denote the alternating tensor product of two vector spaces.
    The multiplication is given by the tensor product modulo the identifications
    $a'\otimes a'' = a'a''$
    and $b'\otimes b'' = b'b''$ already present in Proposition~\ref{prop:algfreeproduct}.
\end{proposition}
  \begin{example}
    The free associative algebra $\IC\langle X,Y\rangle$ is isomorphic to
    algebraic free product  $\IC[X]\freeprod\IC[X]$.
    The decomposition
    \eqref{eq:freeprodaugmalg} corresponds to the decomposition into the
    constant term and alternating monomials
    $X^{k_1}Y^{l_1}X^{k_2}Y^{l_2}\dotsm$ with exponents $k_j,l_j>0$.
  \end{example}
    \begin{notation}
    Extending the terminology of the preceding example to free products of general augmented algebras,
    alternating products of the form
    $
    w = a_1b_1a_2b_2\dotsm  
    $
    (resp.~$w = b_1a_1b_2a_2\dotsm$) 
    with $a_i\in\widebar{\alg{A}}$ and $b_i\in\widebar{\alg{B}}$,
    i.e., the images of elementary tensors from $T_n(\alg{A},\alg{B})$ (resp.~ $T_n(\alg{B},\alg{A})$),
    will be called \emph{monomials}.
  \end{notation}

  The reduced free product of ncps $(\alg{A},\varphi)$ and $(\alg{B},\psi)$
  will be denoted by  $(\alg{A},\varphi)*(\alg{B},\psi)$ or
  $\alg{A}*\alg{B}$ in short.
  It 
  is realized as the image of the free product of their GNS representations, see   
  \cite[\S1.5]{VoiDykNica}. For our purpose the following subalgebra is sufficient.
  \begin{proposition}[{\cite[\S1.5]{VoiDykNica}}]
    \label{prop:reducedfreeprod}
    Let $(\alg{A}_1,\varphi_1)$ and $(\alg{A}_2,\varphi_2)$ be ncps and denote by
    $\bub{A}_i=\ker\varphi_i$ their centered components.
    Then the orthogonal direct sum
    $$
    \IC1\oplus \bigoplus_{n=1}^\infty\oplus \bigoplus_{i_1\ne i_2\ne\cdots\ne i_n}
    \bub{\alg{A}}_{i_1}\bub{\alg{A}}_{i_2} \dotsm \bub{\alg{A}}_{i_n} 
    $$
    is a dense subalgebra of the reduced free product.
  \end{proposition}
  
  \begin{definition}
    Unital subalgebras $\alg{A}$, $\alg{B}$ of an algebra $\alg{M}$ are
    called \emph{algebraically free} if they do not satisfy any mutual
    algebraic relation,
    i.e., if the free extension of the embeddings $\iota_A:\alg{A}\to \alg{M}$,
    $\iota_B:\alg{B}\to \alg{M}$ to a homomorphism
    $h:\alg{A}\FP\alg{B}\to\alg{M}$ is injective.
  \end{definition}
  
  \begin{proposition}
    \label{prop:algfree}
    Let $(\alg{M},\varphi)$ be a ncps with faithful state $\varphi$
    and let $\alg{A}$, $\alg{B}$ be freely independent subalgebras
    in the sense of Definition~\ref{def:freeindep}.
    Then  $\alg{A}$ and $\alg{B}$ are algebraically free.
  \end{proposition}
  \begin{proof}
    The algebra generated by $\alg{A}$ and $\alg{B}$ is isomorphic
    to their reduced free product $\alg{A}*\alg{B}$
    and thus by \cite[Proposition~2.3]{Avitzour:1982:free} contains a faithful copy
    of the algebraic free product.
  \end{proof}

\subsection{Derivations}
Coincidentally it turns out that similar to classical calculus, integration has
strong ties to derivations.
More precisely, we will be concerned with several  derivations
from an algebra $\alg{A}$ into the bimodule $\mathfrak{M} = \alg{A}\otimes\alg{A}$ with
the natural action
\begin{equation}
  \label{eq:AotimAmodule}
a_1\cdot (x\otimes y) \cdot a_2 = (a_1\otimes 1)(x\otimes y)(1\otimes a_2) =
a_1x\otimes ya_2
\end{equation}
\begin{definition}
  Let $\alg{A}$ be an algebra and $\mathfrak{M}$ be an $\alg{A}$-bimodule.
  An \emph{$\mathfrak{M}$-derivation} is a linear map $D:\alg{A}\to\mathfrak{M}$ satisfying
  the Leibniz rule
  \begin{equation}
    \label{eq:leibniz}
  D(a_1a_2) = D(a_1)\cdot a_2 + a_1\cdot D(a_2)
\end{equation}
\end{definition}
If  one allows modifications of the  left and right actions of $\alg{A}$ on
$\alg{A}\otimes\alg{A}$ then
homomorphisms become a rich source of derivations.
\begin{proposition}
  \label{prop:derivphi1-phi2}
Let $\alg{A}$ be a unital algebra and
$\Phi_1,\Phi_2:\alg{A}\to\alg{A}\otimes\alg{A}$ be two homomorphisms,
then $D = \Phi_1-\Phi_2$ is a derivation in the sense that
$D(ab) = D(a)\Phi_2(b) + \Phi_1(a)D(b)$,
i.e.,
regarding $\alg{A}\otimes\alg{A}$ as a bimodule with action
$$
a_1\cdot (x\otimes y) \cdot a_2 = \Phi_1(a_1)(x\otimes y)\Phi_2(a_2)
.
$$
\end{proposition}

The following examples of derivations will play a major role in the
developments below.
\begin{example}
  \label{ex:lnabla}
  The ``tensor commutators''
  $\lnabla:\alg{A}\to\alg{A}\otimes\alg{A}$
mapping $\lnabla a = a\otimes 1-1\otimes a$
  is a derivation for any algebra $\alg{A}$.
  It appears in \cite[Section~5.3]{Voiculescu:1999:entropy6} (see Section~\ref{ssec:characterization} below)
  and has the universal property that every derivation factors through it
  \cite[p.~III.132, Proposition~17]{Bourbaki:A1-3}.
  It will be convenient to denote $\rnabla=-\lnabla$.
\end{example}

\begin{example}
  \label{ex:freederiv}
  The \emph{free derivative} or \emph{free difference quotient} on $\IC[x]$ is the map
  $\partial:\IC[x]\to\IC[x]\otimes\IC[x]$
  given by
  \begin{equation*}
    \partial x^n = \sum_{k=0}^{n-1} x^k\otimes x^{n-k-1}
    .
  \end{equation*}
  In the natural identification $\IC[x]\otimes\IC[x]\simeq \IC[x,y]$ this
  coincides with the difference quotient
  \begin{equation*}
    \partial p(x) =  \frac{p(x)-p(y)}{x-y}
  \end{equation*}
  and for this reason is also known as the \emph{Newtonian coproduct}
  \cite[\S  XII]{JoniRota:1978:coalgebras}.
  It first appeared in free probability in connection with the non-commutative Hilbert transform approach to free entropy
  \cite {Voiculescu:1998:entropy5}.
\end{example}

\begin{example}
  \label{ex:divpower}
  A slight modification of the previous derivation comes from the \emph{divided powers coproduct}
  \cite[\S  VI]{JoniRota:1978:coalgebras}.
  Let again $\alg{A}=\IC[X]$ be the polynomial algebra,
  then the deconcatenation coproduct is a homomorphism
  $$
  \Delta_d(x^n) = \sum_{k=0}^n x^k\otimes x^{n-k}
  $$
  and thus both
  \begin{align*}
    \ldelta p(x) &= \Delta_d p(x) - 1\otimes p(x) = (x\otimes  1)\partial p(x)\\
    \rdelta p(x) &= \Delta_d p(x) - p(x)\otimes 1 = (1\otimes  x)\partial p(x)
  \end{align*}
  are derivations.
\end{example}
\begin{example}
  \label{ex:blockderiv}
  For an augmented algebra, the map $\Delta(a) = a-\epsilon(a)$ is a derivation
  in the sense of Proposition~\ref{prop:derivphi1-phi2}
  and so are the left and right ``block derivatives''
  \begin{equation*}
    \begin{aligned}
      \rDelta:\alg{A}&\to\alg{A}\otimes\alg{A}   &    \lDelta:\alg{A}&\to\alg{A}\otimes\alg{A}   \\
      a&\mapsto 1\otimes( a-\epsilon(a))         &   a&\mapsto(a-\epsilon(a))\otimes 1
      .
    \end{aligned}
  \end{equation*}

\end{example}

Next we discuss free products of derivations.
\begin{proposition}
  Let $\alg{A}_1$ and $\alg{A}_2$ be algebras and
  $\alg{A}=\alg{A}_1\freeprod\alg{A}_2$ their (unital) free product.
  Let $\mathfrak{M}$  an $\alg{A}$-bimodule (equivalently, a bimodule for both
  $\alg{A}_1$ and $\alg{A}_2$)
  and $D_i:\alg{A}_i\to \mathfrak{M}$, $i=1,2$ be derivations.
  Then there exists a unique derivation $D_1\ast D_2:\alg{A}_1\freeprod\alg{A}_2\to\mathfrak{M}$
  extending $D_1$ and $D_2$.
  Moreover, we have the decomposition
  $$
  D_1\ast D_2= D_1\ast0 + 0\ast D_2
  $$
  where $0$ is the trivial derivation.
\end{proposition}

\begin{proof}
  We first extend $\tilde{D}=D_1\oplus D_2:\alg{A}_1\oplus\alg{A}_2\to\mathfrak{M}$
  to the tensor algebra $T^+(\alg{A}_1\oplus\alg{A}_2)$ by the Leibniz rule,
  see Lemma~\ref{lem:derivTV}:
$$
  \tilde{D}(u_1\otimes u_2\otimes\dotsm \otimes u_n)=\sum_{k=1}^n
  u_1\otimes u_2\otimes u_{k-1}\cdot \tilde{D}(u_k)\cdot u_{k+1}\otimes\dotsm \otimes u_n
  $$
  and then check that it passes to the quotient
  (see  Proposition~\ref{prop:algfreeproduct}),
  i.e.,
  the ideal generated by $1_{\alg{A}}-1_{\alg{B}}$, $a'\otimes a''-a'a''$ for
  $a',a''\in\alg{A}_1\cup\alg{A}_2$ is contained in  $\ker\tilde{D}$.
  Clearly $\tilde{D}(1)=0$ and on the other hand for $a',a''\in\alg{A}_1$
  \begin{equation*}
    \begin{multlined}
  \tilde{D}(u\otimes(a'\otimes a''-a'a'')\otimes v) =
  \tilde{D}(u)\cdot (a'\otimes a''-a'a'')\otimes v)
  \\
  + 
  u\cdot(D_1(a')\cdot a''
  + a'\cdot D_1(a'')   - D_1(a'a''))\cdot v
  \\
  +
  u\otimes (a'\otimes a''-a'a'')\cdot\tilde{D}(v)
    \end{multlined}
  \end{equation*}
  which is mapped to $0$ in the quotient space
  and the Leibniz rule is satisfied by definition.
\end{proof}

\begin{example}
  If we identify the algebra of non-commutative polynomials with the unital free product
  $\IC\langle x,y\rangle \simeq \IC[ x]\ast\IC[x]$,
  then we can construct the partial derivatives $\ldelta_x$, $\rdelta_x$, $\partial_x$ and
  $\lnabla_x$ 
  as  free products of the derivations $\lnabla$ (Example~\ref{ex:lnabla}), $\ldelta$, $\rdelta$
  (Example~\ref{ex:divpower})  and $0$  on $\IC[x]$:
  
  \begin{align*}
    \ldelta_x&=\ldelta{}\!\star{} 0 &    \ldelta_y&=0\ast\ldelta \\
    \rdelta_x&=\rdelta{}\!\star{} 0 &    \rdelta_y&=0\ast\rdelta\\        
    \partial_x&=\partial\star 0 &  \partial_y&=0\ast\partial \\
    \lnabla_x&=\lnabla{}\!\ast 0 &    \lnabla_y&=0\ast\lnabla \\
    \rnabla_x&=\rnabla{}\!\ast 0 &  \rnabla_y&=0\ast\rnabla\\
  \end{align*}
  moreover, we obtain decompositions $\ldelta = \ldelta_x+\ldelta_y$ and $\lnabla = \lnabla_x+\lnabla_y$.
\end{example}

\begin{example}
  \label{ex:DeltaA}
  Let $\alg{A}$ and $\alg{B}$ be augmented algebras and set
  $\mathfrak{M}=(\alg{A}\freeprod\alg{B})\otimes(\alg{A}\freeprod\alg{B})$.
  Denote by $\lDelta_{\alg{A}}$ and $\lDelta_{\alg{B}}$
  (resp.~$\rDelta_{\alg{A}}$ and   $\rDelta_{\alg{B}}$) the left (resp.~right) block derivatives
  from Example~\ref{ex:blockderiv} into the corresponding submodules of $\mathfrak{M}$.
  The free products   $\rDelta_{\alg{A}}{}\!\star 0,0*\rDelta_{\alg{B}}:\alg{A}\freeprod\alg{B}\to \mathfrak{M}$
  are called \emph{partial block derivatives}.
  By abuse of notation we will denote them by $\rDelta_{\alg{A}}$ and
  $\rDelta_{\alg{B}}$ as well.
  Their action on the augmentation ideal is deconcatenation
  on monomials
  \begin{align*}
    \rDelta_{\alg{A}} (a_1b_1a_1a_2b_2\dotsm a_nb_n) 
    &= \sum_{k=1}^n  a_1b_2a_2\dotsm b_k \otimes a_k b_{k+1}a_{k+1}\dotsm      a_nb_n
    \\
    \lDelta_{\alg{A}} (a_1b_1a_1a_2b_2\dotsm a_nb_n) 
    &= \sum_{k=1}^n  a_1b_2a_2\dotsm a_k \otimes b_ka_{k+1} b_{k+1}\dotsm      a_nb_n
  \end{align*}
  where $a_i\in\widebar{\alg{A}}$ and $b_i\in\widebar{\alg{B}}$.
  In contrast to the previous derivations, neither the partial block derivatives
  nor the deconcatenation operator $\rDelta=\rDelta_{\alg{A}}+\rDelta_{\alg{B}}$
  satisfy the Leibniz rule with respect to the natural action
  \eqref{eq:AotimAmodule},
but the ``compatibility relation'' of unital infinitesimal bialgebras applies \cite[Definition~2.1]{LodayRonco:2006:structure}
  $$
  \rDelta (uv) = (\rDelta u)(1\otimes v) + (u\otimes 1)\rDelta v - u\otimes v.
  $$
  Note however that $\rDelta_{\alg{A}}-\lDelta_{\alg{A}}=\rnabla_{\alg{A}}$ is
  a derivation.
\end{example}

We conclude the preliminaries with the observation that resolvents behave nicely under derivations
and will in a certain sense serve as an analogue of the exponential function in classical calculus.
\begin{remark}
Let $\alg{A}$ be a unital algebra and $D:\alg{A}\to \mathfrak{M}$ be a derivation into an
$\alg{A}$-bimodule $\mathfrak{M}$. Then for any invertible element $a\in\alg{A}$ we have
as a consequence of the Leibniz rule
\begin{equation}
	\label{eq:derivationofinverse}
	D(a^{-1})=-a^{-1}D(a)a^{-1}  
\end{equation}
In particular, the derivation of a resolvent $R=(z-a)^{-1}$ satisfies
\begin{equation*}
D(R) = RD(a)R
\end{equation*}
while for $\Psi=(1-za)^{-1}$ we have
\begin{equation}
	\label{eq:derivationofpsi}
	D(\Psi) = z\Psi D(a)\Psi
	.
\end{equation}	
\end{remark}

\clearpage{}
\section{Boolean cumulants of free random variables}
\label{sec:boolchar}

\hfill{}
\parbox{0.6\textwidth}{
  \begin{verse}
\emph{If you climb up a tree, \\ you must climb down the same tree.}
  \end{verse}

  \hfill{}proverb from Sierra Leone\footnotemark
}
\footnotetext{D.~Crystal, \emph{As they say in Zanzibar},
Collins, London 2006}

\subsection{Characterization of freeness by Boolean cumulants}
\label{ssec:characterization}
In this section we elaborate on the main property which allowed us to derive
recurrence \eqref{eq:introCondExprecurrence}, i.e., 
the fact that mixed Boolean cumulants of free variables which start and end
with mutually free variables vanishes.
The main result is  a surprisingly simple characterization of freeness
which is inspired by an algebraic characterization of freeness found by
Voiculescu in \cite{Voiculescu:1999:entropy6}.
We are grateful to R.~Speicher and J.~Mingo for bringing this paper to our attention.
\begin{definition}
  \label{def:propertyCAC}
  \begin{enumerate}[(i)]
   \item []
   \item 
  Subalgebras $\alg{A},\alg{B}$ of a ncps $(\alg{M},\varphi)$ have vanishing
  \emph{cyclically alternating cumulants}  if
  $$
  \beta_{n}(u_1,u_2,\dots,u_n)=0
  $$
  whenever $u_i\in\alg{A}\cup\alg{B}$ such that $u_1$ and $u_n$ come from
  different algebras.
  We will call this \emph{Property $(CAC)$}.
   \item 
  Subalgebras $\alg{A},\alg{B}$ of a ncps $(\alg{M},\varphi)$ satisfy 
   \emph{Property $(WCAC)$} (weak $(CAC)$), if they satisfy Property $(CAC)$ for alternating
  words,   i.e.
  $$
  \beta_{n}(a_1,b_1,a_2,b_2,\dots,a_n,b_n)=0
  $$
  and 
  $$
  \beta_{n}(b_1,a_1,b_2,a_2,\dots,b_n,a_n)=0
  $$
  for any choice of $a_i\in\alg{A}$, $b_i\in\alg{B}$.
 \item
  Subalgebras $\alg{A},\alg{B}$ of a ncps $(\alg{M},\varphi)$
satisfy property $(\nabla)$ if
  \begin{equation}
    \label{eq:propertynablaA}
  (\varphi\otimes\varphi)\circ\lnabla_{\alg{A}}(a_1b_1a_2b_2\dotsm a_nb_n)=0
  \end{equation}
  for any choice of $a_i\in\alg{A}$ and $b_i\in\alg{B}$.
  Here the expression on the right hand side of \eqref{eq:propertynablaA}
  is meant to be evaluated as follows:
  \begin{enumerate}[1.]
   \item The formal derivative $\lnabla_x x_1y_1x_2y_2\dotsm x_ny_n$ is computed in
    the free associative algebra
    $\IC\langle x_1,x_2,\dots,x_n,y_1,y_2,\dots,y_n\rangle$;
   \item substitute $x_i=a_i$, $y_i=b_i$ to obtain an element of
    $\alg{M}\otimes\alg{M}$;
   \item evaluate $\varphi\otimes\varphi$ on the latter.
  \end{enumerate}
  \end{enumerate}
\end{definition}
\begin{remark}
  Since $\lnabla = \lnabla_{\alg{A}}+\lnabla_{\alg{B}}$ and trivially
  $(\varphi\otimes\varphi)\circ\lnabla=0$,
  identity \eqref{eq:propertynablaA} is equivalent to
    \begin{equation*}
  (\varphi\otimes\varphi)\circ\lnabla_{\alg{B}}(a_1b_1a_2b_2\dotsm a_nb_n)=0
  \end{equation*}
\end{remark}

\begin{lemma}
  \label{lem:altABimpliesAB}
 Property $(WCAC)$  is equivalent to Property $(CAC)$.
\end{lemma}
\begin{proof}
  Clearly Property $(CAC)$ implies  Property $(WCAC)$.
  The converse can be seen as a special case of Lemma~\ref{lem:BoolSplitGroup} below,
  but here is a proof by induction,  assuming that   Property $(WCAC)$ holds for all orders.
  
  Suppose that Property $(CAC)$ holds  for all orders up to
  $n-1$ and 
  pick a tuple $u_1,u_2,\dots,u_n\subseteq \alg{A}\cup\alg{B}$ such that $u_1$
  and
  $u_n$ come from different algebras.
  
  If the tuple is alternating then there is nothing to prove.

  Therefore assume that it is not alternating.
  Without loss of generality we consider the following configuration (the proof
  of the other cases is  analogous):
  $u_1=a,u_k=a',u_{k+1}=a''\in\alg{A}$ and $u_n=b\in\alg{B}$.
  Thus we have to show that the cumulant
  $\beta_n(a,u_2,u_2,\dots,a',a'',u_{k+2},\dots,u_{n-1},b)$ vanishes.
  We apply the product formula \eqref{eq:lem:productformula} in the reverse direction and obtain
  \begin{multline*}
    \beta_n(a,u_2,\dots,u_{k-1},a',a'',u_{k+2},\dots,u_{n-1},b)
    \\
    =
\beta_{n-1}(a,u_2,u_2,\dots,u_{k-1},a'a'',u_{k+2},\dots,u_{n-1},b)\\
      -\beta_{k}(a,u_2,u_2,\dots,u_{k-1},a')\,\beta_{n-k}(a'',u_{k+2},\dots,u_{n-1},b)
\end{multline*}
  All cumulants on the right hand side are of lower order and with the
  exception of the
  second one satisfy the
  assumptions of the induction hypothesis. It follows that the right hand side
  vanishes.
\end{proof}

\begin{proposition}
  \label{prop:ABiffnabla}
  Property $(WCAC)$  is equivalent to Property $(\nabla)$.
\end{proposition}
\begin{proof}
We apply recurrence \eqref{eq:recurrenceboolcumright} to the
first factor of the first sum
and \eqref{eq:recurrenceboolcum} to the
second factor of the second sum in the following expression
  \begin{multline*}
    (\varphi\otimes\varphi)(\lnabla_{\alg{A}} a_1b_1\dotsm a_nb_n)\\
    \begin{aligned}[t]
    &=
\sum_{k=1}^n\varphi(a_1b_1\dotsm a_k)\,\varphi(b_ka_{k+1}\dotsm a_nb_n)
      - \varphi(a_1b_1\dotsm a_{k-1}b_{k-1})\,\varphi(a_kb_k\dotsm a_nb_n)
    \\
    &=
    \begin{multlined}[t]
      \sum_{k=1}^n \sum_{p=1}^k\varphi(a_1b_1\dotsm b_{p-1})\,\beta_{2(k-p)+1}(a_p,b_p,\dots,a_k)\,\varphi(b_ka_{k+1}\dotsm a_nb_n)
      \\
      + \sum_{k=1}^n \sum_{p=1}^{k-1}\varphi(a_1b_1\dotsm a_p)\,\beta_{2(k-p)}(b_p,\dots,a_k)\,\varphi(b_ka_{k+1}\dotsm a_nb_n)      
      \\
      - \sum_{k=1}^n \sum_{p=k}^n\varphi(a_1b_1\dotsm b_{k-1})\,\beta_{2(p-k)+1}(a_k,b_k,\dots,a_p)\,\varphi(b_pa_{p+1}\dotsm a_nb_n)
      \\
      - \sum_{k=1}^n \sum_{p=k}^n\varphi(a_1b_1\dotsm b_{k-1})\,\beta_{2(p-k)+2}(a_k,b_k,\dots,b_p)\,\varphi(a_{p+1}\dotsm a_nb_n)
    \end{multlined}
    \\
    &=
    \begin{multlined}[t]
      \sum_{1\leq p\leq k\leq n}\varphi(a_1b_1\dotsm b_{p-1})\,\beta_{2(k-p)+1}(a_p,b_p,\dots,a_k)\,\varphi(b_ka_{k+1}\dotsm a_nb_n)
      \\
      + \sum_{k=1}^n \sum_{p=1}^{k-1}\varphi(a_1b_1\dotsm a_p)\,\beta_{2(k-p)}(b_p,\dots,a_k)\,\varphi(b_ka_{k+1}\dotsm a_nb_n)      
      \\
      -       \sum_{1\leq k\leq p\leq n}\varphi(a_1b_1\dotsm b_{k-1})\,\beta_{2(p-k)+1}(a_k,b_k,\dots,a_p)\,\varphi(b_pa_{p+1}\dotsm a_nb_n)
      \\
      - \sum_{k=1}^{n-1} \sum_{p=k}^n\varphi(a_1b_1\dotsm
      b_{k-1})\,\beta_{2(p-k)+2}(a_k,b_k,\dots,b_p)\,\varphi(a_{p+1}\dotsm
      a_nb_n)
      \\
      -\beta_{2n}(a_1,b_1,\dots,a_n,b_n)
    \end{multlined}
    \\
    &=
    \begin{multlined}[t]
      \sum_{k=1}^n \sum_{p=1}^{k-1}\varphi(a_1b_1\dotsm a_p)\,\beta_{2(k-p)}(b_p,\dots,a_k)\,\varphi(b_ka_{k+1}\dotsm a_nb_n)      
      \\
      - \sum_{k=1}^{n-1} \sum_{p=k}^n\varphi(a_1b_1\dotsm
      b_{k-1})\,\beta_{2(p-k)+2}(a_k,b_k,\dots,b_p)\,\varphi(a_{p+1}\dotsm
      a_nb_n)
      \\
      -\beta_{2n}(a_1,b_1,\dots,a_n,b_n)
    \end{multlined}
  \end{aligned}
\end{multline*}
Now all but the last term on the right hand side
involve lower order cumulants which vanish by by induction hypothesis.
  Therefore the left hand side vanishes if and only if $\beta_{2n}(a_1,b_1,\dots,a_n,b_n)=0$.
\end{proof}

\begin{lemma}
  \label{lem:beta(ab)=0}
  Free subalgebras satisfy Property $(CAC)$.
\end{lemma}
\begin{proof}
  This is an immediate consequence of the vanishing of mixed free cumulants and
  the formula expressing Boolean cumulants as a sum of free cumulants
  indexed by irreducible non-crossing partitions
  \cite{Lehner:2002:connected,BelinschiNicaBBPforKTuples},
  and is also a special case of Proposition~\ref{prop:vnrp} below.

  Here is a direct self-contained proof using only the recurrence
  \eqref{eq:recurrenceboolcum} and induction.
  By Lemma~\ref{lem:altABimpliesAB} it suffices to prove Property $(WCAC)$.

  To begin with, in the case $n=2$ the recurrence \eqref{eq:recurrenceboolcum}
  immediately resolves into the covariance:
  $$
  \beta_2(a,b) = \varphi(ab)-\varphi(a)\varphi(b)=0
  .
  $$

  For the induction step,
  we first verify that cumulants of alternating centered words vanish
  and then reduce the general case to this.

  Assume that the assertion holds for any order up to $k\leq n-1$
  and that  the tuple $u_1,u_2,\dots,u_n$ is alternating,
  i.e., neighbouring elements come from different algebras,
  and that all elements are centered.
  Then it follows from recurrence \eqref{eq:recurrenceboolcum} that
  \begin{equation*}
    \beta_{n}(u_1,u_2,\dots,u_n)
    =     \varphi(u_1u_2\dotsm u_n) -
    \sum_{k=1}^n\beta_k(u_1,u_2,\dots,u_k)\,\varphi(u_{k+1}u_{k+2}\dotsm u_n)
  \end{equation*}
  and all terms on the right hand said contain expectations of alternating
  words in centered elements. Freeness implies that they vanish and so does
  the cumulant on the left hand side.

  It remains to show that we can replace the letters of an alternating word
  with centered elements.
  We will do this step by step using multilinearity and
  Corollary~\ref{cor:boolcumunit}.
  The condition that the first and last arguments of the involved cumulants
  come from different algebras is not violated throughout the following manipulations.
  \begin{align*}
    \beta_n(u_1,u_2,\dots,u_n)
    &=     \beta_n(\bub{u}_1,u_2,\dots,u_n) +    \varphi(u_1)\,\beta_n(1,u_2,\dots,u_n) \\
    &=     \beta_n(\bub{u}_1,u_2,\dots,u_n) &    \text{by    \eqref{eq:boolcumunit:1}}\\
    &=     \beta_n(\bub{u}_1,\bub{u}_2,u_3,\dots,u_n) +\varphi(u_2)\,\beta_n(\bub{u}_1,1,u_3,\dots,u_n) \\
    &=     \beta_n(\bub{u}_1,\bub{u}_2,u_3,\dots,u_n) +\varphi(u_2)\,\beta_{n-1}(\bub{u}_1,u_3,\dots,u_n)     &\text{by    \eqref{eq:boolcumunit:3}}\\
    &=     \beta_n(\bub{u}_1,\bub{u}_2,u_3,\dots,u_n) &  \text{by  induction hypothesis} \\
    &\ \ \vdots\\
    &=     \beta_n(\bub{u}_1,\bub{u}_2,\dots,\bub{u}_n)
      .
  \end{align*}
\end{proof}

\begin{remark}
  \label{rem:quenell}
  In the case of random walks on groups, Boolean cumulants correspond to first
  return probabilities. The fact that Boolean cumulants of free variables
  vanish unless the first and last argument come from the same subalgebra
  is an algebraic counterpart of the following elementary dendrological exercise:
  the Cayley graph of a free product group is a bouquet of tree-like
  graphs. Hence any closed walk starting at the root and entering a certain
  subtree must pass  through
  the same subtree on its first return to the root;
  see \cite{Quenell:1994:combinatorics} for a detailed exposition.
\end{remark}

In conclusion, in the tracial case we have the following extension of Voiculescu's freeness criterion \cite[\S14.4]{Voiculescu:1999:entropy6}.

\begin{proposition}[Characterization of freeness in terms of Boolean cumulants]
  \label{prop:characterization}
  Let $(\alg{M},\varphi)$ be a tracial non-commutative probability space
  and $\alg{A},\alg{B}\subset\alg{M}$  two  subalgebras.
  Then the following are equivalent.
  \begin{enumerate}[(i)]
   \item\label{it:charfree1} $\alg{A}$ and $\alg{B}$ are free.
   \item\label{it:charfree2} $\alg{A}$ and $\alg{B}$ satisfy property $(CAC)$.
   \item\label{it:charfree3} $\alg{A}$ and $\alg{B}$ satisfy property $(\nabla)$.
  \end{enumerate}
\end{proposition}
\begin{proof}
  Items \eqref{it:charfree2} and \eqref{it:charfree3} are equivalent by
  Proposition~\ref{prop:ABiffnabla} even without traciality.

  Item \eqref{it:charfree2} implies \eqref{it:charfree3} by
  Lemma~\ref{lem:beta(ab)=0} and it remains to prove the converse in the tracial case, i.e., Property~$(CAC)$
  implies freeness.
	
  Recall that in \cite{Voiculescu:1999:entropy6} it is shown that in the
  tracial case for every $n$ the original freeness condition $(F(n))$
  \begin{equation*}
\varphi(a_1b_1\cdots a_n b_n )=0 
  \end{equation*}
  whenever all  elements $a_i$, $b_i$ are centered can be strengthened
  to condition $(F(n)')$ allowing one of $a_1$ or $b_n$ to have nonzero expectation. 
	
  We will use this equivalence and  proceed by induction
  and  show  that for every $n$ 
  condition $(F(k)')$ for $k<n$ together with $(CAC)$ implies condition $(F(n))$.
  To this end we pick centered elements $a_1,a_2,\ldots,a_n\in\alg{A}$
  and $b_1,b_2,\ldots,b_n\in\alg{B}$ and  show that
  $\varphi(a_1b_1\cdots a_n b_n)=0$.
  
  The case $n=1$ is obvious.

  For the induction step, observe that we can rewrite recurrence
  \eqref{eq:recurrenceboolcum} as follows
  \begin{multline*}
    \varphi(a_1b_1\dotsm a_n b_n)
    =
    \sum_{k=1}^{n} \beta_{2k-1}(a_1,b_1,\ldots, a_k)\,\varphi(b_{k}a_{k+1}b_{k+1}\dotsm b_n)
    \\
    +\sum_{k=1}^{n}\beta_{2k}(a_1,b_1,\ldots, b_k)\, \varphi(a_{k+1}b_{k+1}\dotsm b_n)
    .
  \end{multline*}
  All terms in the second sum vanish by property (CAC).
  In the first sum  the term corresponding to $k=n$ vanishes because we assumed
  that $\varphi(b_n)=0$. 
  In order to see
  that the remaining terms corresponding to $k<n$ vanish as well,
  first note that
  by traciality   $\varphi(b_{k}a_{k+1}\ldots b_n)=\varphi(a_{k+1}\ldots b_n b_k)$
  which is an alternating product with at most $2(n-1)$ factors.
  Now the element $b_n b_k$ needs not to be centered but  as discussed above, 
  the freeness conditions $(F(n-1)')$   and $(F(n-1))$ are equivalent 
  and thus the expectation vanishes.
\end{proof}
\begin{remark}
  Note that traciality of the linear functional is essential for this
  characterization. Property $(CAC)$ (and hence Property $(\nabla)$) also holds
  for Boolean independent and more generally conditionally free
  random variables. 
\end{remark}

\subsection{Mixed Boolean cumulants of free random variables}
The following characterization of freeness from \cite{FMNS2} generalizes
Property $(CAC)$ and will provide an essential tool for later considerations.

\begin{proposition}[{\cite[Theorem~1.2]{FMNS2}, \cite{JekelLiu:2019:operad}}]\label{prop:vnrp}
  Subalgebras $\alg{A}_1,\alg{A}_2,\ldots,\alg{A}_s\subseteq\alg{M}$ of a
  ncps $(\alg{M},\varphi)$ are free if and only if for any colouring
  $c:\{1,\ldots,n\}\to\{1,\ldots,s\}$ we have
  \begin{equation*}
\beta_n(a_1,a_2,\dots,a_n) = \sum_{\pi\in\NCirr(n) \text{ with VNRP}} \beta_\pi(a_1,a_2,\dots,a_n)
  \end{equation*}
  whenever  $a_i\in\alg{A}_{c(i)}$.
  Here a  partition $\pi\in \NCirr(n)$ is said to have $VNRP$ if $\pi\leq\ker
  c$ and every inner block covered nested by a block of different colour, i.e.,
  $c$ induces a proper coloring on the nesting tree of $\pi$.
\end{proposition}
We will only use the preceding theorem in the special case of alternating
arguments,
which was first explicitly stated in \cite{SzpojanWesol2}.
\begin{proposition}
	Let $\{a_1,a_2,\ldots,a_{n}\}$ and $\{b_1,b_2,\ldots,b_{n-1}\}$ be free, $n\ge 1$. Then 
	\begin{equation}
		\label{bocu2}
                \begin{multlined}
                  \beta_{2n-1}(a_1,b_1,\ldots,a_n,b_{n-1},a_{n})\\
                  =\sum_{k=2}^{n}\,\sum_{1=j_1<j_2<\dots<j_k=n}\beta_k(a_{j_1},a_{j_2},\ldots,a_{j_k})\,\prod_{\ell=1}^{k-1}\,\beta_{2(j_{\ell+1}-j_{\ell})-1}(b_{j_l},a_{j_{\ell}+1},\ldots,a_{j_{\ell+1}-1},b_{j_{\ell+1}-1}).
                \end{multlined}
	\end{equation}
\end{proposition}

Next we record another property which will be important in the following,
namely as we already know a Boolean cumulant which takes as arguments elements
of two free algebras $\alg{A}$ and $\alg{B}$ vanishes
unless it starts and ends with elements from the same algebra.
Suppose this is the case, say the both the first and last argument come from
$\alg{A}$, then the Boolean cumulant does not change under arbitrary splitting
of thee arguments coming from the subalgebra $\alg{B}$ into factors.
This will allow us to write any Boolean cumulant in free variables as an alternating cumulant.

\begin{lemma}\label{lem:BoolSplitGroup}
  Suppose $\alg{A},\alg{B}\subseteq\alg{M}$ are free and let
  $a_1,a_2,\ldots,a_n\in\alg{A}$ and
  $b_1,b_2,\ldots,b_{n-1}\in\alg{B}$.
  Assume further that for each $j=1,2,\ldots,n-1$ we have $b_i=c^{(i)}_1\cdots c^{(i)}_{j_i}$ with $c^{(i)}_1,\ldots, c^{(i)}_{j_i}\in\alg{B}$, then we have
	\begin{align*}
		\beta_{2n-1}(a_1,b_1,a_2,\ldots,b_{n-1},a_n)=\beta_{n+j_1+\ldots+j_{n-1}}(a_1,c^{(1)}_1,\ldots, c^{(1)}_{j_1},a_2,\ldots,c^{(n-1)}_1,\ldots ,c^{(n-1)}_{j_{n-1}},a_n).
	\end{align*}
\end{lemma}
\begin{proof}
  This follows from the formula for Boolean cumulants with products as
  entries \eqref{eq:BoolProd} and  property $(CAC)$.
  More precisely,
  applying said formula to the Boolean cumulant
  $$
  \beta_{n+j_1+\dots+j_{n-1}}(a_1,c^{(1)}_1c^{(1)}_2\dotsm  c^{(1)}_{j_1},
    a_2,\dots,c^{(n-1)}_1c^{(n-1)}_2\dotsm c^{(n-1)}_{j_{n-1}},a_n)
  $$
  recall that we have to sum over interval partitions $\pi$ such that $\pi \vee \sigma=1_{n+j_1+\dots+j_{n-1}}$, where $\sigma$ is given by the product structure, hence we sum over $\pi$ greater than $\{\{1,2\},\{3\}\dots,\{j_1\},\{j_1+1,j_1+2\},\dots\{n-1,n\}\}$ or in other words are complementary
  to $\tilde{\rho}=(1,j_1,1,j_2,\dots,j_{n-1},1)$.
  Observe that the block of $\pi$ containing $a_1$ cannot end in any
  $c_{i}^{l}$, as in that case the corresponding  Boolean cumulant  vanishes
  by property $(CAC)$. On the other hand the block containing $a_1$ cannot
  finish in any of the variables $a_1,a_2,\dots,a_{n-1}$ as this would violate
  the complementarity property. Thus the block which contains $a_1$ has to end
  in $a_n$ and since we consider only interval partitions, there is only one such partition $\pi=1_{n+j_1+\dots+j_{n-1}}$.
\end{proof}

\begin{remark}
  \label{rem:alternating}
  Lemma~\ref{lem:BoolSplitGroup} allows  to rewrite an arbitrary joint Boolean cumulant
  of free random variables as an alternating cumulant, by grouping together terms from the
  algebra $\alg{B}$ above.
  Indeed
  after regrouping the inner variables into
  free blocks, Corollary~\ref{cor:boolcumunit} allows us to fill the gaps
  between elements of $\alg{A}$ with  units  to obtain
  $$
  \beta_{k}(a_1,a_2,\dots,a_{l},a_{l+1},\dots,a_n)
  =\beta_{k}(a_1,1_{\alg{M}},a_2,\dots,a_{l},1_{\alg{M}},a_{l+1},\dots,a_n)
  .
  $$ 
\end{remark}

\clearpage{}
\section{Calculus for conditional expectations}\label{sec:calcCond}

In order to define several objects objects used in this and the next section we require
additional structure on the ncps we work with.
\begin{assumption}
  \label{ass:augmented}
  From this point we will assume that $(\alg{M},\varphi)$ is a ncps,
  where $\alg{M}$ is an augmented algebra (see
  Definition~\ref{def:augmentation}).
  Furthermore we assume 
  $\alg{A}$, $\alg{B}\subseteq \alg{M}$ are freely independent subalgebras which generate
  $\alg{M}$ as an algebra, that is, $\alg{M}=\alg{A}\freeprod\alg{B}$.
  The latter is a moderate restriction since it is a dense subalgebra
  by   Proposition~\ref{prop:algfree}.
\end{assumption}

\begin{remark}
  \label{rem:augmented-polynomialalg}
  \begin{enumerate}[(i)]
   \item []
   \item 
  This assumption is necessary to make sure that the functionals introduced
  below are well   defined.
  Indeed in the augmented case
  the algebraic free product is canonically isomorphic to a subspace of the tensor
  algebra (see Proposition~\ref{prop:freeprodaugmented})
  and   thus families of multilinear
  functionals can be identified with linear functionals on the latter.
 \item   \label{rem:augmented-polynomialalg:it:polynomial}
  The reader not familiar with these notions should think of the algebras as  polynomial
  algebras from Example~\ref{ex:polyalg:augmented} in the following. More
  precisely,
  $$
  \alg{M} = \IC\langle X_1,X_2,\dots,X_m,Y_1,Y_2,\dots,Y_n\rangle
  \qquad
  \alg{A} = \IC\langle X_1,X_2,\dots,X_m\rangle
  \qquad
  \alg{B} = \IC\langle Y_1,Y_2,\dots,Y_n\rangle
  $$
  and $\varphi=\mu$ is some formal distribution, i.e., a unital linear
  functional $\mu:\alg{M}\to\IC$.
  The augmentation is $\epsilon(P)=P(0)$.
  
  In fact it is sufficient to understand the bivariate case
  $$
  \alg{M} = \IC\langle X,Y\rangle
  \qquad
  \alg{A} = \IC\langle X\rangle
  \qquad
  \alg{B} = \IC\langle Y\rangle
  $$
  and the general case will be straightforward.

  In the case of polynomial algebras $w\in\widebar{\alg{M}}$ means that $w$ has zero constant term
  and in fact it is sufficient to consider monomials $w=a_1b_2a_2b_2\dotsm$
  where
  $a_i$ and $b_i$ come from the monomial basis,
  i.e., 
  in the case of $\IC\langle X,Y\rangle$ this means that $a_i\in\{X^k|k\geq 1\}$
  and  $b_i\in\{Y^k|k\geq 1\}$.
  \end{enumerate}

\end{remark}

\begin{notation}

  It follows from Proposition~\ref{prop:freeprodaugmented} that
  in the free product of augmented algebras any  monomial (i.e., simple tensor)
  $W\in\widebar{\alg{M}}$   has a unique factorization into an
  alternating product of elements from $\widebar{\alg{A}}$ and $\widebar{\alg{B}}$
  in one out of four types:
  \begin{align*}
    \text{type $\alg{A}$--$\alg{A}$}& &\text{if }
                                                w&=a_1b_1a_2\dotsm b_{n-1}a_n,
                                                   \\
\text{type $\alg{A}$--$\alg{B}$} && \text{if }
                                            w&=a_1b_1a_2\dotsm a_nb_{n},
                                                   \\
  \text{type $\alg{B}$--$\alg{A}$} &&\text{if }
                                              w&=b_1a_1b_2\dotsm b_na_{n},
                                                   \\
  \text{type $\alg{B}$--$\alg{B}$} &&\text{if }
                                              w&=b_1a_1b_2\dotsm a_{n-1}b_n,
  \end{align*}
  with $a_1,a_2,\dots,a_{n}\in\widebar{\alg{A}}$ and
  $b_1,b_2,\dots,b_{n-1}\in\widebar{\alg{B}}$.
We will call such monomials $\alg{A}$--$\alg{A}$ monomials etc.~and
  we will refer to the factorization as the $\alg{A}$--$\alg{B}$ \emph{block factorization}.
\end{notation}

\begin{remark}\label{rem:42}
  At this point augmentation is essential. Otherwise allowing
  units in monomials leads to inconsistencies like
  $
  (1+a)b = b + ab
  $ belonging to the image of $\alg{A}\otimes \alg{B}$ and $(\alg{B}\oplus
  \alg{A}\otimes \alg{B})$ simultaneously.
  However the direct sum decomposition \eqref{eq:freeprodaugmalg}
  is essential to consistently extend the multilinear maps to be defined shortly
  (see Definition~\ref{def:bbeta}) 
  to linear maps on the free product.
\end{remark}

We now take formula \eqref{eq:EYXY} as a formal definition of a conditional expectation.
\begin{definition}
  \label{def:condexpploy}
  Define a linear mapping $\E_{\alg{B}}:\alg{M}\to \alg{B}$ via the following requirements:
  \begin{enumerate}[(i)]
   \item
    $\E_{\alg{B}}[1]=1$,
   \item
    $\E_{\alg{B}}[bwb']=b\E_{{B}}[w]b'$ for any $w\in\alg{M}$ and any $b,b'\in\alg{B}$.
   \item
    We define the conditional expectation of a
    monomial of type $\alg{A}$--$\alg{A}$
    $w=a_1b_1a_2\dotsm b_{n-1}a_n$ with $a_i\in\widebar{\alg{A}}$ and $b_i\in\widebar{\alg{B}}$
    as
    \begin{multline*}
      \E_{\alg{B}}\left[a_1b_1a_2\dotsm b_{n-1}a_n\right]
      = \beta_{2n-1}\left(a_1,b_1,a_2,\ldots ,b_{n-1},a_n\right)+\\
      \sum_{k=1}^{n-1}\sum_{1\leq i_1<i_2<\dots< i_k\leq n-1}
      b_{i_1}b_{i_2}\dotsm b_{i_k}\prod_{j=0}^{k}\beta_{2(i_{j+1}-i_j)-1}(a_{{i_j}+1},b_{i_{j}+1},a_{{i_j}+2},\ldots,a_{i_{j+1}}),
    \end{multline*}
   \item We define the \emph{block cumulant} functional
    $\bbeta{}:\alg{M}\to\IC$ by prescribing the values on monomials
    \begin{align*}	
      \bbeta{}(1)&=1\\
      \bbeta{}(u_1u_2\dotsm u_n)
                          &= \beta_{n}(u_1,u_2,\dots, u_n)
    \end{align*}
    whenever $u_1u_2\dotsm u_n$ is an alternating word with $u_i\in\widebar{\alg{A}}\cup\widebar{\alg{B}}$.
  \end{enumerate} 
\end{definition}

\begin{remark}
  \label{rem:augmentation}
  \begin{enumerate}[(i)]
   \item []
   \item 
    Note that property $(CAC)$ implies $\beta{}(u_1u_2\dotsm u_n)$ vanishes
    unless $n$ is odd.
   \item  \label{rem:item:augmentation}
    Augmentation is essential here to make sure that the definition of $\bbeta{}$
    does not depend on the choice of basis.
    Indeed  if instead we choose orthonormal bases
    $\{1\}\cup(a_k)_{k\in\IN}\subseteq \alg{A}$ and
    $\{1\}\cup(b_k)_{k\in\IN}\subseteq \alg{B}$ then
    it follows from Proposition~\ref{prop:reducedfreeprod} that
    the alternating words in $a_i$ and $b_j$ together with $1$ form an
    orthonormal basis of the reduced free product.
    However $\beta(u_1,u_2,\dots,u_n)=0$ for any such word
    and thus the corresponding notion of $\bbeta{}$ would coincide with
    $\varphi$ in this case.
  \end{enumerate}
\end{remark}
As an immediate consequence of Proposition~\ref{prop:freeprodaugmented} and
\eqref{eq:phiYXY}
we have the following.

\begin{proposition}
  Consider the maps defined in Definition~\ref{def:condexpploy} with
  Assumption~\ref{ass:augmented},
  $\varphi$ not necessarily faithful.
  \begin{enumerate}[(i)]
   \item The maps $\E_{\alg{B}}$ and $\bbeta{}$ are well defined.
   \item The map $\E_{\alg{B}}$ defined above is a conditional expectation.
  \end{enumerate}
\end{proposition}

Note that the above conditional expectation is universal and allows for
calculations on the level of
arbitrary algebras.
\begin{proposition}
  \label{prop:condexpformaleval}
  Let $(\alg{M},\varphi)$ be an arbitrary ncps
  with faithful state $\varphi$ 
  and $\alg{A},\alg{B}\subseteq\alg{M}$ be free subalgebras such that the conditional
  expectation $\E_{\alg{B}}$ exists (e.g., when $\alg{B}$ and $\alg{M}$ are von
  Neumann algebras or $\alg{M}=\alg{A}*\alg{B}$).
  Pick elements   $a_1,a_2,\ldots,a_m\in \alg{A}$
  and elements $b_1,b_2,\dots,b_n\in\alg{B}$.
  Let
  $$\mu:\IC\langle   X_1,X_2,\dots,X_m,Y_1,Y_2,\dots,Y_n\rangle\to\IC$$ be
  the joint distribution of the tuple $(a_1,a_2,\dots,a_m,b_1,b_2,\dots,b_n)$,
  i.e., 
  $$
  \mu(P(X_1,X_2,\dots,Y_n)) = \varphi(P(a_1,a_2,\dots,a_m,b_1,b_2,\dots,b_n))
  $$
  and let $\E^\mu_Y:\IC\langle X_1,X_2,\dots, X_m,Y_1,Y_2,\dots,Y_n\rangle\to\IC\langle Y_1,Y_2,\dots,Y_n\rangle$ be the conditional expectation
  from Definition~\ref{def:condexpploy} onto the subalgebra $\IC\langle Y_1,Y_2,\dots,Y_n\rangle$.
 Then	\[
          \E_{\alg{B}}[P(a_1,a_2,\dots,a_m,b_1,b_2,\dots,b_n)]
          =(\E_Y^{\mu}[P])(b_1,b_2,\dots,b_n),
        \]
	where $(\E_Y^{\mu}[P])(b_1,b_2,\dots,b_n)$
        is the
        polynomial $\,\E_Y^{\mu}[P]\in \IC \langle Y_1,Y_2,\dots,Y_n \rangle$ evaluated in $b_1,b_2,\ldots,b_n$.
\end{proposition}
\begin{proof}
  This is an immediate consequence of the definition of $\mu$ and Corollary \ref{cor:condexp}.
\end{proof}

It is no surprise that this formal conditional expectation satisfies the 
recurrence   \eqref{eq:introCondExprecurrence}. 

\begin{proposition}\label{Prop:34}
    For any $\alg{A}$--$\alg{A}$ monomial $w$ the mapping $\E_{\alg{B}}$ satisfies the following recurrence
  \begin{align}
    \label{eq:EBrecurrence}
    \E_{\alg{B}}[a_1b_1a_2\dotsm b_{n-1}a_n]
    &=
      \begin{multlined}[t]
        \beta_{2n-1}\left(a_1,b_1,a_2,\ldots ,b_{n-1},a_n\right)\\
    +\sum_{k=1}^{n-1}\beta_{2k-1}(a_1,b_1,a_2,\ldots,a_k)\,b_k\E_{\alg{B}}[a_{k+1}b_{k+1}a_{k+2}\dotsm
    a_n]
      \end{multlined}
    \\
    \label{eq:EBrecurrencebbeta}    
    &=
      \begin{multlined}[t]
        \bbeta{}\left(a_1b_1a_2\dotsm b_{n-1}a_n\right)\\
        +\sum_{k=1}^{n-1}\bbeta{}(a_1b_1a_2\dotsm a_k)\,b_k\E_{\alg{B}}[a_{k+1}b_{k+1}a_{k+2}\dotsm a_n].
      \end{multlined}
  \end{align}
\end{proposition}
\begin{proof}
  Observe that in Definition~\ref{def:condexpploy} (iii) on the right hand side
  of the equation we sum over all possible choices of subsets of
  $\{b_1,b_2,\dots,b_{n-1}\}$ and for variables between the chosen $b_i$'s we apply
  Boolean cumulant, where we split $b_i$'s from $a_i$'s. In order to obtain the
  recurrence above we reorganize the sum appearing in the definition of
  $\E_{\alg{B}}$ according to the value of $i_1$, i.e., we have
  \begin{multline*}
    \sum_{k=1}^{n-1}\sum_{1\leq i_1<\dots< i_k\leq n-1}
    b_{i_1}b_{i_2}\dotsm b_{i_k}\prod_{j=0}^{k}\beta_{2(i_{j+1}-i_j)-1}(a_{{i_j}+1},b_{i_{j}+1},a_{{i_j}+2},\ldots,a_{i_{j+1}})
    \\
    =\sum_{i_1=1}^{n-1} \beta_{2i_k-1}(a_1,b_1,a_2,\ldots,a_{i_1}) b_{i_1}
    \sum_{k=0}^{n-2}\sum_{i_1<i_2<\dots< i_k\leq n-1}
    b_{i_2}b_{i_3}\dotsm b_{i_k}
    \\
    \times
    \prod_{j=1}^{k}\beta_{2(i_{j+1}-i_j)-1}(a_{{i_j}+1},b_{i_{j}+1},a_{{i_j}+2},\ldots,a_{i_{j+1}}).
  \end{multline*}
  To conclude the proof observe that each fixed value of $i_1$ the inner summation over
        $i_2,\ldots,i_{n-1}$ reproduces exactly the definition of
        $\E_{\alg{B}}[a_{i_1+1}b_{i_1+1}a_{i_1+2}\dotsm a_n]$.
\end{proof}

Proposition \ref{Prop:34} offers a nice recursive formula for calculations of
conditional expectations in free variables, however a drawback is that it works
only on monomials starting and ending elements of algebra $\alg{A}$. In
order to make this formula useful for general calculations we need to have a
recurrence which works on any given polynomial.
To this end we split the functional $\bbeta{}$ somewhat analogously to the
partial block derivatives
$\rDelta=\rDelta_{\alg{A}}+\rDelta_{\alg{B}}$.
The latter turns out to be the right tool to capture the structure of the recurrence from Proposition \ref{Prop:34}.

\begin{definition}\label{def:bbeta}
  Under Assumption~\ref{ass:augmented}
  we define the \emph{block cumulant functional}
  $\bbeta{\alg{A}}:\alg{A}\freeprod\alg{B}\to\IC$
  by prescribing its values on monomials as follows:
  \begin{align*}	
    \bbeta{\alg{A}}(1)&=1\\
\bbeta{\alg{A}}(a_1b_1a_2\dotsm b_{n-1}a_n)
                          &= \beta_{2n+1}(a_1,b_1,a_2,\dots, b_{n-1},a_n)
  \end{align*}
  for any $\alg{A}$--$\alg{A}$ monomial $w$ with
  $\alg{A}$--$\alg{B}$ factorization $w=a_1b_1a_2\dotsm b_{n-1}a_n$ and
  $a_i\in\widebar{\alg{A}}$ and $b_i\in\widebar{\alg{B}}$.
  For monomials $w$ which are not of type $\alg{A}$--$\alg{A}$ we set 	$\bbeta{\alg{A}}(w)=0$.
	
  We define the analogous functional $\bbeta{\alg{B}}$,
  which produces the value of Boolean cumulants of ``blocked'' variables from $\widebar{\alg{A}}$ and $\widebar{\alg{B}}$,
  which vanishes unless a words starts and ends with elements from $\widebar{\alg{B}}$.
\end{definition}
\begin{remark}
  Note that because of Property $(CAC)$ we have
  $$
  \bbeta{}-\epsilon=\bbeta{\alg{A}}-\epsilon+\bbeta{\alg{B}}-\epsilon
  .
  $$
\end{remark}

With the help of these functionals and the block derivatives from Example~\ref{ex:DeltaA} we
can now extend recurrence
\eqref{eq:EBrecurrencebbeta} for $\E_{\alg{B}}$ to arbitrary monomials and
to compress it into an intuitive formula, which is the main result  of this section.

\begin{theorem}
  \label{thm:condexprec}
  Suppose $\alg{M}=\alg{A}\freeprod\alg{B}$ and Assumption~\ref{ass:augmented},
  then for any $w\in\alg{M}$  
  \begin{equation} 
    \label{eq:condexprec}
    \begin{aligned}
    \E_{\alg{B}}[w] &= \bbeta{\alg{A}}(w)
                      + (\bbeta{\alg{A}}\otimes \E_{\alg{B}}) [\rDelta_{\alg{B}}(w)]\\
                    &= \bbeta{\alg{A}}(w)
                      + (\bbeta{\alg{A}}\otimes\E_{\alg{B}}) [\rnabla_{\alg{B}}(w)]\\
                    &= \bbeta{\alg{A}}(w) + 
                            (\E_{\alg{B}}\otimes\bbeta{\alg{A}})[\lDelta_{\alg{B}}(w)].
    \end{aligned}
  \end{equation}
\end{theorem}
\begin{proof}
  Consider the first identity.  
  By linearity it is enough to consider  monomials. For a monomial of type
  $\alg{A}$--$\alg{A}$
  the statement is equivalent to Proposition~\ref{Prop:34}.
  
  If $w$ is a monomial of type $\alg{B}$--$\alg{B}$ or
  $\alg{B}$--$\alg{A}$  then by the definition of
  $\bbeta{\alg{A}}$  only one term contributes on the RHS and
  we obtain the trivial identity
  $$
  \E_{\alg{B}}[w]=
  (\bbeta{\alg{A}}\otimes
  \E_{\alg{B}})[1\otimes w]
  .
  $$
  If $w$ is a monomial of type $\alg{A}$--$\alg{B}$ then it can be written as $w=w'b$, where $w'$ is of type $\alg{A}$--$\alg{A}$ and $b\in\widebar{\alg{B}}$ and the right hand side is
\begin{align*}
  \bbeta{\alg{A}}(w'b) + (\bbeta{\alg{A}}\otimes E_{\alg{B}}) [\rDelta_{\alg{B}}(w') (1\otimes b) + w'\otimes b]
  &= 0 + 
  \Bigl(
   (\bbeta{\alg{A}}\otimes E_{\alg{B}})[ \rDelta_{\alg{B}} w'] + \bbeta{\alg{A}}(w')
  \Bigr)b\\
  &= \E_{\alg{B}}[w']b,
\end{align*}
where we observed that the term inside the parenthesis
contains  the recurrence \eqref{eq:EBrecurrence}
for the conditional expectation of the word $w'$
which is of type $\alg{A}$--$\alg{A}$.

For the second identity
recall that   $\rnabla_{\alg{B}}= \rDelta_{\alg{B}}-\lDelta_{\alg{B}}$.
and observe the additional terms produced by $\lDelta_{\alg{B}}$ are tensors whose left
legs are monomials of type $*$--$\alg{B}$
which are annihilated by $\bbeta{\alg{A}}$.

The last identity is proved using the mirrored version
\eqref{eq:recurrenceboolcumright}
of the Boolean recurrence.
\end{proof}

\begin{remark}
  \begin{enumerate}[(i)]
   \item []
 \item
    It is straightforward to extend the functional $\bbeta{}$ and the recurrence
  \eqref{eq:condexprec}
  to  formal power series (provided that the resulting series converge) and in
  particular compute conditional expectations of resolvents of
  the form $(1-zw)^{-1}$ with $w\in\alg{M}$.
    \item
     In the setting of Remark~\ref{rem:augmented-polynomialalg} \eqref{rem:augmented-polynomialalg:it:polynomial},
i.e.,
  $\alg{A}=\IC\langle \alg{X}\rangle$ for some alphabet $\alg{X}=\{X_1,X_2,\dots,X_m\}$
  and   $\alg{B}=\IC\langle \alg{Y}\rangle$ for some alphabet
  $\alg{Y}=\{Y_1,Y_2,\dots,Y_n\}$
  we can also use the partial divided power derivations from Example~\ref{ex:divpower}
  \begin{align*}
    \rdelta_{\alg{Y}}(P)&=\sum_{i=1}^n (1\otimes Y_i)\partial_{Y_i}(P),\\
    \ldelta_{\alg{Y}}(P)&=\sum_{i=1}^n (Y_i\otimes 1)\partial_{Y_i}(P),
  \end{align*} 
  where $\partial_Y$ is the free different quotient
  from Example~\ref{ex:freederiv}
  and we can write
  \begin{align}
    \label{eq:EY=bbetardeltaY}
    \E_{\alg{Y}}^\mu[P] &= \bbeta{\alg{X}}(P) + (\bbeta{\alg{X}}\otimes
                      \E_{\alg{Y}}^\mu)[\rdelta_{\alg{Y}}(P)]
                      .
  \end{align}
  Again the additional terms are annihilated by $\bbeta{\alg{X}}$.

 \item 
  The map $\rDelta_{\alg{B}}$  produces the least number of terms and
  therefore is convenient for explicit calculations.
  On the other hand, $\rnabla_{\alg{B}}$
  (and $\rdelta_{\alg{Y}}$ in the case of   polynomial algebras)
  are derivations, which is a great  advantage in connection with
  non-commutative formal power
  series,  in particular resolvents  and other rational
  functions.
  Indeed, for simplicity let us consider the resolvent $\Psi =
  (1-zP(X,Y))^{-1}$ of a bivariate polynomial $P(X,Y)\in\IC\langle X,Y\rangle$.
  From identity   \eqref{eq:derivationofinverse} we infer  that
  $$
  \rnabla_X\Psi = z (\Psi\otimes 1) \rnabla_X P(X,Y) (1\otimes\Psi)
  $$
  and a similar formula is true for $\rdelta_X\Psi$.
   It does \emph{not} hold for $\rDelta_X$, which does not obey the
   Leibniz rule, yet in the case of  recurrence \eqref{eq:condexprec}
   we can pretend that it does and we have
    \begin{align*}
      \E_{X}[\Psi]
      &= \bbeta{Y}(\Psi)
                      + z(\bbeta{Y}\otimes \E_{X}) [(\Psi\otimes
                         1)(\rdelta_X w)(1\otimes \Psi]\\
 &= \bbeta{Y}(\Psi)
                      + z(\bbeta{Y}\otimes \E_{X}) [(\Psi\otimes
                         1)(\rnabla_X w)(1\otimes \Psi]\\
&= \bbeta{Y}(\Psi)
                      + z(\bbeta{Y}\otimes \E_{X}) [(\Psi\otimes
                         1)(\rDelta_X w)(1\otimes \Psi]
    \end{align*}
    because the extra terms arising in $\rdelta_Xw$ and $\rnabla_Xw$ are annihilated by $\bbeta{Y}$.
\end{enumerate}
\end{remark}
\begin{corollary}
  For any $w\in\alg{M}$
  $$
  (\bbeta{\alg{A}}\otimes \E_{\alg{B}})[\rDelta_{\alg{A}}w]
  =   (\bbeta{\alg{A}}\otimes \E_{\alg{B}})[\lDelta_{\alg{A}}w]
  .
  $$
\end{corollary}
\begin{proof}
  Observe that from     \eqref{eq:condexprec} we obtain the identity
  \begin{align*}
    (\bbeta{\alg{A}}\otimes \E_{\alg{B}})[\rnabla w]
    &= \E_{\alg{B}}[w] - \bbeta{\alg{A}}(w) \\
    &= (\bbeta{\alg{A}}\otimes \E_{\alg{B}})[\rnabla_{\alg{B}} w];
  \end{align*}
  on the other hand, $\rnabla=\rnabla_{\alg{A}} + \rnabla_{\alg{B}}$ and it
  follows that
  $$
  (\bbeta{\alg{A}}\otimes \E_{\alg{B}})[\rnabla_{\alg{A}} w]=0
  $$
  which is equivalent to the claimed identity.
\end{proof}
We illustrate this machinery with the problem of additive free convolution,
done in two ways.
\begin{example}
  \label{example:condExpAdditive}
  Consider $\Psi=(1-z(X+Y))^{-1}=\sum_{n=0}^{\infty}
  \left(z(X+Y)\right)^n$ then $\rdelta_{X}\Psi=z\Psi\otimes X\Psi$ and
  the conditional expectation is
  \begin{equation*}
\E_{X}[\Psi]
    =\bbeta{Y}(\Psi) + z\bbeta{Y}\otimes \E_{X}[\Psi\otimes    X\Psi] = \bbeta{Y}(\Psi)(1+z X \E_{X}[\Psi])
    .
  \end{equation*}
  Thus
  \begin{equation}
    \label{eq:X+Y:ExPsi=bbetaPsi}
    \E_X[\Psi] =\bbeta{Y}(\Psi)
    \bigl(
    1- z\bbeta{Y}(\Psi)X
    \bigr)^{-1}
    = 
    \bigl(
    1/\bbeta{Y}(\Psi)-z X
    \bigr)^{-1}.
  \end{equation}
  The evaluation of $\bbeta{Y}(\Psi)$ will be discussed in Example~\ref{eg:Intro_1.6}.
\end{example}

In order to obtain the usual formulation of free additive subordination one has to consider the
resolvent $R=(z-(X+Y))^{-1}$ instead of $\Psi$.

\begin{example}
  \label{ex:RX+Y}
  Let $R=R(z)=(z-X-Y)^{-1}$, then $\rdelta_XR=R\otimes XR$ and the solution of the recurrence for the conditional expectation  
  $$
  \E_X[R] = \bbeta{Y}(R) + \bbeta{Y}(R)X \E_X[R]
  $$
  is immediately obtained as
  \begin{equation*}
    \E_X[R] = (\bbeta{Y}(R)^{-1}-X)^{-1}
    .
  \end{equation*}
  Now $\bbeta{Y}(R)$ is an analytic function of $z$ (via $R=R(z)$) 
  and its reciprocal $\omega(z)=1/\bbeta{Y}(R)$ 
  is the subordination function from \eqref{eq:subordination}. For its further evaluation see Example~\ref{ex:Rx+Ydelta}.
\end{example}

In fact the recurrence \eqref{eq:condexprec}  allows for finding a system of
equations for the conditional expectation of resolvents of arbitrary polynomials. 
\begin{proposition}\label{thm:CondExpHX}
  Let $P\in \IC\langle X,Y\rangle $ and suppose
        $\rdelta_X P = \sum_{i=1}^m u_i\otimes  X v_i$,
        i.e., $\partial_X P = \sum_{i=1}^m u_i\otimes  v_i$,
        and let $\Psi=(1-zP)^{-1}$. Then we have
	\begin{align}
          \label{eq:EXPsi=sumbPsiuiXEXviPsi}
          \IE_X[\Psi] &= \bbeta{Y}(\Psi) + z \sum_{i=1}^m \bbeta{Y}(\Psi {u}_i)
		X \IE_X[{v}_i\Psi]
	\end{align}
	where the conditional expectations on the right hand side satisfy the linear system of equations
	
	\begin{equation}
          \label{eq:linearizationForCondExp}
		\begin{bmatrix}
			\IE_X[v_1\Psi]\\
			\IE_X[v_2\Psi]\\
			\vdots\\
			\IE_X[v_m\Psi]
		\end{bmatrix}
		=
		\begin{bmatrix}
			\bbeta{Y}[v_1\Psi]\\
			\bbeta{Y}[v_2\Psi]\\
			\vdots\\
			\bbeta{Y}[v_m\Psi]
		\end{bmatrix}
		+H X
		\begin{bmatrix}
			\IE_X[v_1\Psi]\\
			\IE_X[v_2\Psi]\\
			\vdots\\
			\IE_X[v_m\Psi]
		\end{bmatrix}
	\end{equation}
	where
	$H_{ij} = \bbeta{Y}({v}_i\Psi {u}_j)z+
	\bbeta{Y}({u}_{ij})$.
\end{proposition}

\begin{remark}
	\begin{enumerate}[(1)]
         \item
          Let us label each appearance of $X$ in $P(X,Y)$ by consecutive labels $X_1,X_2,\ldots$ and similarly for $Y$ we label them as $Y_1,Y_2,\ldots$. For example for $XY+YX$ we write $X_1Y_1+X_2Y_2$. Then one can see that the entry $H_{ij}$ is exactly the functional $\bbeta{Y}$
          evaluated in all
	possible sub-words of $\Psi$ which occur between $X_i$ and $X_j$. More precisely
	if we consider $(1-z P)^{-1}$ then if $X_i$ and $X_j$ are in different
	monomials of $P$ then these monomials are of the form
	$u_iX_iv_i$ and $u_jX_jv_j$ and all subwords
	of $(1-zP)^{-1}$ which are between $X_i$ and $X_j$ are exactly of the
	form $v_i(1-zP)^{-1}u_j$ if $X_i$ and $X_j$ are in
	the same monomial then this monomial is of the form
	$u_iX_i u_{ij} X_j v_j$. Thus in this case we get the
	additional term $u_{ij}$ to which we apply $\bbeta{Y}$.
	
	Observe this very matrix $H_Y$ (with slightly
	different powers of $z$) appeared in \cite[Theorem~6.1]{FMNS2} 
	in the computation of the anti-commutator.
        However there the goal was to determine the
	distribution, not the conditional expectation.
	
       \item
        At this point it is not clear how to determine the matrix $H$ above.
        For this it is necessary to evaluate the functional $\bbeta{}$ and this
        will be done  in Section~\ref{sec:calcbbeta}.
	
       \item
        Equation \eqref{eq:linearizationForCondExp} is linear and can
        immediately be turned into a  \emph{linearization}
        (in the sense of \cite{HeltonMaiSpeicher:ApplicationsOfRealizations})
        of the rational
        function $\IE_X[v_i\Psi]$.
        We will develop this systematically in Section~\ref{sec:LinearizationAndCondExp}.
        
\end{enumerate}
\end{remark}
\begin{proof}[Proof of Proposition \ref{thm:CondExpHX}]
	For the recurrence 
	\eqref{eq:condexprec}
	we can use either $\rDelta_X$ or $\rdelta_X$.
	The latter has the advantage of being a derivation.
	In any case
	\begin{align}
          \IE_X[\Psi]
          &= \bbeta{Y}(\Psi)
            + z (\bbeta{Y}\otimes \IE_X[ (\Psi\otimes 1)(\rdelta_X T)(1\otimes\Psi)]
            \label{eq:EXpsidelta}
          \\
          &= \bbeta{Y}(\Psi)
            + z (\bbeta{Y}\otimes \IE_X[ (\Psi\otimes 1)(\rDelta_X T)(1\otimes\Psi)]
\end{align}
	because the extra terms  in  \eqref{eq:EXpsidelta} are annihilated by $\bbeta{Y}$.
	
	Write
	\begin{align*}
		\rdelta_X T &= \sum_{i=1}^m u_i\otimes  X v_i\\
\end{align*}
	i.e., $\partial_X T = \sum_{i=1}^m u_i\otimes v_i$.
Then recurrence \eqref{eq:condexprec} reads
	\begin{align*}
		\IE_X[\Psi] &= \bbeta{Y}(\Psi) + z \sum_{i=1}^m \bbeta{Y}(\Psi {u}_i)
		X \IE_X[{v}_i\Psi]\\
	\end{align*}
	and we can finally establish a system of equations for the conditional expectations
	appearing on the RHS of \eqref{eq:EXPsi=sumbPsiuiXEXviPsi}. Again we employ
	\eqref{eq:condexprec} and get
	\begin{align*}
		\IE_X[{v}_i\Psi] &= \bbeta{Y}({v}_i\Psi) +\sum_j\bbeta{Y}({u}_{ij})X \IE_X[v_j\Psi] + z \sum_{i=1}^m \bbeta{Y}(v_i\Psi {u}_j)X \IE_X[v_j\Psi].
	\end{align*}
	
	Define a matrix $H$ as
\begin{equation*}
		H_{ij} = \bbeta{Y}({v}_i\Psi {u}_j)z+
		\bbeta{Y}({u}_{ij}).     
\end{equation*}
	It is straightforward to see that identity
        \eqref{eq:linearizationForCondExp} holds with $H_Y$ defined as above.
\end{proof}

\clearpage{}
\section{Calculus for the block cumulant functional $\bbeta{}$}
\label{sec:calcbbeta}

The recurrence from Theorem~\ref{thm:condexprec} opens a door to the study of
conditional expectations and in the case of resolvents the resulting linear system
can be solved explicitly to obtain explicit formulas in terms of the functional
$\bbeta{}$. 
Its  evaluation  is the subject of the present section.
However at the time of this writing we  lack sufficient understanding of
Boolean cumulants required to  handle this problem
in the general setting  of Assumption~\ref{ass:augmented}.
Therefore in the remainder of this paper we will mostly work in the formal setting of polynomial
algebras from Remark~\ref{rem:augmented-polynomialalg},
which allows the reduction to univariate Boolean cumulants which are well understood.

\begin{assumption}
  \label{ass:ABpolynomialalg}
  We will work in a formal ncps,
  i.e., the free associative algebra $\IC\langle\alg{X}\rangle$ over some
  alphabet $\alg{X}$ and some distribution
  $\mu:\IC\langle\alg{X}\rangle\to\IC$.
  As usual the augmentation map $\epsilon:\IC \langle \alg{X} \rangle\to \IC$ is the
  constant coefficient map   $\epsilon(P)=P(0)$.

  More specifically, we will work in one of the following settings:
  \begin{enumerate}[(i)]
   \item 
  \label{ass:ABpolynomialalg:it:1}
    $\alg{M}=\IC\langle \alg{X}\cup\alg{Y}\rangle$
    where $\alg{X}=\{X_1,X_2,\dots,X_m\}$ and $\alg{Y}=\{Y_1,Y_2,\dots, Y_n\}$
    are free from each other with respect to $\mu$ and so are
    $\alg{A}=\IC\langle \alg{X}\rangle$ and $\alg{B}=\IC\langle \alg{Y}\rangle$.
    In this case we will denote the corresponding conditional expectations,
    derivations, and functionals
    by $\IE_{\alg{X}}$, $\rDelta_{\alg{X}},\rdelta_{\alg{X}}$, $\bbeta{\alg{X}}$ etc.
    \item 
     $\alg{M}=\left(\IC\langle\uXn \rangle,\mu\right)$
     and 
     \begin{align*}
\alg{A}&=\IC\langle X_i \mid i\in I\rangle
       &
         \alg{B}&=\IC\langle X_j \mid j \in J\rangle.  
     \end{align*}
     where $I\dot{\cup}J=[n]$ is a partition and we  assume in addition that \emph{all} variables $X_1,X_2,\ldots,X_n$ are free from
     each other with respect to the functional $\mu$.
    In this case we will denote the corresponding conditional expectations,
    derivations, and functionals by
     In this case we will denote conditional expectations 
     by $\IE_I$, $\rDelta_I$, $\bbeta{I}$ etc.

     The typical case is 
     \begin{align*}
\alg{A}&=\IC\langle \uXk\rangle
       &
         \alg{B}&=\IC\langle X_{k+1},X_{k+2},\ldots,X_n\rangle.  
     \end{align*}
  \end{enumerate}
\end{assumption}

The following functionals provide the key to evaluate the functionals
$\bbeta{}$. They operate by fully splitting monomials into their factors.
\begin{definition}\label{def:fbeta}
  On the formal ncps $(\IC\langle\alg{X}\rangle,\mu)$
  we define the \emph{fully factored Boolean cumulant} $\fbeta{}$  by prescribing their values on monomials
  \begin{align*}
    \fbeta{}(1) &=1,\\
    \fbeta{}(X_{i_1}X_{i_2}\dotsm X_{i_k})
                &= \beta_{k}(X_{i_1},X_{i_2},\dots,X_{i_k})
                  .
  \end{align*}
  We can decompose the functional $\fbeta{}$ along
  free subsets of variables   similar to the functional $\bbeta{}$:

  For a subset  $\alg{Y}\subseteq\alg{X}$ of the variables
  we define linear functionals $\fbeta{\alg{Y}}$ on monomials as follows:
  \begin{align*}
    \fbeta{\alg{Y}}(1) &=1,\\
    \fbeta{\alg{Y}}(X_{i_1}X_{i_2}\dotsm X_{i_k})
    &=  \begin{cases}
      \beta_{k}(X_{i_1},X_{i_2},\dots,X_{i_k}) & \text{if $X_{i_1},X_{i_k}\in\alg{Y}$}\\
        0 & \text{otherwise}
    \end{cases}
  \end{align*}
  Then $\fbeta{}=\fbeta{\alg{X}}$ and if a set of variables $\alg{Y}$ is the disjoint union of mutually
  free (with respect to $\mu$) subsets $\alg{Y}_j$ then thanks to property
  $(CAC)$
  we have
  $$
  \fbeta{\alg{Y}}=\epsilon+\sum (\fbeta{\alg{Y}_j}-\epsilon)
  .
  $$
\end{definition}

Next we state Lemma \ref{lem:BoolSplitGroup} in terms of the functional
$\fbeta{}$,
which will turn out to be useful in the proof of the next theorem.

\begin{lemma}
  Assume the setting from Assumption~\ref{ass:ABpolynomialalg}
  \eqref{ass:ABpolynomialalg:it:1},
  i.e.,   $\alg{M}=\IC\langle \alg{X}\cup\alg{Y}\rangle$ with $\alg{X}$ and
  $\alg{Y}$ mutually free.
  Then for any letters $X_{i_j}\in\alg{X}$ and any elements $V_j\in\IC\langle\alg{Y}\rangle$
  we have
  \begin{align*}
    \fbeta{\alg{X}}(X_{i_0}V_1X_{i_1}V_2\dotsm V_kX_{i_k})
    =\beta_{2k+1}(X_{i_0},V_1,X_{i_1},V_2,\ldots,V_{k},X_{i_k})
    .
  \end{align*}
  We emphasize that it is not required that $V_j$ have vanishing constant term.
\end{lemma}

For the next theorem we enhance our toolbox with yet another derivation, which
implements a kind of elementary unshuffle coproduct.
\begin{definition}
  For a subset of variables $\alg{Y}=\{Y_1,Y_2,\dots,Y_k\}\subseteq\alg{X}$
  define the first order unshuffle operator
  \begin{align*}
  L_{\alg{Y}}^*\odot \partial_{\alg{Y}} = \sum L_{Y_j}^*\otimes
  \partial_{Y_j}:\IC\langle\alg{X}\rangle &\to
                                            \IC\langle\alg{Y}\rangle\otimes\IC\langle\alg{X}\rangle\otimes\IC\langle\alg{X}\rangle\\
    W&\mapsto \sum Y_j\otimes \partial_{Y_j}W
       \intertext{Its iterations are defined on the left leg}
  (L_{\alg{Y}}^*\odot\partial_{\alg{Y}})^2 W
       &= \sum_{i,j} Y_iY_j\otimes (\partial_{Y_i}\otimes\id)\partial_{Y_j} W\\
  (L_{\alg{Y}}^*\odot\partial_{\alg{Y}})^p W
       &= \sum_{i_1,i_2,\dots,i_p} Y_{i_1}Y_{i_2}\dotsm Y_{i_p}\otimes
         (\partial_{Y_{i_1}}\otimes\id^{\otimes p-1})(\partial_{Y_{i_2}}\otimes\id^{\otimes p-2})\dotsm\partial_{Y_{i_p}} W
  \end{align*}

\end{definition}
In the theorem below we summarize the relations between the functionals
$\bbeta{}$ and $\fbeta{}$ 
which will allow to evaluate these functionals on some non-commutative power series.

\begin{theorem}
  \label{thm:intro_bbeta_sbeta}
  Assume the setting from Assumption~\ref{ass:ABpolynomialalg}   \eqref{ass:ABpolynomialalg:it:1},
  i.e.,   $\alg{M}=\IC\langle \alg{X}\cup\alg{Y}\rangle$ with $\alg{X}$ and
  $\alg{Y}$ mutually free.
  \begin{enumerate}[(i)]
   \item 
  For any element $P\in \IC\langle \alg{X}\cup\alg{Y} \rangle$ we have
  \begin{equation}\label{eq:intro_betas_prod_as_entries}
    \begin{aligned}
      \bbeta{\alg{X}}(P)
      &= \epsilon(P) + (\fbeta{\alg{X}}\otimes\bbeta{\alg{X}})(\ldelta_{\alg{X}} P)
      \\
      &= \epsilon(P) +
      (\bbeta{\alg{X}}\otimes\fbeta{\alg{X}})(\rdelta_{\alg{X}} P)
      .
    \end{aligned}
  \end{equation}
\item
  For any element $P\in \IC\langle \alg{X}\cup\alg{Y} \rangle$ we have  
\begin{equation}
  \label{eq:IntroVNRP_for_betas}
  \begin{aligned}
    \fbeta{\alg{X}}(P)
    &=\epsilon(P)+\sum_{k=1}^\infty
    \fbeta{}\otimes\left[\epsilon\otimes
      \left(\bbeta{\alg{Y}}\right)^{\otimes(k-1)}\otimes
      \epsilon\right]\left( (L_{\alg{X}}^*\odot\partial_{\alg{X}})^k(P)\right)
    \\
    &= \epsilon(P)+\sum_{k=1}^\infty
    \sum_{i_1,i_2,\dots,i_k} \beta_k(X_{i_1},X_{i_2},\dots, X_{i_k})
    \left[\epsilon\otimes
      \left(\bbeta{\alg{Y}}\right)^{\otimes(k-1)}\otimes
      \epsilon\right]\left(
          \partial_{X_{i_1}}\partial_{X_{i_2}}\dotsm\partial_{X_{i_k}} P
        \right)
  \end{aligned}
\end{equation}
In particular, when $\alg{X}=\{X\}$ consists of a single variable, then
\begin{equation}
  \label{eq:IntroVNRP_for_betas:uni}
  \fbeta{X}(P)
  = \epsilon(P)+\sum_{k=1}^\infty
    \beta_k(X)
    \left[\epsilon\otimes
      \left(\bbeta{\alg{Y}}\right)^{\otimes(k-1)}\otimes
      \epsilon\right]\left(
          \partial_X^k P
        \right)
\end{equation}
  \end{enumerate}
\end{theorem}

\begin{remark}
  \begin{enumerate}[(1)]
\item
    It is straightforward to extend by linearity the functionals
    $\bbeta{\alg{X}}$ and $\fbeta{\alg{X}}$
    and Theorem~\ref{thm:intro_bbeta_sbeta} to formal power series.
    One the one hand, for elements of the algebra     $\IC\llangle \alg{X}\rrangle$
    of formal power series     in
    non-commuting variables $\alg{X}=\{X_1,X_2,\ldots,X_n\}$,
    provided the resulting series converges;
    on the other hand, to the algebra $\IC\langle\alg{X}\rangle((z))$
    of formal power series with non-commuting coefficients.
   \item
    Formula~\eqref{eq:IntroVNRP_for_betas} can also be stated in terms of the
    half-shuffle coproduct $\Delta_{\prec}$ of \cite
    {EbrahimiFardPatras:2015:halfshuffles},
    but this will be dealt with elsewhere.

\end{enumerate}
\end{remark}
\begin{proof}[Proof of Theorem \ref{thm:intro_bbeta_sbeta}]
  \begin{enumerate}[(i)]
   \item 
    Essentially formula \eqref{eq:intro_betas_prod_as_entries}
    is a reformulation of Corollary~\ref{cor:RecursiveBoolProd} in terms of
    the functionals and derivations which we introduced above,
    after observing that the extra terms
    vanish as a consequence of Property $(CAC)$.
	
    We will only prove the first of the two identities
    \eqref{eq:intro_betas_prod_as_entries} as the proof of second equation
    is essentially the same.
    By linearity it suffices to consider monomials  $w\in\IC\langle\alg{X}\cup\alg{Y}\rangle$
    and in fact only words $w$ of type $\alg{X}$--$\alg{X}$,
    any word of different type being annihilated on both sides of the equation by
    definition of $\bbeta{\alg{X}}$.
    So suppose $w=u_1v_1u_2\dotsm v_{k} u_{k+1}$ is a free factorization
    into words $u_i\in\alg{X}^+$ and $v_i\in\alg{Y}^+$, then
    $\bbeta{\alg{X}}(w)=\beta_{2k+1}(u_1,v_1,u_2,\ldots,v_k,u_{k+1})$.
    On the other hand
  \begin{equation}
    \label{eq:deltaAW}
    \ldelta_{\alg{X}}
    w=\sum_{i=1}^{k+1}u_1v_1u_2\cdots v_{i-1}\cdot\ldelta_{\alg{X}}(u_i)\cdot
    v_iu_{i+1}\cdots u_{m+1}.
  \end{equation}
  This produces the same splittings of the factors from $\alg{X}$ in the
  $\alg{X}$--$\alg{Y}$ factorization
  as in \eqref{eq:RecursiveBoolProd}, except for two kinds configurations:
  \begin{enumerate}[(1)]
   \item
    Formula \eqref{eq:RecursiveBoolProd} contains splittings of factors of
    $w$ coming from both  $\alg{X}^+$ and $\alg{Y}^+$.
    The latter are absent from \eqref{eq:deltaAW} because whenever
    a factor  $v_i=v_i' v_i''$ is split, at least one of the factors of the arising term
    $\fbeta{\alg{X}}(u_1v_1u_2\cdots v_i')\bbeta{\alg{X}}(v_i''u_{i+1}\dotsm u_{m+1})$
    vanishes by definition of $\fbeta{\alg{X}}$ and $\bbeta{\alg{X}}$.    
   \item
    In the derivative $\ldelta_{\alg{X}}w$ there are extra terms of 
    type
    $u_1v_1u_2\dotsm u_k\otimes v_ku_{k+1}v_{k+1}\dotsm u_{m+1}$ which do not
    appear in Corollary~\ref{eq:RecursiveBoolProd},
    but the application of
    $\fbeta{\alg{X}}\otimes\bbeta{\alg{X}}$
    creates the factor $\bbeta{\alg{X}}\left(v_k\cdots u_{m+1}\right)=0$.
  \end{enumerate}
 \item 	
  For the proof of \eqref{eq:IntroVNRP_for_betas}  
  first observe that both sides of this equation are $1$ for the empty word
  and  that  both sides are zero unless $w$ is of type $\alg{X}$--$\alg{X}$.
  Thus fix a monomial $w\in(\alg{X}\cup\alg{Y})^+$  of the latter type.
  We write it as $w=X_{i_1}v_1X_{i_1}v_2\dotsm v_k X_{i_{k+1}}$
  with $v_j\in\alg{Y}^*$, i.e., we allow some $v_j=1$.
  However the latter will be considered later  and we assume first that $v_j\neq 1$ for
  $i=1,2,\ldots,k$. Comparing \eqref{eq:IntroVNRP_for_betas} with formula
  \eqref{bocu2} observe that according to \eqref{bocu2} we have a sum over all
  choices of variables from $\alg{X}$, with the restriction that both $X_{i_1}$  and
  $X_{i_{k+1}}$ are always selected. In formula \eqref{eq:IntroVNRP_for_betas}, all
  choices of $p$ $X$'s are given by the $p$-th derivative, the variables
  annihilated by the free derivative are put to the outer block,
  moreover application of
  $\epsilon\otimes\left(\bbeta{\alg{Y}}\right)^{\otimes (k-1)}\otimes
  \epsilon$ evaluates to zero unless both $X_{i_1}$ and $X_{i_k}$ are
  annihilated by the derivative. Further from \eqref{eq:IntroVNRP_for_betas} we
  get a product of cumulants of type
  $\beta_{r}(v_{p},X_{i_{p+1}},\ldots,X_{i_q},v_{q})$,
  which are extracted between the two $X$'s moved to the outer block and
  this is
  exactly equal to $\bbeta{\alg{Y}}(v_{p}X_{i_{p+1}}\cdots X_{i_q}v_{q})$.  
	
  It remains to consider the case where some $v_q=1$.
  Assume that it comes from a pocket determined by $X_{i_p}$ and $X_{i_r}$
  which are both chosen to the outer block, i.e., $p\leq q<r$.
  The existence of such $p$ and $q$ is guaranteed as we always choose the first
  and the last $X$ to the outer block.
  We will consider the contribution of this pocket in both formulas
  $\eqref{bocu2}$ and \eqref{eq:IntroVNRP_for_betas}.
  There are three possible cases depending on how $v_i$ is placed between $j$-th and $k$-th $X_1$:
  \begin{itemize}
   \item $p=q$ and $r=p+1$ then the pocket created by these two elements contains only $v_q$ and formula \eqref{bocu2} gives $\beta_{1}(v_q)=1$ while from \eqref{eq:IntroVNRP_for_betas} we get $\bbeta{\alg{Y}}(1)=1$.
   \item $p=q$ and $r>p+1$, then  the pocket contains
    $v_pX_{i_{p+1}}\cdots v_{r-1}$ and \eqref{bocu2} gives precisely
    $\beta_{2(k-i)-1}(v_p,X_{i_{p+1}},v_{p+1},X_{i_{p+2}},\dots,v_{r-1})=0$
    because the first argument is $1$ and   \eqref{eq:boolcumunit:1} applies.
    On the other hand \eqref{eq:IntroVNRP_for_betas} we get
    $\bbeta{\alg{Y}}(v_pX_{i_{p+1}}v_{p+1}X_{i_{p+1}}\dotsm v_{r-1})
     =\bbeta{\alg{Y}}(X_{i_{p+1}}v_{p+1}X_{i_{p+1}}\dotsm  v_{r-1})=0$
     by definition of the functional $\bbeta{\alg{Y}}$. 
     The case $p<q=r-1$ is treated similarly.

    \item
     Having eliminated all such units, it remains to consider  the case
     $p<q<r-1$.
     By \eqref{bocu2} the contribution of the pocket containing $v_q$ equals
     \begin{multline*}
       \beta_{2(k-i)-1}(v_p,X_{i_{p+1}},v_{p+1},X_{i_{p+2}},\dots,X_{i_q},1,X_{i_{q+1}},\dots,X_{i_{r-1}},v_{r-1})
       \\
       =
       \beta_{2(k-i)-1}(v_p,X_{i_{p+1}},v_{p+1},X_{i_{p+2}},\dots,X_{i_q},X_{i_{q+1}},\dots,X_{i_{r-1}},v_{r-1})
       \\
       =
       \beta_{2(k-i)-1}(v_p,X_{i_{p+1}},v_{p+1},X_{i_{p+2}},\dots,X_{i_q}X_{i_{q+1}},\dots,X_{i_{r-1}},v_{r-1})
     \end{multline*}
     because of   \eqref{eq:boolcumunit:3} and  Lemma~\ref{lem:BoolSplitGroup}.
     
     Assuming that $v_p,v_{p+1},\dots,v_{r-1}\neq 1$, then $\bbeta{\alg{Y}}(v_p
     X_{i_{p+1}}v_{p+1}\dotsm X_{i_q} v_q X_{i_{q+1}}\dotsm
     X_{i_{r-1}}v_{r-1})$ evaluates to the same value. If there are $v_l=1$ for
     $p<l<r-1$ and $l\neq q$ we repeat the same argument.
     Observe that with the previous steps we have already made sure that $v_{p},v_{r-1}\neq 1$.
   \end{itemize}
  \end{enumerate}
\end{proof}

Let us illustrate this with some examples.

\begin{example}
A direct calculation using VNRP
shows that
\begin{align*}
  \fbeta{X}(XYX)=\beta_3(X,Y,X)=\beta_2(X)\beta_1(Y).
\end{align*}

Let us calculate the derivatives of $\partial_X(XYX)=1\otimes YX+XY\otimes 1$ of
course both corresponding terms vanish after application of $\bbeta{Y}$ and
thus $n=1$ does not contribute (it never will, effectively the sum starts from
n=2). The second derivative gives $D_2(XYX)=1\otimes Y\otimes 1$, and hence
from the formula in the theorem we also get 
\begin{align*}
	\fbeta{X}(XYX)=\beta_3(X,Y,X)=\beta_2(X)\beta_1(Y).
\end{align*}
\end{example}

\begin{example}[The additive subordination function]
  \label{ex:Rx+Ydelta}
  In Example~\ref{ex:RX+Y} we concluded that the additive subordination function for
  $R=(z-X-Y)^{-1}$ is $1/\bbeta{Y}(R)$.
  Now $\ldelta_YR =RY\otimes R$ and
  the recurrence \eqref{eq:intro_betas_prod_as_entries} yields
  \begin{align*}
    \bbeta{Y}(R)
    &= \epsilon(R) + (\fbeta{Y}\otimes\bbeta{Y})(\ldelta_YR)\\
    &= \frac{1}{z} + \fbeta{Y}(RY)\bbeta{Y}(R)
  \end{align*}
  and thus
  the subordination function is
  $$
  \omega(z) = 1/\bbeta{Y}(R) = z-z\fbeta{Y}(RY)
  $$
  cf.~\cite[Corollary~3.7]{LehnerSzpojan}.
\end{example}

\begin{example}\label{eg:Intro_1.6}
  In anticipation of matrix valued formulas arising in Section~\ref{sec:LinearizationAndCondExp}
  we continue Example~\ref{example:condExpAdditive}
  and consider the simple example where $n=2$
  and the power series $\Psi=(1-z(X+Y))^{-1}=\sum_{n=0}^{\infty}\left(z(X+Y)\right)^n\in\IC\langle X,Y\rangle((z))$.
  Since we work only with two variables we will write $\bbeta{X},\fbeta{X}$ etc. which should not lead to any confusion.

Clearly we have
\begin{align*}
\rdelta_{X}(\Psi)=1+z\Psi\otimes X \Psi,\qquad\rdelta_{Y}(\Psi)=1+z\Psi\otimes Y \Psi.
\end{align*}
Thus from \eqref{eq:intro_betas_prod_as_entries} we obtain
\begin{align*}
\bbeta{X}(\Psi)=1+z\bbeta{X}(\Psi)\fbeta{X}(X\Psi),\qquad\bbeta{Y}(\Psi)=1+z\bbeta{Y}(\Psi)\fbeta{Y}(Y\Psi).
\end{align*}
Hence we obtain 
\begin{align*}
  \bbeta{X}(\Psi)=\left(1-z\fbeta{X}(X\Psi)\right)^{-1},\qquad\bbeta{Y}(\Psi)=\left(1-z\fbeta{Y}(Y\Psi)\right)^{-1}.
\end{align*}
Comparing with \eqref{eq:X+Y:ExPsi=bbetaPsi} we conclude that
$$
\E_X[\Psi] = (1-z\fbeta{Y}(Y\Psi)-zX)^{-1}
.
$$
Moreover observe that 
\begin{align*}
  \partial_{X}^n(X\Psi)=z^{n-1}1\otimes \Psi^{\otimes n}+z^{n}X\Psi\otimes \Psi^{\otimes n},\qquad\partial_{Y}^n(Y\Psi)=z^{n-1}1\otimes \Psi^{\otimes n}+z^{n}Y\Psi\otimes \Psi^{\otimes n}.
\end{align*}
Thus equation \eqref{eq:IntroVNRP_for_betas:uni} gives
\begin{align*}
  \fbeta{X}(X\Psi)&=\sum_{n=1}^{\infty} \beta_n(X)\bbeta{Y}(\Psi)^{n-1}z^{n-1}=\widetilde{\eta}_{X}(z\bbeta{Y}(\Psi)),\\
  \fbeta{Y}(Y\Psi)&=\sum_{n=1}^{\infty} \beta_n(Y)\bbeta{X}(\Psi)^{n-1}z^{n-1}=\widetilde{\eta}_{Y}(z\bbeta{X}(\Psi)).
\end{align*}
Finally we obtain the following system of equations
\begin{align*}
  \fbeta{X}(X\Psi)=\widetilde{\eta}_{X}\left(z\left(1-z\fbeta{Y}(Y\Psi)\right)^{-1}\right),\quad
  \fbeta{Y}(Y\Psi)=\widetilde{\eta}_{Y}\left(z\left(1-z\fbeta{X}(X\Psi)\right)^{-1}\right).
\end{align*}
We shall see below that this system of equations has   unique  power
series solutions $\fbeta{X}(X\Psi)$ and $\fbeta{Y}(Y\Psi)$ analytic at 0
and in fact yields the fixed point equation for subordination function for free
additive convolution.
Thus this system determines the function needed in Example \ref{example:condExpAdditive}.
\end{example}

\clearpage{}
\section{Linearization and conditional expectations}
\label{sec:LinearizationAndCondExp}
The procedure presented  in Proposition~\ref{thm:CondExpHX}
can be systematized using  linearizations of resolvents from the very beginning.
For the reader not familiar with  rational series and linearizations
the basic facts are collected in  Appendix~\ref{app:linearization}.
Here we will adapt and amplify all previously defined operations to the level of matrices
and then lift  Example~\ref{eg:Intro_1.6} to the matrix-valued
setting.

\subsection{Amplifications of expectations and cumulants}
Most concepts considered in the present paper can be generalized to the
operator valued case.
Rather than do this in the general case we 
restrict the discussion to amplifications to tensor products.
In fact the matrix valued case
would suffice for  later applications to linearizations,
however the proofs are conceptually simpler in the language of
tensor products.

\begin{notation}
  Let $(\alg{A},\varphi)$ be an ncps and $\alg{C}$ a unital algebra.
  We shall consider $\alg{C}\otimes\alg{A}$ as a $\alg{C}$-bimodule with
  action
  $c_1\cdot (c\otimes a)\cdot c_2 = c_1cc_2\otimes a$.

  In the case where $\alg{C}=M_N(\IC)$ is a matrix algebra
  we will denote the elements by $\sum C_iu_i$
  where for a matrix $C\in M_N(\IC)$ and $u\in\alg{A}$ we denote by $Cu=uC\in M_N(\alg{A})$
  the matrix with entries
  $$
  (Cu)_{ij} = c_{ij}u
  .
  $$
\end{notation}

The following lemma is easily verified on elementary tensors.
\begin{lemma}
  \label{lem:amplifications}
  Let $(\alg{A},\varphi)$ be a ncps and $\alg{C}$ a unital algebra.
  and denote by
  $\varphi_{\alg{C}}=\id_\alg{C}\otimes\varphi:\alg{C}\otimes\alg{A}\to\alg{C}$
  the amplification of $\varphi$.
  \begin{enumerate}[(i)]
   \item 
    $\varphi_{\alg{C}}$     is a $\alg{C}$-bimodule map:
    \begin{equation*}
    \varphi_{\alg{C}}(c_1\cdot u\cdot c_2) = c_1 \varphi_{\alg{C}}(u) c_2
    \end{equation*}
    for all $c_1,c_2\in\alg{C}$ and $u\in\alg{C}\otimes\alg{A}$.
   \item 
  Let $\alg{B}\subseteq\alg{A}$ be a subalgebra and $\E_{\alg{B}}:\alg{A}\to\alg{B}$ be a
  conditional expectation for $\varphi$.
  Then its amplification
  $\id_{\alg{C}}\otimes\E_{\alg{B}}:\alg{C}\otimes\alg{A}\to\alg{C}\otimes\alg{B}$
  is a $\alg{C}\otimes\alg{B}$-bimodule map
  and 
  moreover it is a conditional expectation for the amplification 
  $\varphi_{\alg{C}}$
  in the sense that
  \begin{equation*}
    \varphi_{\alg{C}}((\id_{\alg{C}}\otimes\E_{\alg{B}})[u]v) = \varphi_{\alg{C}}(uv)    
  \end{equation*}
  for any $u\in\alg{C}\otimes\alg{A}$ and any $v\in\alg{C}\otimes\alg{B}$.
 \item
  The $\alg{C}$-valued Boolean cumulants defined by the amplification of the
  recurrence 
  \eqref{eq:recurrenceboolcum}
  $$
  \varphi_\alg{C}(u_1u_2\dotsm u_n)
  = \sum_{k=1}^n \beta_k^{\alg{C}}(u_1,u_2,\dots,u_k)\,\varphi_\alg{C}(u_{k+1}u_{k+2}\dotsm u_n)
  $$
  are given by
  $$
  \beta_n^{\alg{C}}(c_1\otimes a_1,c_2\otimes a_2,\dots,c_n\otimes a_n)
  =  c_1c_2\dotsm c_n \,\beta_n( a_1, a_2,\dots, a_n)
  .
  $$
  \end{enumerate}
\end{lemma}

In the case where $\alg{C}=M_N(\IC)$ is a matrix algebra  we can apply the usual identification of
$M_n(\IC)\otimes\alg{A}$ with $M_N(\alg{A})$ and  reformulate the lemma in
terms of matrix operations as follows.
\begin{corollary}
  \label{cor:matrixamplifications}
  Let $(\alg{A},\varphi)$  be a ncps.
  \begin{enumerate}[(i)]
   \item 
    The amplification     $\varphi^{(N)}:M_N(\alg{A})\to M_N(\IC)$ to the matrix algebra
    is the     entry-wise application
    $\varphi^{(N)}([a_{ij}]) = [\varphi(a_{ij})]$.
    These maps
    form a family of matrix bimodule maps:
    $$
    \varphi^{(k)}(U\cdot A\cdot V) = U \varphi^{(N)}(A) V
    $$
    holds for any matrix $A\in M_N(\alg{A})$ and any scalar matrices
    $U\in M_{k\times N}(\IC)$ and $V\in M_{N\times k}(\IC)$,
    $k\in\IN$.
   \item 
    Let $\alg{B}\subseteq\alg{A}$ be a subalgebra and $\E_{\alg{B}}:\alg{A}\to\alg{B}$ be a
    conditional expectation for $\varphi$.
  Then the entry-wise application maps
  $\E_{\alg{B}}^{(N)}
  \bigl[
  [a_{ij}]_{ij}
  \bigr]  = 
  \bigl[\E_{\alg{B}}[a_{ij}]
  \bigr]_{ij}$
  form a family of  $M_N(\alg{B})$-bimodule maps:
  \begin{equation}
    \label{eq:EBnbimodule}
    \E_{\alg{B}}^{(k)}[U\cdot A\cdot V] = U \cdot \E_{\alg{B}}^{(N)}[A] \cdot V
  \end{equation}
  holds for any matrix $A\in M_N(\alg{A})$ and any scalar matrices
  $U\in M_{k\times N}(\IC)$ and $V\in M_{N\times k}(\IC)$.
  Moreover it is a conditional expectation for the map 
  $\varphi^{(N)}$
  in the sense that
  $\varphi^{(N)}(\E_{\alg{B}}^{(N)}(A)B) = \varphi^{(N)}(AB)$
  for any $A\in M_N(\alg{A})$ and any $B\in M_N(\alg{B})$.
 \item
  The entries of the $M_N(\IC)$-valued Boolean cumulants defined by the amplification of the
  recurrence 
  \eqref{eq:recurrenceboolcum}
  $$
  \varphi^{(N)}(A_1A_2\dotsm A_n)
  = \sum_{k=1}^n \beta_k^{(N)}(A_1,A_2,\dots,A_k)\,\varphi^{(N)}(A_{k+1}A_{k+2}\dotsm A_n)
  $$
  are given by
  $$
\beta_n^{(N)}(A_1,A_2,\dots,A_n)_{ij}
  =  \sum_{i_1,i_2,\dots,i_{n-1}}\beta_n( a_{ii_1}^{(1)}, a_{i_1i_2}^{(2)},\dots, a^{(n)}_{i_{n-1}j})
  .
  $$
  \end{enumerate}
\end{corollary}

\begin{notation}
  The corresponding matrix valued versions of the functionals $\bbeta{}$ and
  $\fbeta{}$
  are defined analogously as
\begin{align*}
    \bbetan{\alg{A}}{N}(A)_{ij}
    &= \bbeta{\alg{A}}(a_{ij})
    &
      \fbetan{\alg{A}}{N}(A)_{ij}
    &= \fbeta{\alg{A}}(a_{ij})
      ,
  \end{align*}
  whenever these are defined (Assumption~\ref{ass:augmented}, resp.\ Assumption~\ref{ass:ABpolynomialalg}).
\end{notation}

\subsection{Amplifications of derivations}

\begin{lemma}
  \label{lem:amplifyderiv}
  Let $\alg{A}$ and $\alg{C}$ be algebras, $D:\alg{A}\to\mathfrak{M}$ be a
  derivation into an $\alg{A}$-bimodule $\mathfrak{M}$.
Then $D^{(\alg{C})}=\id_{\alg{C}}\otimes D:\alg{C}\otimes\alg{A}\to \alg{C}\otimes\mathfrak{M}$
  is a derivation, where $\alg{C}\otimes\mathfrak{M}$ is an
  $\alg{C}\otimes\alg{A}$-module
  with action
  $$
  (c_1\otimes a_1) \cdot (c\otimes\mathfrak{m}) \cdot (c_2\otimes a_2) =  c_1cc_2\otimes(a_1\cdot \mathfrak{m}\cdot a_2)
  $$
\end{lemma}
\begin{proof}
  It suffices to verify the Leibniz rule    \eqref{eq:leibniz}
  for   $D^{(\alg{C})}$ on elementary tensors:
  \begin{align*}
    D^{(\alg{C})}
    \bigl(
    (c_1\otimes a_1)(c_2 \otimes a_2)
    \bigr)
    &= c_1c_2 \otimes D(a_1a_2)\\
    &= c_1c_2 \otimes (D(a_1)\cdot a_2) + c_1c_2 \otimes (a_1\cdot D(a_2))\\
    &= D^{(\alg{C})}(c_1 \otimes a_1)\cdot(c_2\otimes a_2) +  (c_1\otimes
      a_1)\cdot D^{(\alg{C})}(c_2\otimes a_2) 
  \end{align*}
\end{proof}

This is true in particular for $\alg{C}=M_N(\IC)$
(cf.~\cite[Section~3]{MaiSpeicher:1805.04150}).

\begin{corollary}
  Let $\mathfrak{M}$ be an $\alg{A}$-bimodule and  $D:\alg{A}\to\mathfrak{M}$ be  a derivation.
  Then $M_N(\mathfrak{M})$ is an $M_N(\alg{A})$-bimodule with action
  $$
  [a'_{ij}]_{ij}\cdot [\mathfrak{m}_{ij}]_{ij}\cdot [a''_{ij}]_{ij}
  = 
  \Bigl[
  \sum_{k,l} a'_{ik}\cdot \mathfrak{m}_{kl}\cdot a''_{lj}
  \Bigr]_{ij}
  $$
  and the matrix amplification $D^{(N)}:M_N(\alg{A}) \to M_N(\mathfrak{M})$,
  i.e., entry-wise application
  $$
  D^{(N)}([a_{ij}]) = [D(a_{ij})]
  $$
  satisfies the Leibniz rule \eqref{eq:leibniz}.
\end{corollary}
\begin{example}
  Let $D:\alg{A}\to \mathfrak{M}$ be a derivation where
  $\mathfrak{M}=\alg{A}\otimes\alg{A}$
  is the bimodule with action
  $$
  a'\cdot (a_1\otimes a_2)\cdot a'' = a'a_1\otimes a_2a'' = (a'\otimes
  1)(a_1\otimes a_2)(1\otimes a'')
  .
  $$
  Then using Notation~\ref{notation:tensor}
  the Leibniz rule for the product of two matrices $A_1,A_2\in M_N(\alg{A})$
  reads
$$
  D^{(N)}(A_1A_2) = D^{(N)}(A_1)(1_{\alg{A}}\odot A_2) + (A_1\odot
  1_{\alg{A}})D^{(N)}(A_2)
 \in  M_N(\alg{A}\otimes\alg{A}),
  $$
i.e.,
  $$
  D^{(N)}(A_1A_2)_{ij}
  = \sum_k D(a^{(1)}_{ik})(1_{\alg{A}}\otimes  a^{(2)}_{kj})
  + (a^{(1)}_{ik}\otimes 1_{\alg{A}}) D(a^2_{kj})
  .
  $$
  In particular,
  for a resolvent $\bPsi = (I_N-zA)^{-1}$ the derivation results in
  $$
  D^{(N)}\bPsi = z(\bPsi\odot 1)D^{(N)}(A)(1\odot\bPsi)
  $$
  and in the case of an elementary tensor $A=C\otimes a$ with $D(a) = \sum u_i\otimes v_i$,
  $$
  D^{(N)}\bPsi = z\sum \bPsi u_i C\odot v_i\bPsi.
  $$
  
\end{example}

\subsection{Computing conditional expectations via linearizations}
In the following we will omit the
superscripts from   $\IE_{\alg{B}}^{(N)}$  etc.\ and write  $\IE_{\alg{B}}$ etc.\ whenever the context is unambiguous.
Now the bimodule property \eqref{eq:EBnbimodule} allows us to rewrite
the conditional expectation of a linearization
$$
(1-z^mP)^{-1} = u^t(I-zL)^{-1}v
$$
as
\begin{equation}
  \label{eq:EB1-zP=utEB1-zLv}
\E_{\alg{B}}
  \bigl[
  (1-z^mP)^{-1}
  \bigr]
= u^t\E_{\alg{B}}
\bigl[(I-zL)^{-1}\bigr]v
\end{equation}
and the problem is reduced to the computation of conditional expectations of matrix pencils.
We shall see that this accomplished by repeating the computation from
Example~\ref{ex:1:additiveconvolution} with matrix coefficients.

It is immediate to verify the
amplification of the recurrence
\eqref{eq:condexprec}, namely
$$
(\id_{\alg{C}}\otimes\E_{\alg{B}})[c\otimes P]
= (\id_{\alg{C}}\otimes\bbeta{\alg{A}})(P)
+ (\id_{\alg{C}}\otimes\bbeta{\alg{A}}\otimes\E_{\alg{B}})[\id_{\alg{C}}\otimes\rDelta_{\alg{B}}(P)]
$$
etc.
The matrix analog of
Theorem~\ref{thm:condexprec}
in terms of the amplifications of $\E_{\alg{B}}$, $\bbeta{\alg{A}}$,
$\rDelta_{\alg{B}}$ and $\rnabla_{\alg{B}}$
from Lemmas~\ref{lem:amplifications} and   \ref{lem:amplifyderiv}
reads as follows.
\begin{proposition}
  Let $M\in M_N(\alg{M})$ and $\alg{A},\alg{B}\subseteq\alg{M}$ as in
  Assumption~\ref{ass:augmented}.
  Then
  \begin{align*}
    \E_{\alg{B}}^{(N)}[M] &= \bbetan{\alg{A}}{N}(M)
                      + (\bbeta{\alg{A}}\otimes \E_{\alg{B}})^{(N)} [\rDelta_{\alg{B}}^{(N)}(M)]\\
                    &= \bbetan{\alg{A}}{N}(M)
                      + (\bbeta{\alg{A}}\otimes\E_{\alg{B}})^{(N)}
                      [\rnabla_{\alg{B}}^{(N)}(M)]\\
                     &= \bbetan{\alg{A}}{N}(M)
                      + (\E_{\alg{B}}\otimes\bbeta{\alg{A}})^{(N)} [\lDelta_{\alg{B}}^{(N)}(M)]
  \end{align*}
In particular, for a linear matrix pencil $L=\sum C_i'a_i+C_j''b_j\in M_N(\alg{M})$ 
  the conditional expectation of the resolvent
  $\bPsi = 
  \bigl(
  I_N-z(\sum C_i'a_i+ \sum C''_j b_j)
  \bigr)^{-1}$ is
  \begin{equation}
    \label{eq:EnPsi=1-zHAsumCbHA}
  \E_{\alg{B}}^{(N)}[\bPsi] = 
  \bigl(
  I_N- z H_{\alg{A}}\sum C_j'' b_j
  \bigr)^{-1} H_{\alg{A}}
  ,
  \end{equation}
  where
  $H_{\alg{A}}= \bbetan{\alg{A}}{N}(\bPsi)\in M_N(\IC)$.
\end{proposition}

Formula \eqref{eq:EnPsi=1-zHAsumCbHA} is a matrix valued version of
Example~\ref{example:condExpAdditive}
and the proof runs along the same lines.

Next we switch to Assumption~\ref{ass:ABpolynomialalg} and
lift
    \eqref{eq:EY=bbetardeltaY}
and
Theorem~\ref{thm:intro_bbeta_sbeta} 
to the matrix level.

\begin{proposition}
  Let $\alg{M}= M_N(\IC\langle\alg{X}\cup\alg{Y}\rangle)$ and $\alg{X},\alg{Y}$
  free with respect to $\mu:\alg{M}\to\IC$ as in
  Assumption~\ref{ass:ABpolynomialalg}.
  Then
  \begin{align*}
    \E_{\alg{Y}}^{(N)}[M] &= \bbetan{\alg{X}}{N}(M)
                      + (\bbeta{\alg{X}}\otimes \E_{\alg{Y}})^{(N)} [\rdelta_{\!\alg{Y}}^{(N)}(M)]
  \end{align*} 
\end{proposition}

Next, identity \eqref{eq:intro_betas_prod_as_entries} trivially lifts to
tensor products
\begin{equation*}
  \begin{aligned}
    (\id_{\alg{C}}\otimes\bbeta{\alg{X}})(c\otimes P) 
    &= c \epsilon(P) + c\cdot(\fbeta{\alg{X}}\otimes\bbeta{\alg{X}})(\ldelta_{\alg{X}} P)
    \\
    &= c \epsilon(P) +
    c\cdot(\bbeta{\alg{X}}\otimes\fbeta{\alg{X}})(\rdelta_{\alg{X}} P)
  \end{aligned}
\end{equation*}
which immediately translates to the case of matrices as follows. 
\begin{proposition}
  Let $M\in M_N(\alg{M})$, then with Assumption~\ref{ass:ABpolynomialalg} we have
\begin{equation*}
  \begin{aligned}
    \bbetan{\alg{X}}{N}(M) 
    &=  \epsilon^{(N)}(M) + (\fbeta{\alg{X}}\otimes\bbeta{\alg{X}})^{(N)}(\ldelta_{\alg{X}}^{(N)} M)
    \\
    &=  \epsilon^{(N)}(M) + (\bbeta{\alg{X}}\otimes\fbeta{\alg{X}})^{(N)}(\rdelta_{\alg{X}}^{(N)} M)
  \end{aligned}
\end{equation*}
  In particular, for a linear matrix pencil $L=\sum C_i'X_i+C_j''Y_j\in M_N(\alg{M})$ 
  the block cumulant functional at 
  $\bPsi = 
  \left(
    I_N-z(\sum C_i'X_i+ \sum C''_j Y_j)
  \right)^{-1}$ is
  \begin{equation}
    \label{eq:bbetaPsi=1-zsumfbetaPsiaiCi}
    \begin{aligned}
    \bbetan{\alg{X}}{N}(\bPsi)
    &=
    \bigl(
    I_N- z \sum\fbetan{\alg{X}}{N}(\bPsi X_i)C_i
    \bigr)^{-1}
    = 
    \bigl(
    I_N- z \fbetan{\alg{X}}{N}(\bPsi L)
    \bigr)^{-1}
    \\
    &=
    \bigl(
    I_N- z \sum C_i\fbetan{\alg{X}}{N}(X_i\bPsi)
    \bigr)^{-1}
    = 
    \bigl(
    I_N- z \fbetan{\alg{X}}{N}(L\bPsi)
    \bigr)^{-1}
    .
    \end{aligned}
  \end{equation}
\end{proposition}
Thus in order to conclude the computation
of the conditional expectation   \eqref{eq:EB1-zP=utEB1-zLv}
it remains to evaluate
$\fbetan{\alg{X}}{N}(\bPsi X_i)$ for all $i$.
This can be done if we assume that all variables are free
with the help of   \eqref{eq:IntroVNRP_for_betas:uni} which we now lift to the matrix level.
\begin{proposition}
    \label{prop:fbetaX1NM}
  Let $M\in M_N(\IC\langle X_1,X_2,\dots,X_n\rangle)$, then
  \begin{equation}
    \label{eq:fbetaX1NM}
  \fbetan{X_1}{N}(M)=\epsilon^{(N)}(M)+\sum_{k=1}^\infty\beta_k(X_1)\left[\epsilon\otimes\left(\bbeta{\alg{B}}\right)^{\otimes
      (k-1)}\otimes \epsilon\right]^{(N)}\left(
    \left.
      \partial_{X_1}^k\!\!\!
    \right.^{(N)}(M)\right).
\end{equation}
  
\end{proposition}

We can now subsume the essence of the previous calculations as follows.
Theorem~\ref{thm:1.1} follows by evaluating formula \eqref{eq:EIps=utsumCiXiCjFj}
in a specific algebra $\alg{A}$.
\begin{theorem}
  \label{thm:EIPsimain}
  Let $\mu:\IC\langle X_1,X_2,\dots,X_n\rangle\to\IC$ be a distribution such
  that  $X_1,X_2,\dots,X_n$ are free.
  In the following for an index set $I\subseteq \{1,2,\dots,n\}$
  the conditional expectation onto  the
  subalgebra generated by the subset $\{X_i\}_{ i\in I}$
  will be denoted   by $\E_I$.
  Suppose that the resolvent of a given polynomial
  $P=P(X_1,X_2,\dots,X_n)\in \IC\langle X_1,X_2,\dots,X_n\rangle$
  of degree $m$ has linearization
  \begin{equation}
    \label{eq:psi=utLv}
    \Psi=(1-z^mP)^{-1}= u^t\left(I_N-z(C_1X_1+C_2X_2+\dots+ C_n X_n)\right)^{-1}v
  \end{equation}
  with $C_1,C_2,\ldots,C_n\in M_N(\IC)$.
  
  Then the conditional expectations of $\Psi$ are
  \begin{equation}
    \label{eq:EIps=utsumCiXiCjFj}
  \begin{aligned}
    \E_I
    [(1-z^m P)^{-1}]
    &=u^t \Bigl(I_N-zH_{J}
      \Bigl(\sum_{i\in I}
      C_iX_i\Bigr)\Bigr)^{-1}H_{J}v\\
    &=u^t \Bigl(I_N-z\Bigl(\sum_{i\in I} C_iX_i - \sum_{j\in J} C_jF_j\Bigr)\Bigr)^{-1}v
  \end{aligned}
  \end{equation}
  where $J=[n]\setminus I$ is the complement and $H_J=
  \bigl(
  I_N-z\sum_{j\in J} C_j F_j
  \bigr)^{-1}$
  and the matrices $F_i=\fbetan{X_i}{N}(X_i\bPsi)\in M_N(\IC)$, $i=1,2,\dots,n$
  are  the unique solution  of the system of matrix equations
  \begin{equation}
    \label{eq:Fi=etaXi}
    \begin{aligned}
      F_i
      &=\widetilde{\eta}_{X_i}
      \Bigl(
      z\bigl(
      I_N-z\sum_{j\neq i} C_jF_j
      \bigr)^{-1} C_i
      \Bigr)
      &i=1,2,\dots,n
  \end{aligned}
  \end{equation}
  which is analytic at $z=0$.
  In particular,
  \begin{equation}
    \label{eq:phi=uMv}
    \varphi((1-z^m P)^{-1})
    =u^t \Bigl(I_N-z\Bigl(\sum_{j=1}^n C_jF_j\Bigr)\Bigr)^{-1}v
  \end{equation}
\end{theorem}

\begin{proof}
  Let
  $$
  \bPsi =\left(I_N-z(C_1X_1+C_2X_2+\dots+ C_n X_n)\right)^{-1}
  $$
  be the matrix of the linearization  $\Psi=u^t\bPsi v$ from \eqref{eq:psi=utLv}.
  In the following we abbreviate the functionals
  \begin{align*}
    \bbeta{I} &= \sum_{i\in I}\bbeta{X_i} &
    \fbeta{I} &= \sum_{i\in I}\fbeta{X_i}
                .
  \end{align*}
  Then from     \eqref{eq:EnPsi=1-zHAsumCbHA} with
  $\alg{B}$ the algebra generated by $\{X_i\}_{i\in I}$
  and $\alg{A}$ the algebra   generated by $\{X_j\}_{j\in J}$ we obtain
  $$
  \E_I[\Psi]
  = u^t
  \bigl(
     I_N-zH_J\sum_{i\in I}C_iX_i
  \bigr)^{-1}H_Jv
  $$
  where
  \begin{equation*}
  H_J =\bbetan{J}{N}(\bPsi)
= \bigl(
     I_N-z\sum_{j\in J} C_j\fbetan{J}{N}(X_j\bPsi)
     \bigr)^{-1}
     .
  \end{equation*}
   by \eqref{eq:bbetaPsi=1-zsumfbetaPsiaiCi}.

  Next we apply Proposition~\ref{prop:fbetaX1NM}
  in order to establish equations for $F_i = \fbetan{X_i}{N}(X_i\bPsi)$,
  similar to Example~\ref{eg:Intro_1.6}.
  First observe that
  $$
  \partial_{X_i}\bPsi = z (\bPsi\odot 1)C_i(1\odot\bPsi) = z\bPsi C_i\odot \bPsi
  $$
  thus
  $$
  \partial_{X_i}X_i\bPsi = 1\odot\bPsi + z X_i\bPsi C_i\odot \bPsi
  $$
  and applying iterated derivatives to the left leg
  $$
  \partial_{X_i}^kX_i\bPsi = z^{k-1}1\odot(\bPsi C_i)^{\odot k-1}\odot\bPsi
  + z^k  X_i(\bPsi C_i)^{\odot k} \odot\bPsi
  $$
  and \eqref{eq:fbetaX1NM} becomes
  \begin{align*}
    \fbetan{X_i}{N}(X_i\bPsi)
    &=
      \sum_{k=1}^\infty
      \beta_k(X_i)
      z^{k-1}
      \bbetan{[n]\setminus i}{N}(\bPsi C_i)^{k-1}
    \\
    &= \tilde{\eta}_{X_i}(zH_{[n]\setminus i}C_i)
      .
  \end{align*}
  Uniqueness of the  matrices $F_i$ follows from the iteration in
  Lemma~\ref{lem:iteration} below.
\end{proof}

We have seen that the matrices
$
F_i(z) = \fbetan{X_i}{N}(X_i\bPsi(z))
$
satisfy the fixed point equation     \eqref{eq:Fi=etaXi}
and it remains to show that the latter has a unique  solution analytic at 0.
To this end we expand $F_i(z)$ into a power series
\begin{equation}
  \label{eq:Fiz=series}
F_i(z) = \fbetan{X_i}{N}(X_i\bPsi(z))
= \sum_{k=0}^\infty
\fbetan{X_i}{N}(X_i L^k)z^{km}
\end{equation}
and show that iterating the fixed point equation produces a series
whose coefficients converge to those of the series   \eqref{eq:Fiz=series}.

\begin{lemma}\label{lem:iteration}
   Start with constant matrices
\begin{align*}
F_i^{(0)}(z) &= \beta_1(X_i)I_N
                   & i = 1,2,\dots,n
\end{align*}
and iterate
\begin{align*}
  F_i^{(r+1)}(z) &= \tilde{\eta}_i
                   \Bigl(
                   z
                   \bigl(
                   I_N-z\sum_{j\ne i}C_jF_j^{(r)}(z)
                   \bigr)^{-1}C_i
                   \Bigr)
                   & i = 1,2,\dots,n
\end{align*}
Then for all $r\in\IN_0$ we have
  \begin{align*}
    F_i^{(r)}(z) - F_i(z) &= \Ord(z^{r+1})
                   & i = 1,2,\dots,n
  \end{align*}
\end{lemma}
\begin{proof}
  We proceed by induction.
  First observe that if $F_i^{(r)}(z) - F_i(z) = \Ord(z^{r+1})$ for $i=1,2,\dots,n$
  then
  $$
  \bigl(
  I_N-z\sum_{j\ne i}C_jF_j^{(r)}(z)
  \bigr)^{-1}C_i
  -
  \bigl(
  I_N-z\sum_{j\ne i}C_jF_j(z)
  \bigr)^{-1}C_i
  =\Ord(z^{r+2})
  $$
  as well
  and thus
  \begin{align*}
  F_i^{(r+1)}(z)-  F_i(z)
    &=
      \begin{multlined}[t]
        \sum_{k=0}^{r+1}
         \beta_{k+1}(X_i)z^k
         \Bigl(
          \bigl(
           I_N-z\sum_{j\ne i}C_jF_j^{(r)}(z)
          \bigr)^{-1}C_i
         \Bigr)^k
        +
        \Ord(z^{r+2})\\
      - 
      \sum_{k=0}^{r+1}
       \beta_{k+1}(X_i)z^k
      \Bigl(
      \bigl(
      I_N-z\sum_{j\ne i}C_jF_j(z)
      \bigr)^{-1}C_i
      \Bigr)^k
      + \Ord(z^{r+2})      
    \end{multlined}
    \\
    &= \Ord(z^{r+2})
  \end{align*}
\end{proof}

  We conclude this section with a few hints for the practical
  solution of the system \eqref{eq:Fi=etaXi}.

\begin{remark}[Compression]
  \label{rem:projectionsQi}
  The main computational problem  with the solution of the system
  \eqref{eq:Fi=etaXi} is caused by the large number of variables
  featuring in the matrices $F_i$.
  One way to address this issue arises from the observation
  that the solutions $F_i$ enter
  the final expression \eqref{eq:EIps=utsumCiXiCjFj} for the conditional
  expectation only in the products $C_iF_i$ and we can take
  advantage of the fact that the  coefficient matrices $C_i$
  usually are sparse.
  Indeed let $P_i$ be the projection onto $\ker C_i$ and
  $Q_i=I-P_i$, i.e., $C_iP_i=0$ and $C_iQ_i=C_i$.
  Then in the expression \eqref{eq:EIps=utsumCiXiCjFj} we
  can replace the matrices  $F_i$ with the compressed matrices 
  $\tilde{F}_i = Q_iF_iQ_i$ which satisfy the
  slightly modified system
  \begin{equation}
    \label{eq:Fit=QietaXi}
    \begin{aligned}
      \tilde{F}_i
      &=Q_i\widetilde{\eta}_{X_i}
      \Bigl(
      z\Bigl(
      I_N-z\sum_{j\neq i} C_j\tilde{F}_j
      \Bigr)^{-1} C_i
      \Bigr)
      &i=1,2,\dots,n
      .
  \end{aligned}
  \end{equation}
  In this case the starting point for the fixed point
  iteration \eqref{eq:Fit=QietaXi} is
  $\tilde{F}_i^{(0)} = \beta_1(X_i)Q_i$.
  We thus obtain power series expansions
  of the matrices $F_i$ which in turn
  can be used to further reduce the  number of variables.
\end{remark}

\begin{remark}[Algebraic functions]
  \label{rem:algebraic}
  In the case when 
  the involved shifted Boolean transforms are algebraic, e.g.,
  in the case of semicircular, arcsine or free Poisson distribution,
  it may be helpful to set up implicit matrix equations
  and manipulate them directly.
  To this end  split the system   \eqref{eq:thm:1.1:eta}
  into two parts 
  \begin{align}
    H_i &= 
          \Bigl(
          I_N-z\sum_{j\ne i} C_jF_j
          \Bigr)^{-1} & i &= 1,2,\dots,n\\
    \label{eq:Fi=etaHiCi}
    F_i
      &=\widetilde{\eta}_{X_i}(zH_i C_i)
      &i&=1,2,\dots,n
  \end{align}
  and turn the algebraic equation for $\tilde{\eta}$ into
  a algebraic equation in the noncommuting variables $F_i$ and $H_i$.
  To illustrate this, consider a standard semicircular family $X_i$.
  Then the individual shifted Boolean transforms of the $X_i$
  satisfy the quadratic equation
  $$
  z\tilde{\eta}(z)^2-\tilde{\eta}(z)+z = 0
  $$
  and consequently equation     \eqref{eq:Fi=etaHiCi} can be rewritten
  as  \begin{equation}
    \label{eq:zHiFi^2-Fi+zHiCi}
  zH_iC_iF_i^2-F_i+zH_iC_i = 0
  .
  \end{equation}
  Now let
  $$
  M(z)=\varphi
  \Bigl(\Bigl(I_N-z\sum_j C_jX_j\Bigr)^{-1}\Bigr)
  =
  \Bigl(I_N-z\sum_j C_jF_j\Bigr)^{-1}
  $$
  be the matrix-valued moment generating function.
  Multiplying     \eqref{eq:zHiFi^2-Fi+zHiCi} with
  $H_i^{-1} = M^{-1}+zC_iF_i$ from the left leads to
  a cancellation and we obtain the identity
  $F_i=zMC_i$.  Substituting this identity into the equations we can eliminate
  the $F_i$ altogether 
  and  the result is the well-known fixed point equation
  $$
  M(z) = \Bigl(I_N-z^2\sum_j C_jMC_j\Bigr)^{-1}
  $$
  for the matrix-valued moment generating function of the matrix-valued
  semicircular element $\sum C_i\otimes X_i$,
  cf.~\cite {Lehner:computing:1999,HeltonRashidiSpeicher:operator:2007}.

\end{remark}

\clearpage{}
\section{Examples}
\label{sec:examples}

The following examples illustrate how the different methods developed in this paper can be used to compute distributions and conditional expectations of polynomials and rational functions in free random variables explicitly,
or, in more complex cases, to obtain irreducible algebraic equations for the functions involved.

Roughly, the procedure for the more complicated examples below runs as follows:

\begin{enumerate}[1.]
 \item Given a non-commutative polynomial $P$ of degree $d$ compute a
  linearization
  of $\Psi(z^2)= (1-z^dP)^{-1}$ using the suffix basis or similar as described in
  Algorithm~\ref{algo:linearization}.
 \item Set up the matrix equations
  \eqref{eq:thm:1.1:eta}, see Remark~\ref{rem:algebraic}.
 \item
  Reduce the number of unknowns using the following methods, if applicable:
  \begin{enumerate}[(i)]
   \item If the coefficient matrices $C_i$ are sparse,
    use  the compressed version \eqref{eq:Fit=QietaXi}
    as described in Remark~\ref{rem:projectionsQi}.
   \item  Compute a few iterations of the fixed point equations
    \eqref{eq:thm:1.1:eta} (resp.~\eqref{eq:Fit=QietaXi})
    to obtain power series expansions of the matrices $F_i(z)$.
    An inspection  of these matrices quickly reveals zero entries
    and identities among the nonvanishing ones,
    which allow the immediate elimination of some variables.
    Since the solution is unique, the worst consequence of
    imposing more conditions is a system without solutions.
   \item 
    Alternatively (and more rigorously), compute the entire inverse of the linearization
    matrix using Schur
    complements to determine the entries which are annihilated by $\fbeta{X_i}$ due to property $(CAC)$
    or other symmetries.
   \item If the initial distribution have algebraic Cauchy transforms,
    it can be possible to eliminate entire matrices by noncommutative
    algebraic manipulations as described in Remark~\ref{rem:algebraic}.
  \end{enumerate}
 \item Solve the resulting sparse system using resultants, Gröbner bases etc.
\end{enumerate}

In the following, whenever the dimension $N$ of the linearization is clear
from context it will be omitted from the decorations of the matricial versions
of the mappings  $\E$,  $\bbeta{}$, $\fbeta{}$ etc. 
The  calculations were carried out with the help of \texttt{FriCAS}
\cite{fricas}.
Most solutions in the examples below satisfy algebraic equations of degree at most
four and could in principle be expressed explicitly using Cardano's
formulas, however these being  rather voluminous and not giving any useful insights,
we refrain from doing so and  keep the calculations implicit.

\subsection{The product of free random variables}
We start with the example
\begin{equation*}
  T=XY  
\end{equation*}
which can be solved directly
by working with the resolvent $\Psi(z) = (1-zXY)^{-1}$,
without invoking linearizations.

The derivatives are $\rdelta_X\Psi=z\Psi\otimes XY\Psi$ and
$\rdelta_Y\Psi=z\Psi X\otimes Y\Psi$, respectively.
First notice that from the resolvent identities
\begin{equation}
	\label{eq:XY:resolventidentityPsi}
	\Psi = 1+zXY\Psi  = 1+z\Psi XY 
\end{equation}
we immediately infer that
\begin{equation}
	\label{eq:XY:bbetaPsi=1}
	\bbeta{X}(\Psi) =   \bbeta{Y}(\Psi) = 1
\end{equation}
and thus
\begin{equation}
	\label{eq:EX(Psi)-XY}
	\IE_X[\Psi] = \bbeta{Y}(\Psi) + z \bbeta{Y}\otimes \IE_X[\Psi\otimes XY\Psi] =
	1+ zX \IE_X[Y\Psi]
	;
\end{equation}
further
$$
\IE_X[Y\Psi] = \bbeta{Y}(Y\Psi)(1+zX\IE_X[Y\Psi])
$$
and thus
\begin{equation}
	\label{eq:XY:EX(YPsi)}
	\IE_X[Y\Psi] = \bbeta{Y}(Y\Psi)
	\bigl(
	1-z\bbeta{Y}(Y\Psi)X
	\bigr)^{-1}
	.
\end{equation}
Putting together  \eqref{eq:EX(Psi)-XY} and 
\eqref{eq:XY:EX(YPsi)} we find
$$
\IE_X[\Psi] = 
\bigl(
1-z\bbeta{Y}(Y\Psi)X
\bigr)^{-1}
$$
and the subordination equation
\begin{equation*}
	M(z) = M_X(z\bbeta{Y}(Y\Psi))
	.
\end{equation*}
We have thus identified the first subordination function
$$
\omega_1(z) = z\bbeta{Y}(Y\Psi)
.
$$
The conditional expectation with respect to $Y$ is simpler to obtain,
but the result is analogous:
$$
\IE_Y[\Psi] = 1 + z\bbeta{X}(\Psi X)Y\IE_Y[\Psi]
$$
and thus
\begin{align*}
	\IE_Y[\Psi] &= 
	\bigl(
	1 - z\bbeta{X}(\Psi X)Y
	\bigr)^{-1}\\
	M(z) &= M_Y(z\bbeta{X}(\Psi X))
	.
\end{align*}
The second subordination function is
$$
\omega_2(z) = z\bbeta{X}(\Psi X)
.
$$
To find relations between these functions we will resort
to Theorem~\ref{thm:intro_bbeta_sbeta}.
It is immediate to see combinatorially that
$\bbeta{Y}(Y\Psi)=\fbeta{Y}(Y\Psi)$
and
$\bbeta{X}(\Psi X)=\fbeta{X}(\Psi X)$.
This can also be  verified from the recurrence from Proposition~\ref{thm:intro_bbeta_sbeta};
indeed,
\begin{align*}
  \bbeta{Y}(Y\Psi)
  	&= \fbeta{Y}\otimes\bbeta{Y}(\ldelta_Y(Y\otimes \Psi))\\
	&= \fbeta{Y}\otimes\bbeta{Y}(Y\otimes \Psi + zY\Psi XY\otimes\Psi) \\
	&= \fbeta{Y}(Y(1+z\Psi XY))\bbeta{Y}(\Psi)= \fbeta{Y}(Y\Psi)
\end{align*}
because of the identities   \eqref{eq:XY:resolventidentityPsi} and
\eqref{eq:XY:bbetaPsi=1}.
Similarly
\begin{align*}
	\bbeta{X}(\Psi X)
	&= \fbeta{X}\otimes\bbeta{X}( z\Psi X \otimes Y\Psi X+ \Psi X\otimes 1) = \fbeta{X}(\Psi X).
\end{align*}
Thus we have
$$
\omega_1(z) = z\fbeta{Y}(Y\Psi)
\qquad
\omega_2(z) = z\fbeta{X}(\Psi X)
.
$$

Next we will use the expansion \eqref{eq:IntroVNRP_for_betas}.
To this end observe that $\partial_Y(\Psi)=z\Psi X\otimes \Psi$ and thus
\begin{align*}
	\partial_Y(Y\Psi)&=1\otimes \Psi+zY\Psi X\otimes\Psi\\
	\partial_Y^2(Y\Psi)&=z1\otimes\Psi X\otimes\Psi+z^2Y\Psi X\otimes\Psi X\otimes \Psi\\
	\ldots&\\
	\partial_Y^n(Y\Psi)&=z^{n-1}1\otimes(\Psi X)^{\otimes(n-1)}\otimes\Psi+z^nY\Psi X\otimes(\Psi X)^{\otimes(n-1)}\otimes \Psi.\\
\end{align*}
Hence from \eqref{eq:IntroVNRP_for_betas} we obtain 
\[\fbeta{Y}(Y\Psi)=\widetilde{\eta}_Y(z\bbeta{X}(\Psi X))=\widetilde{\eta}_Y(z\fbeta{X}(\Psi X)).\]
Similarly
\[\fbeta{X}(\Psi X)=\widetilde{\eta}_X(z\fbeta{Y}(Y \Psi)),\]
which gives a system of fixed point equations for the subordination functions.

\subsection{Free commutators and anti-commutators}
\label{ssec:commutator}
We show that for symmetric free random variables
$X,Y$ the conditional expectations of
$
(1-z^2 P(X,Y))^{-1}
$
onto one of the
variables coincide for
the  anti-commutator.
\begin{equation*}
T_1=  XY+YX
\end{equation*}
and the
commutator
\begin{equation*}
T_i=  i(XY-YX)
  .
\end{equation*}
In fact, since $XY$ and $YX$ form an $R$-diagonal pair
\cite[Theorem~15.17]{NicaSpeicherLect},
their joint $*$-distribution is invariant under rotation
and thus the distribution of
$$
T_\alpha=\alpha XY+\overline{\alpha}YX
$$
does not depend on
$\alpha$ as long as $\abs\alpha = 1$.
Indeed we will show that this holds for the conditional expectations as well:
\begin{proposition}
  Assume that $X,Y$ are free and have symmetric distribution. If $\abs{\alpha}=1$ then the conditional expectations
  $E_X[(1-zT_\alpha)^{-1}]$ and $E_Y[(1-zT_\alpha)^{-1}]$
  do not  depend on $\alpha$ and are equal to
  $E_X[(1-zT_1)^{-1}]$ resp.~ $E_Y[(1-zT_1)^{-1}]$;
  in particular the conditional  expectations of the commutator and the
  anticommutator coincide.
\end{proposition}

To this end we apply
Algorithm~\ref{algo:linearization}
to obtain the following  linearization matrix for  $T_\alpha$ :
\[
L = L_\alpha=
\begin{bmatrix}
0 & {  \overline{\alpha}Y} & {  \alpha X} \\
{  X} & 0 & 0 \\
{  Y} & 0 & 0 
\end{bmatrix}
= C_X X+C_Y Y
\]
where
\begin{align*}
C_X&=
\begin{bmatrix}
0 & 0 & \alpha \\
1 & 0 & 0 \\
0 & 0 & 0 
\end{bmatrix}
&
C_Y&=
\begin{bmatrix}
  0 & \overline{\alpha}& 0 \\
0 & 0 & 0 \\
1 & 0 & 0 
\end{bmatrix}
        .
\end{align*}
Then
\[
  \Psi(z) = 
  (1-z^2T_\alpha)^{-1}=
  e_1^t
 (I-zL)^{-1} 
e_1
\]
is the inverse Schur complement of the matrix $I-zL$ with
respect to the upper left corner and its full inverse is
$$
(I-zL)^{-1} =
\begin{bmatrix}
  \Psi & \overline{\alpha}z\Psi Y & \alpha z\Psi X\\
  zX\Psi& 1+\overline{\alpha}z^2X\Psi Y & \alpha z^2 X\Psi X\\
  zY\Psi& \overline{\alpha}z^2Y\Psi Y& 1+\alpha z^2Y\Psi X
\end{bmatrix}
.
$$
Now the expansion of $\Psi$ involves only words in $X$ and $Y$ of even  length 
and since  both variables are assumed to have symmetric distribution, Boolean
cumulants of odd order vanish,
leaving only few nonzero entries in the matrices $F_X$ and $F_Y$:
\begin{align*}
  F_X(z) &= \fbeta{X}(X(I-zL)^{-1})
           =
\begin{bmatrix}
  0 & 0 & \alpha z\fbeta{X}(X\Psi X)\\
  z\fbeta{X}(X^2\Psi)& 0 & 0\\
  z\fbeta{X}(XY\Psi)& 0& 0
\end{bmatrix}
  \\
  F_Y(z) &= \fbeta{Y}(Y(I-zL)^{-1})
           =
\begin{bmatrix}
  0 & \overline{\alpha}z\fbeta{Y}(Y\Psi Y) & 0\\
  z\fbeta{Y}(YX\Psi)& 0 & 0\\
  z\fbeta{Y}(Y^2\Psi)& 0 & 0
\end{bmatrix}
.
\end{align*}
After compression with the cokernel projections
$$
Q_X =
\begin{bmatrix}
1 & 0 & 0 \\ 
0 & 0 & 0 \\ 
0 & 0 & 1 
\end{bmatrix}
\qquad
Q_Y =
\begin{bmatrix}
1 & 0 & 0 \\ 
0 & 1 & 0 \\ 
0 & 0 & 0 
\end{bmatrix}
$$
we find a total of two nonzero entries in each matrix:
\begin{align}
  \label{eq:anti:FX}
  \tilde{F}_X(z)
  &=
    \begin{bmatrix}
      0 & 0 & f_{{x,  1  3}}(z) \\
      0 & 0 & 0 \\
      f_{{x,  3  1}}(z) & 0 & 0 
    \end{bmatrix}
  \\
  \label{eq:anti:FY}  
  \tilde{F}_Y(z)
  &=
    \begin{bmatrix}
      0 & f_{{y,  1  2}}(z) & 0 \\
      f_{{y,  2 1}}(z) & 0 & 0 \\
      0 & 0 & 0 
    \end{bmatrix}
              .
\end{align}
From the main  Theorem~\ref{thm:EIPsimain} in
its compressed version from  Remark~\ref{rem:projectionsQi}
we infer the following  system of equations:
\begin{align}\label{eq:main_system}
  \begin{cases}
    &H_X=(I-zC_Y \tilde{F}_Y)^{-1}\\
    &H_Y=(I-zC_X \tilde{F}_X)^{-1}\\
    &\tilde{F}_X=Q_X  \widetilde{\eta}_X(z H_X C_X)\\
    &\tilde{F}_Y=Q_Y  \widetilde{\eta}_Y(z H_Y C_Y)
  \end{cases}
\end{align}
Substituting $\tilde{F}_X$ and $\tilde{F}_Y$ from equations
\eqref{eq:anti:FX} and \eqref{eq:anti:FY}
into the first half of the system
\eqref{eq:main_system}
we obtain
\begin{align}
	H_X&=
	\begin{bmatrix}\label{eq:HX_ac}
          {\frac{1}{1-\overline{\alpha}zf_{{y,  2  1}} }} & 0 & 0 \\
          0 & 1 & 0 \\
          0 & {zf_{{y,  1  2}}} & 1 
	\end{bmatrix}
	&H_Y&=
	\begin{bmatrix}
          {\frac{1}{1-\alpha zf_{{x,  3  1}}}} & 0 & 0 \\
          0 & 1 & {zf_{{x,  1 3}}} \\
          0 & 0 & 1 
	\end{bmatrix}
  .
\end{align}
Now
\begin{align*}
	H_X C_X&=
	\begin{bmatrix}
          0 & 0 & \frac{\alpha}{1-\overline{\alpha}zf_{y,  2  1}} \\
          1 & 0 & 0 \\
          zf_{y,  1  2} & 0 & 0 
	\end{bmatrix}
  &
    \left(H_X C_X\right)^2
    &=
	\begin{bmatrix}
\frac{\alpha  z f_{y,  12}  }{1-  \overline\alpha  z f_{y,  21} } & 0 & 0 \\
0 & 0 & \frac{\alpha}{1- \overline\alpha  z f_{{y,  2,  1}}  } \\
0 & 0 & \frac{\alpha  z f_{{y,  1,  2}}  }{1-\overline\alpha  z f_{y,  21} } 
      \end{bmatrix}
  .
\end{align*}
and thus  the odd powers of $H_XC_X$ are
\begin{align*}
	\left(H_X C_X\right)^{2n+1}& =\left(\frac{\alpha zf_{y,12}}{1-\overline{\alpha}zf_{y,21}}\right)^{n}H_X C_X
\end{align*}
and similarly 
\begin{align*}
  \left(H_Y C_Y\right)^{2n+1}&=\left(\frac{\overline{\alpha}zf_{x,31}}{1-\alpha z f_{x,13}}\right)^{n}H_Y C_Y.
\end{align*}
Now the right hand side of     \eqref{eq:Fit=QietaXi} is
\begin{align*}
  \widetilde{\eta}_X(z H_X C_X)
  &=\sum_{n=1}^{\infty}\beta_{2n}(X)\left(z H_X C_X\right)^{2n-1}
    =\frac{1-\overline{\alpha}zf_{y,21}}{\alpha z^2f_{y,12}}
    \sum_{n=1}^{\infty}\beta_{2n}(X) z^{2n} \left(\frac{\alpha
    zf_{y,12}}{1-\overline{\alpha}zf_{y,21}}\right)^{n}
     H_XC_X
  \\
  &=
    \frac{1-\overline{\alpha}zf_{y,21}}{\alpha z^2f_{y,12}}
    \,
    \eta_X\!\!\left(z\sqrt{\frac{\alpha zf_{y,12}}{1-\overline{\alpha}zf_{y,21}}}\right)
    H_XC_X
    ,\\
  \intertext{and similarly}
  \widetilde{\eta}_Y(z H_Y
  C_Y)&=
        \frac{1-\alpha zf_{x,31}}{\overline{\alpha}z^2f_{x,13}}
    \,
        \eta_Y\!\!\left(z\sqrt{\frac{\overline{\alpha}zf_{x,13}}{1-\alpha zf_{x,31}}}\right)
        H_Y C_Y
        .
\end{align*}
Substituting this into
the second half of the system
\eqref{eq:main_system}
results in two matrix equations,
each with exactly two non-zero entries, which we can rewrite as
\begin{align}\label{eq:final_ac}
  \left\{
	\begin{aligned}
		f_{x,13}&=\frac{1}{z^2f_{y,12}}\,\eta_X\!\!\left(z\sqrt{\frac{\alpha
                      zf_{y,12}}{1-\overline{\alpha}zf_{y,21}}}\right)
                &
                f_{y,12}&=\frac{1}{z^2f_{x,13}}\,\eta_Y\!\!\left(z\sqrt{\frac{\overline{\alpha}zf_{x,13}}{1-\alpha
                      zf_{x,31}}}\right)
                \\
		f_{x,31}&=\frac{1-\overline{\alpha}zf_{y,21}}{\alpha
                  z}\,\eta_X\!\!\left(z\sqrt{\frac{\alpha zf_{y,12}}{1-\overline{\alpha}zf_{y,21}}}\right)
                &
		f_{y,21}&=\frac{1-\alpha
                  f_{x,31}}{\overline{\alpha}z}\,\eta_Y\!\!\left(z\sqrt{\frac{\overline{\alpha}zf_{x,13}}{1-\alpha
                      zf_{x,31}}}\right)
	\end{aligned}
  \right.
\end{align}
On the other hand,
\begin{equation}
  \label{eq:cond_exp_ac}
\begin{aligned}
  \E_X[(1-z^2 T_\alpha)^{-1}]
  &=
    e_1^t
    \E_X\left[(I-zL)^{-1}\right]
    e_1\\
    &= e_1^t(I-zH_X C_X X)^{-1} H_Xe_1\\
     & =(1-\overline{\alpha}zf_{y,21}-\alpha z^3f_{y,12} X^2)^{-1}\\
  \E_Y[(1-z^2T_\alpha)^{-1}]
  &=
    e_1^t
    \E_Y\left[(I-zL)^{-1}\right]
    e_1
\\
  &  = e_1^t(I-zH_Y C_Y Y)^{-1} H_Ye_1
\\
  &=(1-\alpha z f_{x,31}-\overline{\alpha} z^3 f_{x,13}Y^2)^{-1}
\end{aligned}
\end{equation}

Assuming $\abs\alpha=1$, substitute
\begin{align*}
\tilde{f}_{x,13} &= \overline{\alpha}zf_{x,13}    
&
\tilde{f}_{x,31} &= \alpha zf_{x,31}    
&
\tilde{f}_{y,12} &= \alpha z f_{y,12}    
&
\tilde{f}_{y,21} &= \overline{\alpha} zf_{x,21}    
\end{align*}
into the system \eqref{eq:final_ac}
and the conditional expectations
  \eqref{eq:cond_exp_ac}
to obtain the system
\begin{equation}\label{eq:final_ac1}
  \left\{
	\begin{aligned}
		\tilde{f}_{x,13}&=\frac{1}{\tilde{f}_{y,12}}\,\eta_X\left(z\sqrt{\frac{\tilde{f}_{y,12}}{1-\tilde{f}_{y,21}}}\right)
                &
                \tilde{f}_{y,12}&=\frac{1}{\tilde{f}_{x,13}}\,\eta_Y\left(z\sqrt{\frac{\tilde{f}_{x,13}}{1-\tilde{f}_{x,31}}}\right)
                \\
		f_{x,31}&=(1-\tilde{f}_{y,21})\,\eta_X\left(z\sqrt{\frac{\tilde{f}_{y,12}}{1-\tilde{f}_{y,21}}}\right)
                &
		f_{y,21}&=(1-\tilde{f}_{x,31})\,\eta_Y\left(z\sqrt{\frac{\tilde{f}_{x,13}}{1-\tilde{f}_{x,31}}}\right)
	\end{aligned}
  \right.
\end{equation}
and the expressions
\begin{align*}
  \label{eq:cond_exp_ac1}
  \E_X[(1-z^2 T_\alpha)^{-1}]
     & =
       \bigl(
       1-\tilde{f}_{y,21}- z^2\tilde{f}_{y,12} X^2
       \bigr)^{-1}\\
  \E_Y[(1-z^2T_\alpha)^{-1}]
  &=
    \bigl(
    1-\tilde{f}_{x,31}- z^2\tilde{f}_{x,13}Y^2
    \bigr)^{-1}
\end{align*}
which both  do not depend on $\alpha$, and neither does the moment generating
function
$$
M(z)
= \frac{1}{1-\tilde{f}_{x,31}-\tilde{f}_{y,21}}
.
$$

\subsection{A Lie polynomial}
\label{ssec:ex:lie}
We consider next the conditional expectation of the Lie polynomial
\begin{equation*}
  \label{eq:liepoly}
  P = \sigma X+i(XY-YX)
  .
\end{equation*}
First we show the general form of the subordination, and next we show that for
specific choices of the distributions of $X$ and $Y$ we can go a step  further
and obtain explicit or implicit formulas  for the coefficients of the rational expression of
the conditional expectation.

The linearization for this example is
$$
(1-z^2P)^{-1} = u^t(I-zL)^{-1}v
$$
with
$u^t=(1,0,0)$, $v^t=(1,0,0)$ and $L=C_X X+C_Y Y$ where
\begin{align*}
C_X&=
\begin{bmatrix}
	\sigma z & 0 & i  \\
	1 & 0 & 0    \\
	0 & 0 & 0   
\end{bmatrix}
&
C_Y&=
\begin{bmatrix}
	0 & -i & 0 \\
	0 & 0 & 0 \\
	1 & 0 & 0
\end{bmatrix}
.
\end{align*}
The cokernel projections are
\begin{align*}
Q_X&=
\begin{bmatrix}
	1 & 0 & 0  \\
	0 & 0 & 0    \\
	0 & 0 & 1   
\end{bmatrix}
&
Q_Y&=
\begin{bmatrix}
	1 & 0 & 0 \\
	0 & 1 & 0 \\
	0 & 0 & 0
\end{bmatrix}
\end{align*}
and the compressed matrices $\widetilde{F}_X=Q_XF_XQ_X$ and $\widetilde{F}_Y=Q_YF_YQ_Y$ have the form
\begin{align*}
  Q_XF_XQ_X
  &=
\begin{bmatrix}
	f_{x,11} & 0 & f_{x,13} \\
	0 & 0 & 0 \\
	f_{x,31} & 0 & f_{x,33} 
\end{bmatrix}
&
Q_YF_YQ_Y&=
\begin{bmatrix}
	f_{y,11} & f_{y,12} & 0 \\
	f_{y,21} & f_{y,22} & 0 \\
	0 & 0 & 0 
\end{bmatrix}
.
\end{align*}
From this we obtain the  conditional expectations
\eqref{eq:EIps=utsumCiXiCjFj} as follows
\begin{align}
  \label{eq:EXlie}
  \E_X\left[(1-z^2P)^{-1}\right]
&=
    \left(
    1 +izf_{y,21}  - z^2(\sigma -if_{y,22}+if_{y,11})X  -iz^3f_{y,12}X^2 
    \right)^{-1}
  \\
    \label{eq:EYlie}
 \E_Y\left[(1-z^2P)^{-1}\right]
  &=
    \left(
    1
-\sigma z^2 f_{x,11}
-iz f_{x,31}
- z^2(\sigma zf_{x,13}+if_{x,33}-if_{x,11}) Y
    + iz^3f_{x,13} Y^2
    \right)^{-1}
\end{align}

\subsubsection{Semicircular and Bernoulli}
\begin{proposition}\label{Cor:semicircleExample}
  Let  $a$ be a standard semicircular random variable,
  free from a standard symmetric Bernoulli random variable $b$.
  Then $a$ and $i(ab-ba)$ have semicircle distribution with variances $1$ and $2$ respectively, they are not free, and yet $\sigma a + i(ab-ba)$ has semicircle  distribution with variance
  $\sigma^2+2$ for every $\sigma\in\R$.
\end{proposition}
In other words, we obtain  a pair of semicircular elements in ``additive free position''
(an infinite dimensional analogue of \cite[Definition~5.1]{Marcus:2021} and \cite{ArizmendiLehnerRosenmann:2309.14343}),
i.e., any real linear combination of $a$ and $i(ab-ba)$ has semicircular distribution, without
actually being free.
The first instance of this phenomenon was observed in \cite{BoseDeyEjsmont:2018:characterization}.

\begin{proof}
  Indeed $ab$ and $ba$ are free \cite[Lemma~3.4]{HaagerupLarsen},
  but $a$ and $c=i(ab-ba)$ are not:
  a quick calculation shows that $\kappa_4(a,a,c,c)=1$;
  alternatively, this can be seen by comparing the conditional
  expectation     \eqref{eq:EXlie}  with
  the conditional expectation \cite{Biane98}
  $$
  E_X[(z-\sigma X-\sqrt{2}Y)^{-1}]
  = (\omega(z) - \sigma X)^{-1}
  $$
  where
  $X$ and $Y$ are free semicircular, which is the inverse of a linear polynomial in
  $X$.
  Yet because of some nontrivial cancellations $\kappa_n(a+c)=0$ for all $n\neq 2$ and thus $a+c$ has the same distribution as the sum of free copies of $a$ and $c$.
  We note in passing that this phenomenon does not occur with $X + XY+YX$ which
  indeed does not have the same distribution as $X + i(XY-YX)$ as could be surmised from
  example in section~\ref{ssec:commutator}.

  Put $X=a$ where $a$ has standard semicircle distribution and
  $Y=b$ where b has Bernoulli distribution, then
  $z\widetilde{\eta}_X(z)^2-\widetilde{\eta}_X(z)+z=0$ and
  $\widetilde{\eta}_Y(z)=z$ and the system
  \eqref{eq:Fit=QietaXi}
  takes the form
  \begin{equation*}
    \left\{
      \begin{aligned}
        Q_X(I-z C_Y \widetilde{F}_Y )^{-1}C_X\widetilde{F}_X^2+z Q_X(I-z C_Y \widetilde{F}_Y )^{-1}C_X-\widetilde{F}_X&=0\\
        \widetilde{F}_Y-z Q_Y (I-z C_X \widetilde{F}_X )^{-1} C_Y&=0.
      \end{aligned}
    \right.
  \end{equation*}
From the explicit solution we obtain (for $\sigma=1$) the expressions
  \begin{multline*}
    \E_a\left[(1-z^2P(a,b))^{-1}\right]
\\
    =
    (12 z^4+3)
    \left(
      2
      + ( z^4+1 ) \nu(z)+14  z^4
      -( 3  z^2  \nu(z) +24  z^6 ) a
      +( 3  \nu(z) -6 )  z^4 a^2
    \right)^{-1}
    .
  \end{multline*}
  where $\nu(z)=\sqrt{1-12 z^4}$
  and
  \begin{equation*}
    \E_b\left[(1-z^2P(a,b))^{-1}\right]
=\frac{1-\sqrt{1-12 z^4}}{12z^4}
  \end{equation*}
  which does not depend on $b$, and is equal to the moment generating function of a semicircular
  variable of variance 3.
\end{proof}

\subsubsection{Semicircular and Semicircular}
In this case
the system 
    \eqref{eq:Fit=QietaXi} takes the form
\begin{equation*}
  \left\{
	\begin{aligned}
		Q_X(I-z C_Y \widetilde{F}_Y )^{-1}C_X\widetilde{F}_X^2+z Q_X(I-z C_Y \widetilde{F}_Y )^{-1}C_X-\widetilde{F}_X&=0\\
		Q_Y(I-z C_X \widetilde{F}_X )^{-1}C_Y\widetilde{F}_Y^2+z Q_Y(I-z C_X \widetilde{F}_X )^{-1}C_Y-\widetilde{F}_Y&=0.
	\end{aligned}
  \right.
\end{equation*}
and from its solution we obtain the quartic
equation
$$
\sigma^2  z^4  M(z)^4 -z^2  M(z)^3 -(\sigma^2+1 ) z^2  M(z)^2+M(z) -1=0
$$
for the moment generating function $M(z)$.
Plots for $\sigma=1/2$ and $\sigma=1$ are shown in
Figure~\ref{fig:densX+comm05}.

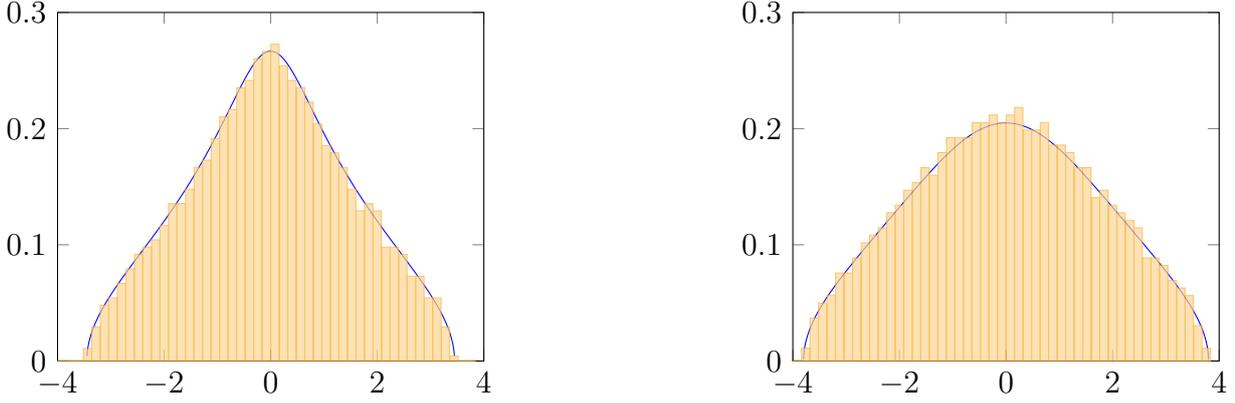
\begin{figure}[h]
	\centering{}
	\begin{minipage}{0.47\textwidth}
\begin{tikzpicture}
 \begin{axis}[
  width=0.9\textwidth,
 xmin=-4,
 xmax=4,
 ymin=0,
 ymax=0.3,
yticklabel style={
        /pgf/number format/fixed,
        /pgf/number format/precision=5
},
scaled y ticks=false
 ]
 \addplot[color=blue] table {pictures/densX+comm05.dat};
 \end{axis}
 \begin{axis}[ybar, bar width=0.155,
  width=0.9\textwidth,
 xmin=-4,
 xmax=4,
 ymin=0,
 ymin=0.3,
 yscale=1,
 yticklabel=\empty,
 xticklabel=\empty,
 hide axis
   ]
   \addplot+[draw=histcolor,fill=histcolor!50!white,opacity=0.7,line width = 0.15] 
file {pictures/histX+comm05.dat};
\end{axis} \end{tikzpicture}  
\end{minipage}\hfill
\begin{minipage}{0.47\textwidth}
	\centering{}
\begin{tikzpicture}
 \begin{axis}[
  width=0.9\textwidth,
 xmin=-4,
 xmax=4,
 ymin=0,
 ymax=0.3,
yticklabel style={
        /pgf/number format/fixed,
        /pgf/number format/precision=5
},
scaled y ticks=false
 ]
 \addplot[color=blue] table {pictures/densX+comm1.dat};
 \end{axis}
 \begin{axis}[ybar, bar width=0.155,
  width=0.9\textwidth,
 xmin=-4,
 xmax=4,
 ymin=0,
 ymin=0.3,
 yscale=0.8,
 yticklabel=\empty,
 xticklabel=\empty,
 hide axis
   ]
   \addplot+[draw=histcolor,fill=histcolor!50!white,opacity=0.7,line width = 0.15] 
file {pictures/histX+comm1.dat};
\end{axis} \end{tikzpicture}  
\end{minipage}
\caption{Densities of $X+i[X,Y]$ and $\frac{1}{2}X+i[X,Y]$ with both $X$ and $Y$ semicircular
	and a sample of eigenvalues of the corresponding 1000$\times$1000 random
	matrix models.}
\label{fig:densX+comm05}
\end{figure}

\subsection{A Symmetric polynomial of  degree 3}
\label{ex:XYZ}
The noncommutative homogeneous symmetric polynomial of degree 3
\begin{equation}
  \label{eq:XYX:T3}
T_3=XYZ+YXZ+XZY+ZXY+YZX+ZYX
\end{equation}
in free identically distributed semicircular variables
$X,Y,Z$ has a linearization of dimension 7 and the procedure described at the
beginning of this section succeeds.
We omit the details of the lengthy computation which is similar to the one in
Example~\ref{ssec:commutator} and which yields the following result.
\begin{proposition}
  Let $\alpha$, $\beta$, $\gamma\in\IC$ with
  $\abs{\alpha}=\abs{\beta}=\abs{\gamma}=1$.
  Then the distribution of
  \begin{equation}
    T_{\alpha\beta\gamma}=
  \alpha ZXY
  +\overline\alpha YXZ
  +\beta XYZ
  +\overline\beta ZYX 
  +\gamma YZX
  +\overline\gamma XZY
\end{equation}
is the same as the distribution of $T_3$.
This is in particular true for the third Amitsur-Levitzki standard polynomial
\begin{equation}
    \label{eq:XYX:S3}
  S_3 = i(XYZ-XZY+YZX - YXZ+ZXY-ZYX)
.
\end{equation}
\end{proposition}
Interestingly, contrary to Example~\ref{ssec:commutator} this is not true for the conditional
expectations:
Indeed,
the conditional expectation
\begin{multline}
\label{eq:XYZ:EXY}
E_{XY}[(1-z^3T_{\alpha\beta\gamma})^{-1}]
\\
=
\bigl(
1 - c_{Z,0}(z) - c_{Z,2}(z)(X^2+Y^2) - c_{Z,4}(z) (\overline{\alpha}YX+\beta XY)(\alpha XY+\overline{\beta}YX)
\bigr)^{-1}
\end{multline}
(and analogously $E_{YZ}$ and $E_{XZ}$) depends on the coefficients $\alpha$,
$\beta$ and $\gamma$.
Here 
the coefficient functions $c_{X,i}(z)$, $c_{Y,i}(z)$ and $c_{Z,i}(z)$  do not depend on $\alpha$, $\beta$ and $\gamma$.
In particular, in the case of the symmetric polynomial   \eqref{eq:XYX:T3}
we have
$$
E_{XY}[ (1-z^3T_3)^{-1}] = 
\bigl(
1 - c_{Z,0}(z) - c_{Z,2}(z)(X^2+Y^2) - c_{Z,4}(z) (XY+YX)^2
\bigr)^{-1}
$$
while for the Amitsur-Levitzki  polynomial   \eqref{eq:XYX:S3} we have
$$
E_{XY}[ (1-z^3S_3)^{-1}] = 
\bigl(
1 - c_{Z,2}(z)(X^2+Y^2) + c_{Z,4}(z) (XY-YX)^2
\bigr)^{-1}
.
$$
On the other hand,
\begin{equation}
  \label{eq:XYZ:EX}
\E_X[(1-z^3T_{\alpha\beta\gamma})^{-1}] 
=
\left(
  1 - c_{Y,0}(z)-c_{Z,0}(z)-(c_{Y,2}(z)+c_{Z,2}(z))X^2
\right)^{-1}
\end{equation}
does not depend on $\alpha$, $\beta$ and $\gamma$.

\subsubsection{Semicircular case}
In the case where all variables have standard semicircular distribution
the moment generating function of the element $T_3$ from   \eqref{eq:XYX:T3}
has power series expansion
\begin{equation*}
M(z) = 1+6 z^2 + 96 z^4 + 2064 z^6 + \dotsm
\end{equation*}
and satisfies the surprisingly simple quartic equation
\begin{equation*}
  2M(z)^2(M(z)+2)^2z^2-3(M(z)-1)=0
\end{equation*}
It follows that 
the Cauchy transform $G(z)$
satisfies  the  equation
\begin{equation}
  {2  {{G(z)}^4}  {{z}^2}}
  +{8  {{G(z)}^{3}} z}
  +{8  {{G(z)}^2}} 
  -{3  G(z)  {{z}^5}}
  +{3  {{z}^4}}
  = 0  
\end{equation}
from which the following equation for the density $-\frac{1}{\pi}v(t)$ can be extracted:
\begin{multline*}
  {{65536}  {{t}^{8}}  {{v}^{{12}}}}
  +{{262144}  {{t}^{6}}  {{v}^{{10}}}}
  +{{(
        {{98304}  {{t}^{6}}}+{{393216}  {{t}^4}} )} {{v}^{8}}}
  +{{(
      -{{59904}  {{t}^{6}}}+{{98304}  {{t}^4}}+{{262144}  {{t}^2}} )}
  {{v}^{6}}}
\\
+{{( -{{313344}  {{t}^4}} -{{98304}  {{t}^2}}+{65536} )}
  {{v}^4}}
- {{( {{93312}  {{t}^4}} +{{262656}  {{t}^2}} +{98304}
    )} {{v}^2}}
\\
-{{2187}  {{t}^4}}+{{79488}  {{t}^2}}+{27648} =0\end{multline*}
The spectral radius is 
$$
\rho 
  =\frac{4}{9}\sqrt{26\sqrt{ 13}+92}
  \approx
  6.05724678788
$$
and
a plot and histogram of the density is shown in Fig.~\ref{fig:densXYZ}.

\begin{figure}[h]
	\centering{}
\begin{tikzpicture}
 \begin{axis}[
  width=0.6\textwidth,
 xmin=-6.06,
 xmax=6.06,
 ymin=0,
 ymax=0.3,
yticklabel style={
        /pgf/number format/fixed,
        /pgf/number format/precision=5
},
scaled y ticks=false
 ]
 \addplot[color=blue] table {pictures/densXYZ.dat};
 \end{axis}
 \begin{axis}[ybar, bar width=0.155,
  width=0.6\textwidth,
 xmin=-6.06,
 xmax=6.06,
 ymin=0,
 ymin=0.2,
 yscale=0.75,
 yticklabel=\empty,
 xticklabel=\empty,
 hide axis
   ]
   \addplot+[draw=histcolor,fill=histcolor!50!white,opacity=0.7,line width = 0.15] 
file {pictures/histXYZ.dat};
\end{axis}
\end{tikzpicture}  
\caption{Density of $XYZ+YXZ+XZY+ZXY+YZX+ZYX$ for semicircular variables and a sample of eigenvalues of
  the corresponding 1000$\times$1000 random matrix model.}
	\label{fig:densXYZ}
\end{figure}
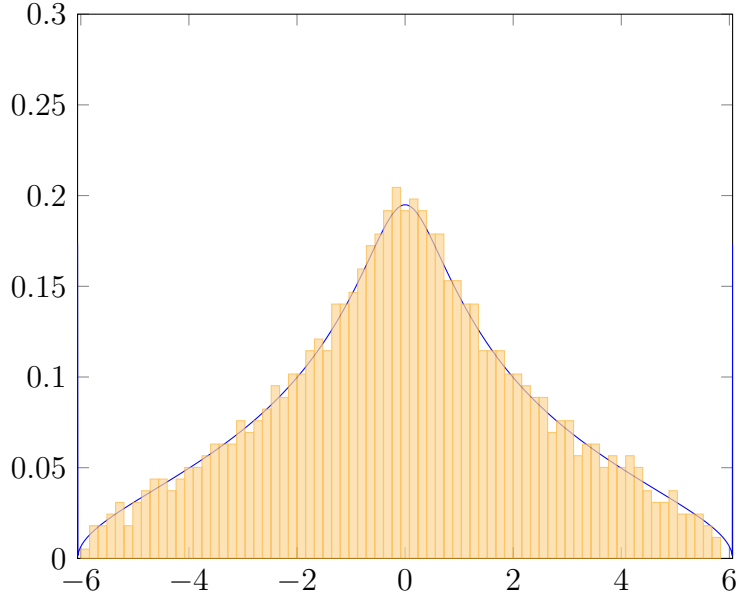

The conditional expectations \eqref{eq:XYZ:EXY} and \eqref{eq:XYZ:EX} are
evaluated as follows.
As a consequence of symmetry the coefficients $c_{X,i}(z) = c_{Y,i}(z) =
c_{Z,i}(z) =:c_i(z)$ are identical for all variables and the system is
solvable:
The constant coefficient
$$
{c}_0(z) = 2z^2+20z^4+376z^6+\dotsm
$$
satisfies the equation
$$
27{c}_0(z)^4-27{c}_0(z)^3+(8z^2+9){c}_0(z)^2
-(8z^2+1){c}_0(z)+2z^2=0,
$$
the quadratic coefficient
$$
{c}_2(z) = z^2+8  z^4+140  z^6+\dotsm
$$
satisfies the equation
$$
12{c}_2(z)^4+8z^2{c}_2(z)^2-z^2{c}_2(z)+z^4=0
$$
and the quartic coefficient is $c_4(z) = z^2M(z)$.

\subsubsection{Arcsine law and a random walk on the free group}
Let $\IF_3$ be the free group with three generators $g_1$, $g_2$ and $g_3$
and put $X=g_1+g_1^{-1}$, $Y=g_2+g_2^{-1}$ and $Z=g_3+g_3^{-1}$.
Then up to the factor $N=48$ the operator
$T_3$ from  \eqref{eq:XYX:T3}
is the transition operator of a simple random walk of step length 3 on the
Cayley graph of $\IF_3$, i.e., the regular tree of order 6.
More precisely,
$$
T_3 = \sum g_{\sigma(1)}^{\epsilon_1} g_{\sigma(2)}^{\epsilon_2} g_{\sigma(3)}^{\epsilon_3}
$$
where the sum extends over all permutations of the generators and all choices
of exponents.
Now $X$, $Y$ and $Z$ have arcsine distribution,
and the calculation results in the quartic equation
\begin{equation}
  \label{eq:XYZ:arcsine}
( 2304  z^2 -1 ) M(z)^4+( 9216 z^2 -20 ) M(z)^3+( 9216  z^2      -114 )
M(z)^2 -140  M(z)+275= 0
\end{equation}
for the moment generating function
$$
M(z) = 1+48  z^2+5184  z^4+720384  z^6+113304576  z^8+19186556928
z^{10}+\dotsm
.
$$
The moment generating function, also called \emph{Green's function},
carries essential information about the random walk:
\begin{enumerate}[1.]
 \item 
The moments of the transition operator reproduce the $n$-step return
probabilities
$$
p^{(n)}(e,e) = \frac{m_n}{N^n}
$$
\item  The smallest positive singularity of $M(z)$ is the reciprocal of the spectral radius
 $\rho$ which governs the asymptotics of the random walk
 \cite[Ch.~2]{Woess:2000},
 in particular, 
 $$
 m_n\propto\rho^n
 $$
\end{enumerate}
In our example the spectral radius can be computed from the discriminant
of equation   \eqref{eq:XYZ:arcsine}. The result
is $\rho=8\sqrt{y_0}\approx  15.08724$ where $y_0$ is the largest real
root of the cubic equation
$$
y^3+326  y^2 -1135  y -132    = 0
.
$$

\subsection{A rational example}
\label{ssec:ex:rational}
Non-commutative rational functions have linearizations just like polynomials,
(see, e.g.,  \cite{HeltonMaiSpeicher:ApplicationsOfRealizations} for excellent
discussion about relations between linearizations and questions in free
probability) and consequently our method is not restricted to
polynomials but also allows to compute conditional expectations and
distributions of non-commutative rational functions in free random variables.
Let us illustrate this with the concrete example of the rational function
$r(X,Y)=X(1-a X-b Y)^{-1}X$ which will show that nevertheless
certain technical issues arise.
First we must put conditions on the coefficients $a$ and $b$ in order to
ensure invertibility of $1-aX-bY$.
Next we compute the conditional expectation of $(1-z r(X,Y))^{-1}$
onto $X$ and $Y$ and then integrate the latter in order to determine the
distribution of  $X(1-aX-bY)^{-1}X$.
Using the method described in Appendix~\ref{app:linearization} we obtain the
linearization
$$
\left(1-zX(1-aX-bY)^{-1}X\right)^{-1}
= u^t(I-L)^{-1}v
$$
where  $L=C_X X+C_Y Y$ with
\begin{align*}
C_X=
\begin{bmatrix}
	0 & z \\
	1 & a   
\end{bmatrix}
\quad
C_Y=
\begin{bmatrix}
	0 & 0 \\
	0 & b
\end{bmatrix}
\end{align*}
and $u=v=(1,0)^t$.

Let us record some observations here. In order to be able to use the iteration scheme from
Lemma~\ref{lem:iteration} it is necessary to start with the correct constant terms
$F_X=\beta_1(X)I$ and $F_Y=\beta_1(Y)I$, however there is a catch:
The constant term in $F_X(z)=\fbeta{X}(X(I-L)^{-1})$ is not equal to
$\beta_1(X)$, since $L$ lacks the factor $z$ and there are infinitely many
terms in the series expansion which do not depend on $z$. Therefore
in order to be able to distinguish the correct solution coming from
Lemma~\ref{lem:iteration} we introduce an extra parameter $s\in\IC$ and 
consider $F_X(s,z)=\fbeta{X}(X(I-s L)^{-1})$. Then our iteration converges to the
unique solution of the system of equations for the matrices $F_X$ and $F_Y$ whose entries are
analytic in $s$. In order to evaluate the conditional expectation
in  $X=a$ and $Y=b$ for  some free variables $a,b\in\alg{M}$, we have to make
sure that the corresponding functional
$\fbeta{X}\left(X(I-sL)^{-1}\right)$  is analytic at $s=1$.
The following rough estimate will ensure this.

We expand the Neumann series into all words $W\in\alg{X}^*$ over the alphabet
$\alg{X}=\{X,Y\}$.
For a word $W=W_1W_2\cdots W_n\in \alg{X}^*$ we denote by
$C_W=C_{W_1}C_{W_2}\cdots C_{W_n}$ the corresponding coefficient matrix.
The length  of a word $W$ will be denoted  by $\abs{W}$
and the number of occurrences of the letters $X$ and $Y$ by
$\abs{W}_X$ and $\abs{W}_Y$, respectively,
i.e.,  $\abs{W}=\abs{W}_X+\abs{W}_Y$. 
Then we have
\begin{align*}
	\fbeta{X}(X(I-sL)^{-1})=\sum_{W\in\alg{X}^*} s^{\abs{W}}\fbeta{X}(XW)C_W.
\end{align*}
A standard estimate gives
$\abs{\fbeta{X}(XW)}\leq
2^{\abs{W}+1}\norm{X}^{\abs{W}_X+1}\norm{Y}^{\abs{W}_Y}$.
Moreover it is easy to verify that $\norm{C_X}^2\leq 1+a^2+\abs{z}^2$ and
$\norm{C_Y}=\abs{b}$.
Thus we have
\begin{align*} 
  \abs{\fbeta{X}(X(I-sL)^{-1})}
  &\leq \sum_{W\in\alg{X}^*} s^{\abs{W}}\abs{\fbeta{X}(XW)}\norm{C_W}\\
	&\leq \sum_{W\in\alg{X}^*} s^{\abs{W}} 2^{\abs{W}+1}\norm{X}^{\abs{W}_X+1}\norm{Y}^{\abs{W}_Y}\norm{C_X}^{\abs{W}_X+1}\norm{C_Y}^{\abs{W}_Y}\\
	&=2 \norm{X}\left(1-2 s(\norm{X} \norm{C_X}+\norm{Y} \norm{C_Y})\right)^{-1}.
\end{align*} 
The last equality holds whenever $2 s(\norm{X} \norm{C_X}+\norm{Y} \norm{C_Y}<1$
and a sufficient condition for this inequality to hold for $s=1$ is $2(\norm{X}\sqrt{1+a^2+\abs{z}^2}+\norm{bY})<1$.

If we abbreviate $R=(1-aX-bY)^{-1}$ and $\Psi = (1-zXRX)^{-1}$ then a short calculation using the Schur
complement results in
$$
(I-L)^{-1} =
\begin{bmatrix}
  \Psi & z\Psi XR\\
  RX\Psi & R +zRX\Psi XR
\end{bmatrix}
.
$$
It follows that
$$
F_Y = \fbeta{Y}(YL) =
\begin{bmatrix}
  0 & f_{y,12}\\
    0 & f_{y,22}
\end{bmatrix}
$$
and
\begin{align}
\label{eq:ratexample:EX}
\IE_X\left[\Psi\right]
  &=
    (1-bf_{y,22} -aX)
    \left(
    1-bf_{y,22} - a X -     zX^2
    \right)^{-1}
    \\
  &=
    \left(
    1-zX(1-  a X -  bf_{y,22})^{-1}X
    \right)^{-1}
    \\
\IE_Y\left[\Psi\right]
  &=
    (1-f_{x,12}-af_{x,22} -bY)
    \left(
    1-f_{x,12}-af_{x,22} - z(f_{x,21}+\det F_X) -b(1+f_{x,21}Y
    \right)^{-1}
  \\
  M(z)
  &= \frac{
    1-f_{x,12}-af_{x,22}-bf_{y,22}
    }{
    1-f_{x,12}-af_{x,22} -bf_{y,22}- z(f_{x,21}+\det F_X) -bf_{x,21}f_{y,22}
    }
\end{align}

\subsubsection{Bernoulli}
If both $X$ and $Y$ have Bernoulli distribution we obtain
\begin{align*}
  f_{y,22}(z)
  &=\frac{
    \sqrt{  ((z-1)^2-(a-b)^2)((z-1)^2-(a+b)^2)  } - (z-1)^2+a^2-b^2
    }{
    2b(z-1)
    }
  \\
  M(z) &=1+\frac{
  z(z-1)\sqrt{ ((z-1)^2-(a-b)^2)((z-1)^2-(a+b)^2)}
  }{
   ((z-1)^2-(a-b)^2)((z-1)^2-(a+b)^2)
  }
\end{align*}
Note that in this case $X^2=1$ and our element $X(1-aX-bY)^{-1}X$ has the same distribution
as  $(1-aX-bY)^{-1}$ which can be obtained by additive free convolution as
well. However this is not true of the conditional expectation \eqref{eq:ratexample:EX}.
A plot of the density for $a=b=\frac{1}{8}$ is shown in
Fig.~\ref{fig:x_rational}.

\subsubsection{Semicircular}
If both $X$ and $Y$ have semicircle distribution then the moment generating
function
satisfies the quintic equation
\begin{multline*}
16b^2z^6M(z)^5+(-4(8b^2+a^2)z^5+(24b^4-36a^2b^2-a^4)z^4)
M(z)^4
\\
+(16b^2z^5-4(8b^4+(-8a^2-5)b^2-2a^2)z^4-(4b^2+a^2)(7b^2-10a^2)z^3+(9b^6-25a^2b^4+22a^4b^2+2a^6)z^2)
M(z)^3
\\
+(4(4b^4-4b^2-a^2)z^4+(40b^4+(-48a^2-4)b^2-8a^4-5a^2)z^3
\\
-(16b^6+(-56a^2-8)b^4+(38a^4+a^2)b^2+2a^6+9a^4)z^2
-(b^2+a^2)
(5b^4-7a^2b^2+6a^4)z+(b-a)(b+a)(b^2+a^2)^3)
M(z)^2
\\
+(-4(4b^4+(-4a^2-1)b^2-a^2)z^3+(8b^6+(-48a^2-12)b^4+(16a^4+2a^2)b^2+12a^4+a^2)z^2
\\
+(9b^6-2a^2b^4+(-a^4-a^2)b^2+10a^6+a^4)z-2(b-a)(b+a)(b^2+a^2)^3)
M(z)
\\
+(4a^2+1)(2b^2-a)(2b^2+a)z^2-(4b^6-a^2b^2+4a^6+a^4)z+(b-a)(b+a)(b^2+a^2)^3
\end{multline*}

\begin{figure}[h]
	\centering{}
	\includegraphics[width=0.43\textwidth]{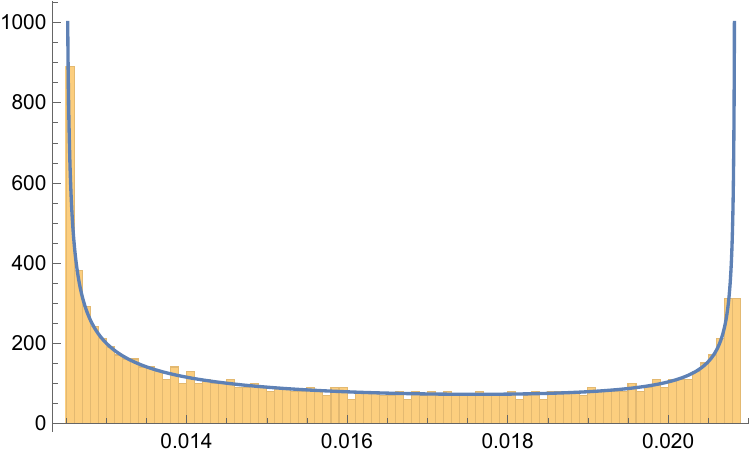}
	\caption{Density of $a(1-a-b)^{-1}a$ for $a,b$ with distribution
          $\frac{1}{2}(\delta_{-1/8}+\delta_{1/8}).$} together with a random matrix approximation.
	\label{fig:x_rational}
\end{figure}

\clearpage{}
\appendix{}
\section{An algebraic approach to Boolean cumulants}
\label{sec:algbool}
In this first appendix we present purely algebraic proofs of basic facts about
Boolean cumulants\ and implement a formal calculus on the double tensor algebra, which avoids the complications discussed in Remarks~\ref{rem:42} and~\ref{rem:augmentation}~\eqref{rem:item:augmentation}.

\subsection{Algebraic proofs of Lemma~\ref{lem:productformula} and
 Corollary~\ref{cor:RecursiveBoolProd}}
\label{ssec:algproof}
\begin{proof}[Proof of Lemma~\ref{lem:productformula}]
  Induction on $n$.
  We compute $\varphi(a_1a_2\cdots a_n)$ in two ways.
  First consider $a_pa_{p+1}$ as one factor 
  and apply recurrence \eqref{eq:recurrenceboolcum},
  splitting the sum into two at
  the entry $a_pa_{p+1}$.
  \begin{multline}
    \label{eq:phiap2}
    \varphi(a_1a_2\dotsm (a_pa_{p+1})a_{p+2}\cdots a_n)
    \\
    \begin{aligned}[t]
      &=
      \begin{multlined}[t]
        \sum_{k=1}^{p-1}
        \beta_k(a_1,a_2,\dots,a_k) 
        \,
        \varphi(a_{k+1}a_{k+2} \cdots a_n)
        \\
        + \sum_{k=p+1}^{n-1}
        \beta_{k-1}(a_1,a_2,\dots,a_pa_{p+1},a_{p+2},\dots,a_k) 
        \,
        \varphi(a_{k+1}a_{k+2} \cdots a_n)
        \\
        +         \beta_{n-1}(a_1,a_2,\dots,a_pa_{p+1},a_{p+2},\dots,a_n) 
      \end{multlined}
    \end{aligned}
  \end{multline}
  The left hand side is unchanged if we consider $a_p$ and $a_{p+1}$
  as separate factors.
  We apply recurrence \eqref{eq:recurrenceboolcum}, (artificially) splitting the sum 
  at the entry $a_p$.
  \begin{multline}
    \label{eq:phiap}      
    \varphi(a_1a_2\dotsm a_pa_{p+1}a_{p+2}\cdots a_n)
    \\
    \begin{aligned}[t]
      &=
      \begin{multlined}[t]
        \sum_{k=1}^{p-1}
        \beta_k(a_1,a_2,\dots,a_k) 
        \,
        \varphi(a_{k+1}a_{k+2} \cdots a_n)
+ \beta_{p}(a_1,a_2,\dots,a_p) \,\varphi(a_{p+1}a_{p+2}\cdots a_n)
        \\
        + \sum_{k=p+1}^{n-1}
        \beta_{k}(a_1,a_2,\dots,a_k) 
        \,
        \varphi(a_{k+1}a_{k+2} \cdots a_n)
+         \beta_{n}(a_1,a_2,\dots,a_n) 
      \end{multlined}
    \end{aligned}
  \end{multline}
  The first part is the same as in     \eqref{eq:phiap2}.
  By induction hypothesis, the second part of \eqref{eq:phiap2} can be
  replaced by
  \begin{multline*}
    \sum_{k=p+1}^{n-1}
    \beta_{k-1}(a_1,a_2,\dots,a_pa_{p+1},a_{p+2},\dots,a_k) 
    \,
    \varphi(a_{k+1}a_{k+2} \cdots a_n)
    \\
    =
    \sum_{k=p+1}^{n-1}
    (\beta_p(a_1,a_2,\dots,a_p)\beta_{k-p}(a_{p+1},a_{p+2},\dots,a_k)
    + \beta_{k}(a_1,a_2,\dots,a_k))
    \,
    \varphi(a_{k+1}a_{k+2} \cdots a_n)
    .
  \end{multline*}
  On the other hand, applying recurrence
  \eqref{eq:recurrenceboolcum} to
  the second part of \eqref{eq:phiap}  we obtain
  \begin{multline*}
    \beta_{p}(a_1,a_2,\dots,a_p) \,\varphi(a_{p+1}a_{p+2}\cdots a_n)
    \\
    = \beta_{p}(a_1,a_2,\dots,a_p)
    \sum_{k=p+1}^{n-1}
    \beta_{k-p}(a_{p+1},a_{p+2},\dots,a_k)\,\varphi(a_{k+1}a_{k+2}\cdots
    a_n)
    \\
    +
    \beta_{p}(a_1,a_2,\dots,a_p)
    \,
    \beta_{n-p}(a_{p+1},a_{p+2},\dots,a_n)
  \end{multline*}
  Finally canceling equal terms from  \eqref{eq:phiap2} and
  \eqref{eq:phiap} we arrive at     \eqref{eq:lem:productformula}.	
\end{proof}

\begin{proof}[Proof of Corollary~\ref{cor:RecursiveBoolProd}]
  By induction on the number $r$ of multiplication signs.
  For $r=0$ there is nothing to prove;
  for $r=1$ this is Lemma~\ref{lem:productformula}.

  Assume that the formula holds up to $r-1$ multiplication signs and consider a
  Boolean cumulant of the form \eqref{eq:RecursiveBoolProd}
with $r$ multiplication signs,
  i.e., $r=n-m-1$.
  Now apply \eqref{eq:lem:productformula}
  to remove the first multiplication
  \begin{multline*}
    \beta_{m+1}(a_1a_2\cdots a_{d_1},a_{d_1+1}a_{d_1+2}\cdots a_{d_2},\ldots,a_{d_{m}+1}a_{d_{m}+2}\cdots a_{n}) \\
    = \beta_{1}(a_1)\,
       \beta_{m+1}(a_2\cdots a_{d_1},a_{d_1+1}a_{d_1+2}\cdots    a_{d_2},\ldots,a_{d_{m}+1}a_{d_{m}+2}\cdots a_{n})\\
    +    \beta_{m+2}(a_1,a_2\cdots a_{d_1},a_{d_1+1}a_{d_1+2}\cdots    a_{d_2},\ldots,a_{d_{m}+1}a_{d_{m}+2}\cdots a_{n})
    .
  \end{multline*}
   In both terms there are now $r-1$ multiplications and the induction
   hypothesis can be applied to obtain   \eqref{eq:RecursiveBoolProd}.
\end{proof}

\subsection{Boolean cumulants via tensor algebras}

For a vector space $V$ denote by
$\tensT(V) = \bigoplus_{n\geq 0}V^{\otimes n}$ its \emph{tensor algebra}
and $\tensTb(V) = \bigoplus_{n>0}V^{\otimes n}$ its augmentation ideal,
i.e., $\tensTb(V)=\ker\epsilon$, where $\epsilon$ is 
the projection onto $\IC=V^{\otimes 0}$
called the \emph{counit}.
In the following we have to deal with two different tensor products
(in addition to the internal product in the algebra), which makes a total of
three different multiplication operations.
In order to distinguish different levels in the hierarchy of tensors
we will denote the multiplication in $\tensT(V)$
by the symbol $\itensor$ and thus elementary tensors
are written $a_1\itensor a_2\itensor\dotsm\itensor a_n$.
 The ``outer'' tensor on the double tensor algebra
$\tensT(\tensTb(V))$ is denoted by the usual symbol $\otimes$, i.e.,
any element of $\tensT(\tensTb(V))$ is a sum of simple tensors of the form
$$
(a_1\itensor a_2\itensor\dotsm\itensor a_k)\otimes (b_1\itensor b_2\itensor\dotsm\itensor
b_l)\otimes\dotsm\otimes (c_1\itensor c_2\itensor\dotsm\itensor c_m)
.
$$
The tensor algebra has the following fundamental extension property
\cite[Prop.~1.1.2]{LodayVallette:2012}:

\begin{lemma}
  \label{lem:derivTV}
  Let $V$ be a vector space and $M$ a $\tensT(V)$-bimodule.
  Then any linear map $f:V\to M$ can be extended to a derivation
  $\tensT(f):\tensT(V)\to M$ by setting
  $$
  \tensT(f)(v_1\itensor v_2\itensor \dotsm \itensor v_n) =
  \sum_{k=1}^n v_1\itensor v_2\itensor \dotsm \itensor v_{k-1}\itensor f(v_k) \itensor
  v_{k+1}\itensor v_{k+2}\itensor \dotsm \itensor v_n
  $$
\end{lemma}

In the following the underlying vector space will always be an algebra
$V=\alg{A}$,
whose multiplication will be denoted as usual by $ab$ or $a\cdot b$.

For example, let $\alg{A}$ be the free algebra and 
$\bdeconc:\alg{A}\to \alg{A}\otimes\alg{A} \subseteq
\tensTb(\alg{A})\otimes\tensTb(\alg{A})$ the reduced deconcatenation coproduct
$$
\bdeconc(a_1a_2\dotsm a_n) = \sum_{k=1}^{n-1} a_1a_2\dotsm a_k\otimes
a_{k+1}\dotsm a_n
;
$$
let further
$$ 
\fdeconc(a_1a_2\dotsm a_n) = a_1\itensor a_2\itensor \dotsm\itensor
a_n\in \tensT(\alg{A})
$$
the full deconcatenation, i.e., $\fdeconc(w_1w_2) = \fdeconc(w_1\itensor w_2)$.
Then the derivation from Lemma~\ref{lem:derivTV} is
$$
\tensT(\bdeconc)( w_1\itensor w_2\itensor \dotsm\itensor w_n) =
\sum_{k=1}^{n-1} ((w_1\itensor w_2\itensor\dotsm\itensor w_{k-1}) \otimes 1)\itensor
\bdeconc(w_k) \itensor(1\otimes(w_{k+1}\itensor\dotsm\itensor w_{n}))
$$
We will denote this derivation by $\bdeconc$ as well.

The moment and cumulant functionals are defined on $\tensTb(\alg{A})$ via
\begin{align*}
  \varphi(w_1\itensor w_2\itensor\dotsm\itensor w_n)
  &= \varphi(w_1w_2\dotsm w_n)
  \\
  \beta(w_1\itensor w_2\itensor\dotsm\itensor w_n)
  &= \beta_n(w_1, w_2,\dots, w_n)
\end{align*}
and the recurrence \eqref{eq:recurrenceboolcum} can be reformulated with the
second level deconcatenation operator $\bdeconci: \tensTb(A)\to
\tensTb(A)\otimes\tensTb(A)$ defined by
$$
\bdeconci (w_1\itensor w_2\itensor\dotsm\itensor w_n) =
\sum_{k=1}^{n-1} w_1\itensor w_2\itensor\dotsm\itensor w_k \otimes w_{k+1}\itensor w_{k+2}\itensor\dotsm\itensor w_n
$$
as follows:
\begin{equation*}
  \varphi(w_1\itensor w_2\itensor\dotsm \itensor w_n)  =
    \beta(w_1\itensor w_2\itensor\dotsm \itensor w_n)  +
    (\beta\otimes\varphi) \bdeconci (w_1\itensor w_2\itensor\dotsm\itensor w_n) 
\end{equation*}

On the other hand, if we
define the full Boolean cumulant
$\fbeta{}(w) =\beta(\fdeconc(w))$, i.e.
$$\fbeta{}(a_1a_2\dotsm a_n) = 
\beta(a_1\itensor a_2\itensor \dotsm \itensor a_n)
= \beta_n(a_1, a_2,\dots, a_n)
$$
and extend it to 
$\tensTb(\alg{A})$ via
\begin{align*}
  \fbeta{}(w_1\itensor w_2\itensor\dotsm\itensor w_n)
  &= \fbeta{}(w_1w_2\dotsm w_n)
    \\
  &= \beta(\fdeconc(w_1)\itensor\fdeconc(w_2)\itensor\dotsm\itensor\fdeconc( w_n))
\end{align*}

then the full Boolean cumulants satisfy the corresponding recurrence
\begin{equation*}
\varphi(w) = \fbeta{}(w) + (\fbeta{}\otimes\varphi) \bdeconc(w)  
\end{equation*}
which we can generalize to $\tensTb(\alg{A})$.
In order to formulate it, we introduce some further notation.
\begin{definition}
An interval partition is a partition whose blocks are intervals.
The interval partitions  $\IntPart(\{1,2,\dots,n\})$ are thus in bijection with
compositions  $\Comp(n) = \{ (k_1,k_2,\dots,k_m) \mid k_1+k_2+\dotsm k_m=n\}$.
The bijection maps an interval partition $\pi\in\IntPart(n)$
to the sequence of block lengths (in their natural order).
We denote this bijection by $C:\IntPart(n)\to\Comp(n)$.
A composition $\lambda=(k_1,k_2,\dots,k_m)\in\Comp(n)$
is uniquely determined by its partial sums, called \emph{descents}
$$
\descents(\lambda) = \{k_1,k_1+k_2,\dots,k_1+k_2+\dotsm + k_{m-1}\}
\subseteq \{1,2,\dots,n-1\}
$$
and thus the compositions of order $n$ are in bijection with the Boolean lattice
of order $n-1$,
The descents of an interval partition are the descents of
the corresponding composition $\descents(\pi) = \descents(C(\pi))$.
This bijection is a poset anti-isomorphism  and
$\pi\leq\rho$ if and only if $\descents(\rho)\supseteq\descents(\pi)$.
In particular, $\descents(\pi\vee\rho) = \descents(\pi)\cap\descents(\rho)$.
\end{definition}

\begin{proposition}
  \label{prop:boolcumrectensor}
  \begin{equation*}
    \beta(w_1\itensor w_2\itensor\dotsm\itensor w_n) =
    \fbeta{}(w_1\itensor w_2\itensor\dotsm\itensor w_n) + (\fbeta{}\otimes\beta)\circ\bdeconc(w_1\itensor w_2\itensor\dotsm\itensor w_n) 
  \end{equation*}
\end{proposition}
\begin{proof}
  Let $k_i=\ell(w_i)$ be the lengths of the words
  and
  $\rho=\{I_1,I_2,\dots,I_m\}\in\IntPart(n)$  the induced interval
  partition,
  i.e., $\abs{I_j}=k_j$ and   $C(\rho) = (k_1,k_2,\dots,k_m)$
  with descent set $\descents(\rho) =\{r_1,r_2,\dots,r_{m-1}\}$
  where $r_0=d_0(\rho)=0$, $r_1=d_1(\rho)=k_1$, $r_2=d_2(\rho)=k_1+k_2$, etc.
  Then if $w_i=a_{r_{i-1}+1}a_{r_{i-1}+2}\dotsm a_{r_i}$
  the product   formula
  \eqref{eq:BoolProd} can be rewritten in terms of descents as
  \begin{multline*}
    \beta_{m}(a_1a_2\cdots a_{r_1},a_{r_1+1}\dotsm
    a_{r_2},\ldots,a_{r_{m-1}+1}\cdots a_n)
    \\
    \begin{aligned}
    &=\sum_{\substack{\pi \in \IntPart(n)\\
        \descents(\pi)\cap\descents(\rho)=\emptyset}}\beta_\pi(a_1,\ldots,a_n)
    \\
    &= \beta_n(a_1,a_2,\dots, a_n)
    + \sum_{s=1}^m
    \sum_{\substack{\pi \in \IntPart(n)\setminus \{1_n\}\\
    \descents(\pi)\cap\descents(\rho)=\emptyset \\ d_1(\pi)\in I_s}}\beta_\pi(a_1,\ldots,a_n)
    \end{aligned}
  \end{multline*}
where  we split off the term corresponding to $\pi=1_n$ (i.e.,
    $\descents(\pi)=\emptyset$)
    and regroup the remaining terms according to the location of the first
    descent $d_1(\pi)$.
    Fix $1\leq s\leq m$ and assume that $d=d_1(\pi)\in I_s$.
    Then the first block of $\pi$ is $B_1 = \{1,2,\dots,d_1\}$ and
    $\pi=\{B_1\}\cup \pi'$ with $\pi'=\pi|_{[d+1,n]} \in \IntPart[d+1,n]$ and
    $\descents(\pi')\cap\descents(\rho')=\emptyset$ where
    $\rho'=\rho|_{[d+1,n]}$.
    Now by assumption $r_{s-1}<d<r_s$ and
    thus
    \begin{multline*}
    \sum_{\substack{\pi \in \IntPart(n)\setminus \{1_n\}\\
    \descents(\pi)\cap\descents(\rho)=\emptyset \\ d_1(\pi)\in I_s}}\beta_\pi(a_1,\ldots,a_n)\\
    \begin{aligned}
      &=
        \sum_{d=r_{s-1}+1}^{r_s-1} \beta_{d}(a_1,a_2,\dots,a_{d}) \sum_{
        \substack{\pi' \in \IntPart([d+1,n])\\
    \descents(\pi')\cap\descents(\rho')=\emptyset }}
      \beta_{\pi'}(a_{d+1},a_{d+2},\ldots,a_n)
\\
      &=
      \sum_{d=r_{s-1}+1}^{r_s-1} \beta_{d}(a_1,a_2,\dots,a_{d})
      \,
      \beta_{m-s+1}(a_{d+1}a_{d+2}\dotsm a_{r_s},w_{s+1},w_{s+2},\dots,w_n)
\\
      &= (\beta^\delta\otimes\beta)(
        (w_1 \itensor w_2\itensor \dotsm\itensor w_{s-1}\otimes 1)
        (\bdeconc w_s)
        (1\otimes (w_{s+1} \itensor w_{s+2}\itensor \dotsm\itensor w_{m}))
        )
    \end{aligned}
    \end{multline*}
\end{proof}

\begin{remark}
  Although freeness implies property $(CAC)$
  (Lemma~\ref{lem:beta(ab)=0})
  and thus $\fbeta{}(w_1\itensor
  w_2\itensor \dotsm \itensor w_n)=0$ if $w_1$ begins with $X$ and $w_n$ ends
  in $Y$,
  it is important to keep in mind that this is not necessarily true for
  $\beta{}(w_1\itensor   w_2\itensor \dotsm \itensor w_n)$.
\end{remark}

\begin{example}[Cumulants of products]
  Let us apply  Proposition~\ref{prop:boolcumrectensor}
  to
  $$
  \Psi^\itensor = \sum_{n=0}^\infty (XY)^{\itensor n}z^{2n} 
  $$
  to compute the cumulant generating function
  $$
  B(z)= \beta(\Psi^\itensor) = \sum_{n=0}^\infty \beta_n(XY)z^{2n}
  .
  $$
  From this we can then obtain the moment generating
function
$$
M(z)= \frac{1}{1-B(z)}
.
$$
Applying the Leibniz rule to the identity $\Psi^\itensor\itensor (1- z^2XY) =
1$ we obtain 
  \eqref{eq:derivationofpsi}
  \begin{equation*}
\bdeconc(\Psi^\itensor) =
    (\Psi^\itensor\itensor X)\otimes (Y\itensor
    \Psi^\itensor)
    z^2
    .
  \end{equation*}
  Now by Proposition~\ref{prop:boolcumrectensor} we have
  \begin{align*}
    \beta(\Psi^\itensor)
    &= \fbeta{}(\Psi^\itensor) + \fbeta{}(\Psi^\itensor\itensor X) \beta(Y\itensor \Psi^\itensor) z^2
  \end{align*}
  and similarly
  \begin{align*}
    \bdeconc(Y\itensor \Psi^\itensor)
    &= (Y\otimes 1)\itensor \bdeconc(\Psi^\itensor)
    \\
    &= (Y\itensor \Psi^\itensor\itensor X)\otimes (Y\itensor \Psi^\itensor)z^2
    \\
    \beta(Y\itensor \Psi^\itensor)
    &= \fbeta{}(Y\itensor \Psi^\itensor) +
      \fbeta{}(Y\itensor \Psi^\itensor\itensor X)\beta (Y\itensor
      \Psi^\itensor)z^2
      .
      \intertext{Consequently,}
      \beta(Y\itensor \Psi^\itensor)
    &= \frac{\fbeta{}(Y\itensor \Psi^\itensor)}{1-\fbeta{}(y\itensor \Psi^\itensor\itensor
      X)z^2}
    \\
    \beta(\Psi^\itensor) 
    &= \fbeta{}(\Psi^\itensor) + \frac{\fbeta{}(\Psi^\itensor\itensor
      X)\,\fbeta{}(Y\itensor \Psi^\itensor)z^2}{1-\fbeta{}(Y\itensor \Psi^\itensor\itensor
      X)z^2}
      \intertext{where}
      \fbeta{}(\Psi^\itensor\itensor X)
      &= \sum_{n=0}^\infty \beta_{2n+1}(X,Y,X,Y,\dots,X)z^{2n}\\
      \fbeta{}(Y\itensor \Psi^\itensor)
      &= \sum_{n=0}^\infty \beta_{2n+1}(Y, X,Y,X,\dots,Y)z^{2n}
  \end{align*}

  If we assume $x$ and $y$ free then we infer from the resolvent identity
$$
  \Psi = 1 + z^2XY\Psi
  $$
  that
  $\fbeta{}(\Psi^\itensor)=1$ and   $\fbeta{}(Y\itensor \Psi^\itensor\itensor X)=0$ and thus
  $$
  \beta(\Psi^\itensor)-1
  = \fbeta{}(\Psi^\itensor\itensor X)
    \,
    \fbeta{}(Y\itensor \Psi^\itensor)
  $$
Thus
$$
B(z) = \beta(\Psi^\itensor) = 1 + z \fbeta{X}(\Psi X)\fbeta{Y}(Y\Psi)
$$
and finally
$$
zB(z) = z + \omega_1(z)\,\omega_2(z)
$$
which reproduces
\cite[Theorem~3.2 (3)]{BelBerNewApproach} and
\cite[(10)]{ChistyakovGoetze:2011:arithmetic}.
\end{example}

\clearpage{}
\section{Rational series and linearizations}
\label{app:linearization}
The non-commutative free field and in particular non-commutative rational
functions have seen many applications in free probability recently
\cite{HeltonMaiSpeicher:ApplicationsOfRealizations}.
The main tool for their study are linearizations.
For technical reasons we restrict our study to regular rational functions,
i.e., rational functions which can be represented as formal power series.
These form a subalgebra of $\IC\llangle\bX\rrangle$ and can be characterized
in several ways \cite{BerstelReutenauer:2011}.
\begin{definition}
  For a series $S  \in \IC\llangle\bX\rrangle$ we
  denote the extraction of coefficients by
  $\langle S,w\rangle$, i.e.,
  $$
  S = \sum_{w\in\bX^*} \langle S,w\rangle\, w
  .
  $$

  A series $S\in \IC\llangle\bX\rrangle$ is called \emph{proper} if it has no
  constant term, i.e., $\langle S,1\rangle=0$.
  Proper series have a \emph{quasi-inverse}
  $$
  S^+ = \sum_{k=1}^\infty S^k \in \IC\llangle \bX\rrangle
  .
  $$
  $Q=S^+$ is the unique solution of the equation
  $$
  Q = S + QS.
  $$
  The algebra of rational series is the smallest subalgebra of
  $\IC\llangle\bX\rrangle$ which contains all  quasi-inverses of its proper
  elements. This implies that for any element $S$ without constant term the
  element $1-S$ is invertible and its inverse is given by the geometric
  series
  $$
  (1-S)^{-1}=  \sum_{k=0}^\infty S^k \in \IC\llangle \bX\rrangle
  $$
  and therefore rational.
  
  A series $S \in \IC\llangle\bX\rrangle$ is called \emph{recognizable}
  if there exists a matrix representation of the free monoid,
  i.e., a multiplicative map $\rho:\bX^*\to  M_n(\IC)$,
  and vectors $u,v\in\IC^n$ 
  such that the
  coefficients of $S$ are given by
  $$
  \langle S,w\rangle = u^t\rho(w) v
  .
  $$
  Here the representation $\rho$ is uniquely determined by the matrices
  $M_x=\rho(x)$ for $x\in\bX$.
  Indeed,   $\rho(x_1x_2\dotsm x_k) = M_{x_1}M_{x_2}\dotsm M_{x_k}$
  and thus a recognizable series has a \emph{linearization}
  $$
  S = u^t
  \Bigl(
  \sum_{w\in \bX^*}  \rho(w) \, w
  \Bigr) v
  = u^t\Bigl(I-\sum_{x\in\bX}M_xx\Bigr)^{-1}v
  .
  $$
  Conversely, any series which is given by a linearization is recognizable.
\end{definition}

\begin{definition}
For a letter $x\in\bX$ denote by $L_x$ the left annihilation operator on $\IC\llangle \bX\rrangle$,
i.e., for a word $w\in \bX^*$ set
$$
L_xw =
\begin{cases}
  w' & \text{if $w=xw'$}\\
  0 & \text{otherwise}
\end{cases}
$$
and extend this operator linearly to $\IC\llangle \bX\rrangle$.
For a word $w=x_1x_2\dotsm x_n\in\bX^*$  we denote by $L_w= L_{x_n}L_{x_{n-1}}\dotsm L_{x_1}$
the composition of the annihilation operators.

A subspace $U\subseteq \IC\llangle \bX\rrangle$ is called \emph{stable} if it is
invariant
under $L_{x}$ for all $x\in\bX$.
\end{definition}

\begin{remark}
With this notation the coefficients of a series are given by
\begin{equation}
  \label{eq:s,w=LwS,1}
\langle S, w\rangle = \langle L_wS,1\rangle
\end{equation}
\end{remark}

One of the main results of the theory of rational series is the following
\begin{theorem}[\cite{BerstelReutenauer:2011}]
  For a series $S\in\IC\llangle \bX\rrangle$ the following are equivalent,
  \begin{enumerate}[(i)]
   \item $S$ is rational.
   \item $S$ is recognizable.
   \item $S$ has a linearization.
   \item There is a finite dimensional stable  subspace of $\IC\llangle
    \bX\rrangle$ containing $S$.
  \end{enumerate}
\end{theorem}
The last statement gives rise to the following algorithm for linearizations of rational series.

\begin{algorithm}
  \label{algo:linearization}
Assume our series $S$ is contained in a finite dimensional stable subspace
  spanned by a basis $(S_1,S_2,\dots,S_N)\subseteq \IC\llangle \bX\rrangle$, say
  $S=\sum u_i S_i$.
\begin{enumerate}[1.]
 \item 
  Since the subspace is stable, for any $x\in\bX$  we can expand
  $$
  L_xS_i = \sum \alpha_{ij}(x) S_j
  .
  $$
 \item 
  Collect these coefficients in a matrix $M_x=
  \left[
    \alpha_{ij}(x)
  \right]$.
 \item 
  Then for a word $w\in\bX^*$ after iteration we obtain
  $$
  L_wS_i = \sum_j (M_w)_{ij}S_j
  $$
  (recall that $L_{w'w''} = L_{w''}L_{w'}$).
 \item 
  From \eqref{eq:s,w=LwS,1} we infer
  $
  \langle S_i,w\rangle =  \langle L_wS,1\rangle = (M_wv)_i
  $ where $v_i=  \langle S_i,1\rangle$.
   \item 
Finally
$
\langle S,w\rangle = u^t M_w v
$
for every $w\in\alg{X}^*$,
i.e.,
$$
S = u^t
\Bigl(
I_N-\sum_{x\in\alg{X}} M_xx
\Bigr)^{-1}v
.
$$
\end{enumerate}
\end{algorithm}

To illustrate this let us start with polynomials which are clearly rational.
Let  $w\in\bX^*$ be a word.
A word $w''\in\bX^*$ is a \emph{right suffix} of $w$ if there exists a word
$w'\in\bX^*$
such that $w=w'w''$. Equivalently, $w''$ can be obtained from $w$ by applying the
left annihilation operator $L_{w'}$.

Let now $P\in\IC\langle \bX\rangle$ be a polynomial.
If we successively apply annihilation operators $L_X$, $X\in\bX$ we obtain
elements contained in the span of all right suffixes of the monomials occurring
in $P$ (including the empty word $1$). Clearly the span of $P$ and all its
right suffixes is stable and thus can be used to obtain a linearization for $P$.

From this basis we can immediately construct a basis of a stable subspace
containing $\Psi = (1-z^mP)^{-1}$ and construct a linearization of the latter,
where $m=\deg(P)$,
i.e., the length of the longest monomial in $P$.
Indeed let $\{S_1=P,S_2,\dots,S_r=1\}$ be a basis consisting of $P$ and
the right suffixes of $P$ spanning a stable invariant subspace for $P$.
Now we use the resolvent identity
$$
\Psi  = (1-z^mP)^{-1} = 1+z^mP\Psi
$$
to obtain
$
L_x\Psi =z^m (L_xP)\Psi
$
and thus $\{\Psi,z^{m_2}S_2\Psi,\dots,z^{m_{r-1}}S_{r-1}\Psi\}$ spans a stable subspace containing
$\Psi$, where $m_i=\deg S_i$.
Moreover $\langle S_i\Psi,1\rangle=0$ for $i=2,\dots,r-1$ and thus we can set
$u=v=e_1$.
The inclusion of the factor $z^{m_i}$ 
ensures that our linearization has the form $\Psi=u^{t}(I-zL)^{-1}v$,
where the entries of $L$ are polynomials in $z$ and thus has
a formal power series expansion around zero, which is essential in Section~\ref{sec:LinearizationAndCondExp}.
\begin{example}
  Consider $P(X,Y)=X+XYX$, then we are looking for a linearization of $\Psi=(1-z^3P)^{-1}$. We have $\Psi=1+z^3P \Psi$ and thus we have
  \begin{align*}
    S_1&=\Psi,
         L_X \Psi=z^3\Psi+z^3 YX \Psi=z^3S_1+z S_2,
  \end{align*}
  where $S_2=z^2YX\Psi$. Of course $L_Y S_1=0$. Next $L_X S_2=0$ and $L_Y S_2=z^2 X \Psi=z S_3$. Finally $L_X S_3=z\Psi=z S_1$ and $L_Y S_3=0$. Hence we obtain $\Psi=u^{t}(I-zL)^{-1} v$ with
  \begin{equation*}
    zL=\begin{bmatrix}
      z^3& z &0\\
      0 & 0 & 0\\
      z & 0 & 0
    \end{bmatrix}
    X
    +
    \begin{bmatrix}
      0& 0 &0\\
      0 & 0 & z\\
      0 & 0 & 0
    \end{bmatrix}
    Y.
\end{equation*}
\end{example}

\begin{remark}\label{remark:ourLinearization}
	Another way to obtain a linearization of $\Psi=(1-z^mP)^{-1}$ for a polynomial P is to directly follow the multiplication structure. Let us describe this with an example, occasionally we will use this linearization instead of the one described above. 
	
	Consider the non-commutative polynomial $XYX+YXY$ and the corresponding polynomial with numbered variables$X_1Y_1X_2+Y_2X_3Y_3$. When we look at powers of this polynomial, we see that:
			
		\begin{itemize}
			\item $X_1$ is always followed by $Y_1$,
			\item $X_2$ is followed by $X_1$ or $Y_2$,
			\item $Y_1$ is followed by $X_1$ or $Y_2$,
			\item $Y_2$ is always followed by $X_2$.
		\end{itemize} 
		This can be represented graphically as an automaton, see Figure~\ref{fig:1}. 
		\begin{center}
			\begin{figure}		
				\begin{tikzpicture}[scale=0.5]
					\begin{scope}[every node/.style={circle,thick,draw}]
						\node (A) at (0,3) {$X_1$};
						\node (B) at (3,3) {$Y_1$};
						\node (C) at (6,3) {$X_2$};
						\node (D) at (0,0) {$Y_2$};
						\node (E) at (3,0) {$X_3$};
						\node (F) at (6,0) {$Y_3$} ;
					\end{scope}	
					\begin{scope}[>={Stealth[black]},
						every node/.style={fill=white,circle},
						every edge/.style={draw=red,very thick}]
						\path [->] (A) edge (B);
						\path [->] (B) edge (C);
						\path [->] (C) edge[bend right=60] (A);
						\path [->] (C) edge (D);
						\path [->] (D) edge (E);
						\path [->] (E) edge (F);
						\path [->] (F) edge (A);
						\path [->] (F) edge[bend left=60] (D);
					\end{scope}
				\end{tikzpicture}
				\caption{Multiplication for $X_1Y_1X_2+Y_2X_3Y_3$.}
				\label{fig:1}
			\end{figure}
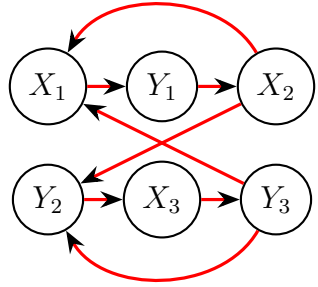
		\end{center}

		Let us name all variables in the considered polynomial $Z_1,Z_2,\ldots,Z_n$, where $n$ is the total number of variables in the polynomial. The polynomial $X_1Y_1X_2+Y_2X_3Y_3$ we have $Z_1 Z_2 Z_3+Z_4 Z_5 Z_6$, we will consider a square matrix in which each row and column corresponds to one of the variables, more precisely we set $L_{i,j}=Z_j$ if $Z_j$ follows $Z_i$, otherwise we set $L_{i,j}=0$.
		We obtain the following linearization matrix (we show the labels of rows and columns)
			\[
		L=\,\begin{blockarray}{ccccccc}
			& X_1 & X_2 & X_3 & Y_1 & Y_2 & Y_3 \\
			\begin{block}{l[cccccc]}
				X_1\, & 0 & 0 & 0 & Y_1 & 0 & 0\\
				X_2\, & X_1 & 0 & 0 & 0 & Y_2 & 0\\
				X_3\, & 0 & 0 & 0 & 0 & 0 & Y_3\\
				Y_1\, & 0 & X_2 & 0 & 0 & 0 & 0\\
				Y_2\, & 0 & 0 & X_3 & 0 & 0 & 0\\
				Y_3\, & X_1 & 0 & 0 & 0 & Y_2 & 0\\
			\end{block}
		\end{blockarray}
		\]
		In order to get a linearization of
                $(1-z^3(X_1Y_1X_2+Y_2X_3Y_3))^{-1}$ we look at $(I-zL)^{-1}$,
                it is clear that every term in expansion of
                $(1-z^3(X_1Y_1X_2+Y_2X_3Y_3))^{-1}$ starts with $X_1$ or $Y_2$,
                hence $u^t=(0,1,0,0,0,0)$ or equivalently
                $u^{t}=(0,0,0,0,0,1)$. On the other hand each term in
                $(1-z^3(X_1Y_1X_2+Y_2X_3Y_3)^{-1}$ ends in $X_2$ or $Y_3$,
                hence we have to terminate with the second or sixth column of $L$, thus $v^t=(0,1,0,0,0,1)$. Clearly we have the equality of formal power series
		\[(1-z^3(X_1Y_1X_2+Y_2X_3Y_3))^{-1}=u^t (I-zL)^{-1} v.\]
		
\end{remark}

\bibliographystyle{abbrv}

\end{document}